\documentclass{article}
\usepackage{amsthm,amsmath,amssymb,color}
\usepackage{lineno}
\usepackage{fullpage}
\usepackage{footnote}
\usepackage{enumitem}
\usepackage{authblk}
\usepackage{tikz,cite}
\usepackage{thm-restate}
\usepackage{tabularray}
\makesavenoteenv{tblr}
\UseTblrLibrary{diagbox}
\usepackage{hyperref}
\hypersetup{
    colorlinks,
    linkcolor={red!60!black},
    citecolor={green!60!black},
    urlcolor={blue!60!black},
    pdftitle={Twin-width of subdivisions of multigraphs},
    pdfauthor={Jungho Ahn, Debsoumya Chakraborti, Kevin Hendrey, and Sang-il Oum}
}

\newtheorem{theorem}{Theorem}[section]
\newtheorem{proposition}[theorem]{Proposition}
\newtheorem{lemma}[theorem]{Lemma}
\newtheorem{claim}{Claim}
\newtheorem{corollary}[theorem]{Corollary}
\newtheorem{observation}{Observation}
\newtheorem{question}{Question}
\newenvironment{subproof}[1][\proofname]{
	
	\begin{proof}[#1]}{\end{proof}}

\newcommand\abs[1]{\lvert #1\rvert}
\newcommand{\DeltaR}{\Delta\!^R}
\newcommand{\tww}{\operatorname{tww}}
\newcommand{\rdeg}{\operatorname{rdeg}}
\newcommand{\dist}{\operatorname{dist}}
\newcommand\sgre[1]{se\-gre\-gated-model}
\newcommand\sena[1]{$\nabla$-sep\-ar\-ation}
\newcommand\subd[1]{$(\ge\!#1)$-sub\-di\-vi\-sion}
\newcommand\cont[1]{$#1$-con\-trac\-tion sequence}

\begin{document}
\date{\today}
\title{Twin-width of subdivisions of multigraphs%
\thanks{All the authors were supported by the Institute for Basic Science (IBS-R029-C1). Jungho Ahn was also supported by the KIAS Individual Grant (CG095301) at Korea Institute for Advanced Study. Debsoumya Chakraborti was also supported by the European Research Council (ERC) under the European Union Horizon 2020 research and innovation programme (grant agreement No.\ 947978).}}
\author[1]{Jungho Ahn}
\author[2]{Debsoumya Chakraborti}
\author[3]{Kevin Hendrey} 
\author[3,4]{Sang-il Oum}
\affil[1]{School of Computational Sciences, Korea Institute for Advanced Study (KIAS), Seoul, South~Korea}
\affil[2]{Mathematics Institute, University of Warwick, Coventry, UK}
\affil[3]{Discrete Mathematics Group, Institute for Basic Science (IBS), Daejeon, South~Korea}
\affil[4]{Department of Mathematical Sciences, KAIST, Daejeon,~South~Korea}
\affil[ ]{\small \textit{Email addresses:} \texttt{junghoahn@kias.re.kr}, \texttt{debsoumya.chakraborti@warwick.ac.uk}, 
\texttt{kevinhendrey@ibs.re.kr}, \texttt{sangil@ibs.re.kr}}

\maketitle

\begin{abstract}
    For each $d\leq3$, we construct a finite set $\mathcal{F}_d$ of multigraphs such that for each graph $H$ of girth at least~$5$ obtained from a multigraph~$G$ by subdividing each edge at least two times, $H$ has twin-width at most~$d$ if and only if $G$ has no minor in~$\mathcal{F}_d$.
    This answers a question of Berg\'{e}, Bonnet, and D\'{e}pr\'{e}s asking for the structure of graphs~$G$ such that each long subdivision of~$G$ has twin-width~$4$.
    As a corollary, we show that the $7\times7$ grid has twin-width~$4$, which answers a question of Schidler and Szeider.
\end{abstract}

\section{Introduction}\label{sec:intro}

In this paper, we use the word \emph{graph} to mean finite simple graph, and we use the word \emph{multigraph} when referring to graphs which may have loops or parallel edges.
Bonnet, Kim, Thomass\'{e}, and Watrigant~\cite{twin-width1} introduced the twin-width of graphs.
The \emph{twin-width} of an $n$-vertex graph~$G$, denoted by $\tww(G)$, is the minimum integer $d$ such that starting from the partition of the vertex set into singletons, by iteratively merging two parts, we can obtain the partition into one part while always maintaining the property that for each part~$P$, there are at most~$d$ other parts~$Q$ such that the graph has both an edge and a non-edge between~$P$ and~$Q$.
We will present a formal definition in Section~\ref{subsec:tww}.
Bonnet~et~al.~\cite{twin-width1} showed that if such a sequence for an $n$-vertex graph is given, then for any first-order formula $\varphi$, one can check whether the graph satisfies~$\varphi$ in $O(n)$ time where the hidden constant depends on~$d$ and the size of~$\varphi$.
Twin-width has received a lot of attention recently~\cite{twin-width2,twin-width3,twin-width4,twin-width6,GPT2021,BKRTW2021,BH2021,DGJOR2021,ST2021,BNOST2021,SS2021,modelcounting,AHKO2021,twwrandom}.
One straightforward but important observation, which will be used several times throughout this paper, is that the twin-width of a graph is at least as large as the twin-width of any of its induced subgraphs.

We analyse the twin-width of subdivisions of multigraphs.
\emph{Subdividing} an edge $e$ of a multigraph is the operation of deleting $e$ and adding a new degree-$2$ vertex whose neighbourhood is the set of ends of $e$.
A \emph{subdivision} of a multigraph~$G$ is a multigraph $H$ obtained from~$G$ by iteratively subdividing edges.
Observe that in~$H$, each edge of $G$ is replaced by a path unless it is a loop, in which case it is replaced by a cycle.
For $k\geq0$, we say that $H$ is an \emph{\subd{k}} of~$G$ if each edge of $G$ is replaced by a path or a cycle of length at least $k+1$ in~$H$,
and is the \emph{$k$-subdivision} of~$G$ if each edge of~$G$ is replaced by a path or a cycle of length exactly $k+1$ in~$H$.
We denote by $K_n$ the complete graph on $n$ vertices and by~$K_{m,n}$ the complete bipartite graph where one part has~$m$ vertices and the other part has~$n$ vertices.

Bonnet, Geniet, Kim, Thomass\'{e}, and Watrigant~\cite{twin-width2} showed that for all positive integers $k$ and $d$, there is a graph whose $k$-subdivision has twin-width at least $d$.

\begin{proposition}[Bonnet, Geniet, Kim, Thomass\'{e}, and Watrigant~{\cite[Proposition~6.2]{twin-width2}}]\label{prop:shortsubd}
    For integers~$d$, $n$, and~$k$ with $d,n\geq2$ and $1\leq k<\log_d(n-1)-1$, the $k$-subdivision of~$K_n$ has twin-width at least $d$.
\end{proposition}

Bonnet et al.~\cite{twin-width2} also showed that there exists some integer $d$ such that every graph has a subdivision of twin-width at most $d$.
Berg\'{e}, Bonnet, and D\'{e}pr\'{e}s~\cite{BBD2021} strengthened this result as follows.

\begin{theorem}[Berg\'{e}, Bonnet, and D\'{e}pr\'{e}s~\cite{BBD2021}]\label{thm:bbd}
    For an $n$-vertex graph $G$, every \subd{2\log_2n-1} of~$G$ has twin-width at most $4$.
\end{theorem}

Berg\'{e} et al.~\cite{BBD2021} suspected that long subdivisions of ``sufficiently complicated'' graphs might have twin-width exactly~$4$ for some unspecified notion of complexity.
We confirm this suspicion and completely characterise the class of graphs whose long subdivisions have twin-width exactly~$d$ for every integer~$d\leq4$.
Surprisingly, our main theorem combined with Theorem~\ref{thm:bbd} implies that the twin-width of the $(\lceil2\log_2\abs{V(G)}\rceil+4)$-subdivision of a multigraph $G$ is a minor-monotone parameter.

A multigraph $H$ is a \emph{minor} of a multigraph~$G$ if there is a collection $(T_u)_{u\in V(H)}$ of vertex-disjoint subtrees~$T_u$ of~$G$ indexed by the vertices $u$ of $H$ and an injection $\varphi:E(H)\to E(G)\setminus\bigcup_{u\in V(H)}E(T_u)$ such that for every $e\in E(H)$, if $\{v,w\}$ is the set of ends of $e$, then $\varphi(e)$ is an edge joining a vertex in $T_v$ with a vertex in~$T_w$.
The \emph{girth} of a multigraph is the minimum length of a cycle.
For every positive integer $\ell$, we denote by $C_\ell$ the $\ell$-edge cycle.
In particular,~$C_1$ is the $1$-vertex multigraph with one loop.
We denote by~$K_6^-$ the graph obtained by removing an edge from $K_6$, by $\overline{C_7}$ the complement of~$C_7$, and by $C_5+\overline{K_2}$ the join of~$C_5$ and the complement of $K_2$.
We denote by $K_{3,\widehat{1},3}$ the graph obtained from the complete tripartite graph $K_{3,1,3}$ by removing one edge from the part of size $1$ to each of the other parts.
We denote by~$Q_3$ the \emph{cube} graph, and by $V_8$ the \emph{Wagner} graph.
We present these six graphs in Figure~\ref{fig:forbidden}.

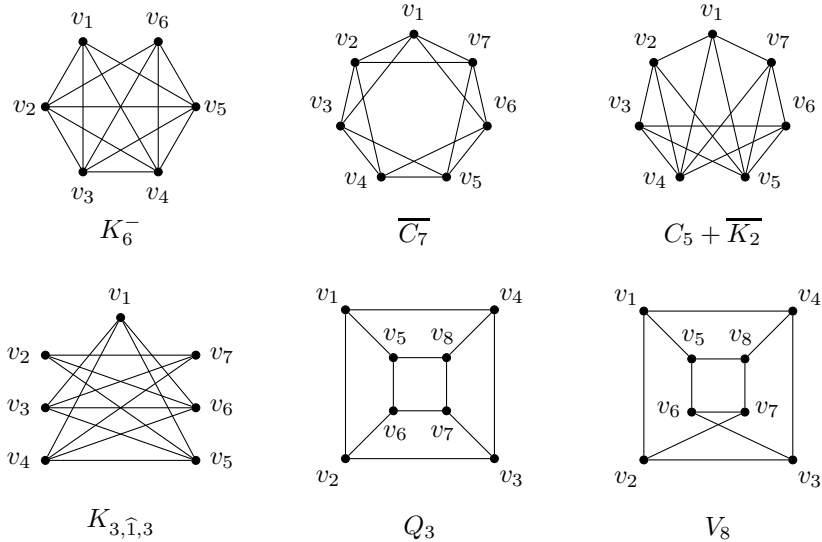
\begin{figure}[t]
    \centering
    \tikzstyle{v}=[circle, draw, solid, fill=black, inner sep=0pt, minimum width=3pt]
    \begin{tikzpicture}
        \draw (0,1.8);
        \draw (0,-1.8);
        \draw (1*60:1) node[v,label={[xshift=0.0mm, yshift=0.0mm]$v_6$}](u1){};
        \draw (2*60:1) node[v,label={[xshift=0.0mm, yshift=0.0mm]$v_1$}](u2){};
        \draw (3*60:1) node[v,label={[xshift=-2.7mm, yshift=-3.0mm]$v_2$}](u3){};
        \draw (4*60:1) node[v,label={[xshift=0.0mm, yshift=-6.0mm]$v_3$}](u4){};
        \draw (5*60:1) node[v,label={[xshift=0.0mm, yshift=-6.0mm]$v_4$}](u5){};
        \draw (6*60:1) node[v,label={[xshift=2.7mm, yshift=-3.0mm]$v_5$}](u6){};
        \foreach \x in {3,4,5,6}{
            \draw (u1)--(u\x)--(u2);
        }
        \draw (u4)--(u3)--(u6)--(u5);
        \draw (u3)--(u5);
        \draw (u4)--(u6);
        \draw (u4)--(u5);
        \draw (0,-1.2) node[label=below:$K_6^-$](){};
    \end{tikzpicture}
    \hspace{0.5cm}
    \begin{tikzpicture}
	\draw (360/7*1+90-360/7:1) node[v,label={[xshift=0.0mm,yshift=0.0mm]$v_1$}](v1){};
	\draw (360/7*2+90-360/7:1) node[v,label={[xshift=-1mm,yshift=0.0mm]$v_2$}](v2){};
	\draw (360/7*3+90-360/7:1) node[v,label={[xshift=-2.5mm,yshift=0.0mm]$v_3$}](v3){};
	\draw (360/7*4+90-360/7:1) node[v,label={[xshift=-3.3mm,yshift=-3.3mm]$v_4$}](v4){};
	\draw (360/7*5+90-360/7:1) node[v,label={[xshift=3.3mm,yshift=-3.3mm]$v_5$}](v5){};
	\draw (360/7*6+90-360/7:1) node[v,label={[xshift=2.5mm,yshift=0.0mm]$v_6$}](v6){};
	\draw (360/7*7+90-360/7:1) node[v,label={[xshift=1mm,yshift=0.0mm]$v_7$}](v7){};
        \draw (-1.5,0) node(){};
        \draw (1.5,0) node(){};
	\draw (v1)--(v2)--(v3)--(v4)--(v5)--(v6)--(v7)--(v1);
	\draw (v1)--(v3)--(v5)--(v7)--(v2)--(v4)--(v6)--(v1);
        \draw (0,-1.2) node[label=below:$\overline{C_7}$](){};
    \end{tikzpicture}
    \hspace{0.5cm}
    \begin{tikzpicture}
	\draw (360/7*1+90-360/7:1) node[v,label={[xshift=0.0mm,yshift=0.0mm]$v_1$}](v1){};
	\draw (360/7*2+90-360/7:1) node[v,label={[xshift=-1mm,yshift=0.0mm]$v_2$}](v2){};
	\draw (360/7*3+90-360/7:1) node[v,label={[xshift=-2.5mm,yshift=0.0mm]$v_3$}](v3){};
	\draw (360/7*4+90-360/7:1) node[v,label={[xshift=-3.3mm,yshift=-3.3mm]$v_4$}](v4){};
	\draw (360/7*5+90-360/7:1) node[v,label={[xshift=3.3mm,yshift=-3.3mm]$v_5$}](v5){};
	\draw (360/7*6+90-360/7:1) node[v,label={[xshift=2.5mm,yshift=0.0mm]$v_6$}](v6){};
	\draw (360/7*7+90-360/7:1) node[v,label={[xshift=1mm,yshift=0.0mm]$v_7$}](v7){};
        \draw (-1.5,0) node(){};
        \draw (1.5,0) node(){};
	\draw (v1)--(v2)--(v3)--(v6)--(v7)--(v1);
        \draw (v4)--(v1);
        \draw (v4)--(v2);
        \draw (v4)--(v3);
        \draw (v4)--(v6);
        \draw (v4)--(v7);
        \draw (v5)--(v1);
        \draw (v5)--(v2);
        \draw (v5)--(v3);
        \draw (v5)--(v6);
        \draw (v5)--(v7);
        \draw (0,-1.2) node[label=below:$C_5+\overline{K_2}$](){};
    \end{tikzpicture}
    \\\vspace{0.3cm}
    \begin{tikzpicture}
        \draw (0,1.2) node[v,label=above:$v_1$](v1){};
        \draw (-1,0.7) node[v,label=left:$v_2$](v2){};
        \draw (-1,0) node[v,label=left:$v_3$](v3){};
        \draw (-1,-0.7) node[v,label=left:$v_4$](v4){};
        \draw (1,-0.7) node[v,label=right:$v_5$](v5){};
        \draw (1,0) node[v,label=right:$v_6$](v6){};
        \draw (1,0.7) node[v,label=right:$v_7$](v7){};
        \draw (-1.5,0) node(){};
        \draw (1.5,0) node(){};
	\draw (v2)--(v5);
        \draw (v2)--(v6);
        \draw (v2)--(v7);
        \draw (v3)--(v5);
        \draw (v3)--(v6);
        \draw (v3)--(v7);
        \draw (v4)--(v5);
        \draw (v4)--(v6);
        \draw (v4)--(v7);
        \draw (v1)--(v3);
        \draw (v1)--(v4);
        \draw (v1)--(v5);
        \draw (v1)--(v6);
        \draw (0,-1.1) node[label=below:$K_{3,\widehat{1},3}$](){};
    \end{tikzpicture}
    \hspace{0.5cm}
    \begin{tikzpicture}
        \draw (1*90+45:1.4) node[v,label={[xshift=-2.3mm,yshift=-1.0mm]$v_1$}](v1){};
        \draw (2*90+45:1.4) node[v,label={[xshift=-2.3mm,yshift=-5.5mm]$v_2$}](v2){};
        \draw (3*90+45:1.4) node[v,label={[xshift=2.3mm,yshift=-5.5mm]$v_3$}](v3){};
        \draw (0*90+45:1.4) node[v,label={[xshift=2.3mm,yshift=-1.0mm]$v_4$}](v0){};
        \draw (1*90+45:0.5) node[v,label={[xshift=0.3mm,yshift=0.0mm]$v_5$}](u1){};
        \draw (2*90+45:0.5) node[v,label={[xshift=0.3mm,yshift=-5.8mm]$v_6$}](u2){};
        \draw (3*90+45:0.5) node[v,label={[xshift=-0.5mm,yshift=-5.8mm]$v_7$}](u3){};
        \draw (0*90+45:0.5) node[v,label={[xshift=-0.5mm,yshift=0.0mm]$v_8$}](u0){};
        \draw (-1.5,0) node(){};
        \draw (1.5,0) node(){};
        \draw (u0)--(u1)--(u2)--(u3)--(u0);
        \draw (v0)--(v1)--(v2)--(v3)--(v0);
        \foreach \x in {0,1,2,3}{
            \draw (u\x)--(v\x);
        }
        \draw (0,-1.5) node[label=below:$Q_3$](){};
    \end{tikzpicture}    
    \hspace{0.5cm}  
    \begin{tikzpicture}
        \draw (1*90+45:1.4) node[v,label={[xshift=-2.3mm,yshift=-1.0mm]$v_1$}](v1){};
        \draw (2*90+45:1.4) node[v,label={[xshift=-2.3mm,yshift=-5.5mm]$v_2$}](v2){};
        \draw (3*90+45:1.4) node[v,label={[xshift=2.3mm,yshift=-5.5mm]$v_3$}](v3){};
        \draw (0*90+45:1.4) node[v,label={[xshift=2.3mm,yshift=-1.0mm]$v_4$}](v0){};
        \draw (1*90+45:0.5) node[v,label={[xshift=0.3mm,yshift=0.0mm]$v_5$}](u1){};
        \draw (2*90+45:0.5) node[v,label={[xshift=-3.0mm,yshift=-2.5mm]$v_6$}](u2){};
        \draw (3*90+45:0.5) node[v,label={[xshift=3.0mm,yshift=-2.5mm]$v_7$}](u3){};
        \draw (0*90+45:0.5) node[v,label={[xshift=-0.5mm,yshift=0.0mm]$v_8$}](u0){};
        \draw (-1.5,0) node(){};
        \draw (1.5,0) node(){};
        \draw (u0)--(u1)--(u2)--(u3)--(u0);
        \draw (v0)--(v1)--(v2)--(v3)--(v0);
        \foreach \x in {0,1}{
            \draw (u\x)--(v\x);
        }
        \draw (u2)--(v3);
        \draw (u3)--(v2);
        \draw (0,-1.5) node[label=below:$V_8$](){};
    \end{tikzpicture}
    \caption{Minor minimal graphs whose \subd{2}s have twin-width at least $4$.}
    \label{fig:forbidden}
\end{figure}

\begin{theorem}\label{thm:main}
    Let $H$ be a graph which is an \subd{2} of a multigraph $G$.
    The following hold.
    \begin{enumerate}[label=(\roman*)]
        \item\label{main:atmost3} $\tww(H)\leq3$ if and only if $G$ has no minor in $\{K_6^-,\overline{C_7},C_5+\overline{K_2},K_{3,\widehat{1},3},Q_3,V_8\}$.
        \item\label{main:atmost2} $\tww(H)\leq2$ if and only if $G$ has no $K_4$ as a minor.
    \end{enumerate}
    In addition, if $H$ is of girth at least~$5$, then the following hold.
    \begin{enumerate}[label=(\roman*)]
    \setcounter{enumi}{2}
        \item\label{main:atmost1} $\tww(H)\leq1$ if and only if $G$ has no minor in $\{K_{1,3},C_1\}$.
        \item\label{main:atmost0} $\tww(H)=0$ if and only if $G$ has no minor in $\{K_2,C_1\}$.
    \end{enumerate}
\end{theorem}

The extra condition for Theorem~\ref{thm:main}\ref{main:atmost1} and \ref{main:atmost0} is necessary since the twin-width of an \subd{2} of the $1$-vertex multigraph with one loop is $0$ if it has girth at most~$4$.

Let $d$ and $t$ be positive integers, and let $\mathcal{S}_{d,t}$ be the set of multigraphs whose $t$-subdivisions have twin-width at most~$d$.
By Theorem~\ref{thm:main}, if $d\leq3$ and $t\geq4$, then $\mathcal{S}_{d,t}$ is closed under taking minors.
However, if $d\geq4$, then $\mathcal{S}_{d,t}$ is not closed under taking minors.
To see this, let $n>(d+1)^{t+1}+1$.
By Proposition~\ref{prop:shortsubd}, the $t$-subdivision of~$K_n$ has twin-width at least~$d+1$, so $K_n$ is not in ${S}_{d,t}$.
Let $H$ be an \subd{\log_2n} of~$K_n$.
Since $t$ is positive, the $t$-subdivision of $H$ is an \subd{2\log_2n} of $K_n$, and so $H$ is in $\mathcal{S}_{d,t}$ 
by Theorem~\ref{thm:bbd}.
This demonstrates that $\mathcal{S}_{d,t}$ is not closed under taking minors since $K_n$ is a minor of $H$.

Using Theorem~\ref{thm:bbd}, Berg\'{e}, Bonnet, and D\'{e}pr\'{e}s~\cite{BBD2021} were able to show that deciding twin-width at most~$4$ is NP-complete.
In contrast, an algorithm of Corneil, Perl, and Stewart~\cite{Corneil1985} can be used to find a \cont{0} of a graph in linear time if one exists, and Bonnet, Kim, Reinald, Thomass\'{e}, and Watrigant~\cite{BKRTW2021} presented a polynomial-time algorithm that finds a \cont{1} if one exists.
For each $d\in\{2,3\}$, it is not known whether one can decide in polynomial time that a graph has twin-width at most $d$.
However, for graphs of twin-width at most~$3$ and girth at least~$5$ which are \subd{2}s of multigraphs, we present explicit contraction sequences witnessing the twin-width in the proofs leading to Theorem~\ref{thm:main}.
A polynomial-time algorithm for finding these contraction sequences can easily be extracted from the proofs.

For the forward direction of Theorem~\ref{thm:main}\ref{main:atmost3}, we actually prove a stronger statement which has the following additional consequence.
Let $\mathcal{F}_3:=\{K_6^-,\overline{C_7},C_5+\overline{K_2},K_{3,\widehat{1},3},Q_3,V_8\}$

\begin{proposition}\label{prop:mainforward}
    If $G$ is a multigraph which has a minor in $\mathcal{F}_3$, then all \subd{1}s of~$G$ and the line graphs of all \subd{2}s of~$G$ have twin-width at least $4$.
\end{proposition}

We determine the maximum number of edges of a simple $n$-vertex graph having some   \subd{1} of twin-width at most $3$.

\begin{restatable}{theorem}{edgenumber}\label{thm:final}
    For every integer $n\geq3$, if an $n$-vertex graph~$G$ has more than $\lfloor(7n-15)/2\rfloor$ edges, then every \subd{1} of~$G$ has twin-width at least~$4$.
    Furthermore, for every integer $n\geq3$, there exists an $n$-vertex graph~$G$ with $\lfloor(7n-15)/2\rfloor$ edges such that every \subd{1} of $G$ has twin-width at most~$3$.
\end{restatable}

\begin{figure}[t]
    \centering
    \tikzstyle{b}=[circle, draw=black!70, solid, fill=black!70, inner sep=0pt, minimum width=5.5pt]
    \tikzstyle{v}=[circle, draw, solid, fill=black, inner sep=0pt, minimum width=1.6pt]
    \begin{tikzpicture}[scale=0.5]
        \draw(0,7.1) node(){};
        \draw(0,0.9) node(){};
        \draw[line width=3.5pt,black!40] (1,1)--(1,7)--(7,7)--(7,1)--(1,1)--(1,7);
        \draw[line width=3.5pt,black!40] (1,5)--(5,5);
        \draw[line width=3.5pt,black!40] (5,7)--(5,3);
        \draw[line width=3.5pt,black!40] (3,3)--(7,3);
        \draw[line width=3.5pt,black!40] (3,5)--(3,1);
        \foreach \x in {1,2,3,4,5,6,7}{
            \draw (\x,1) node[v](){};
            \draw (\x,2) node[v](){};
            \draw (\x,3) node[v](){};
            \draw (\x,4) node[v](){};
            \draw (\x,5) node[v](){};
            \draw (\x,6) node[v](){};
            \draw (\x,7) node[v](){};
            \draw (\x,1)--(\x,7);
            \draw (1,\x)--(7,\x);
        }
        \draw (1,5) node[b](){};
        \draw (3,5) node[b](){};
        \draw (3,3) node[b](){};
        \draw (3,1) node[b](){};
        \draw (5,3) node[b](){};
        \draw (5,5) node[b](){};
        \draw (5,7) node[b](){};
        \draw (7,3) node[b](){};
    \end{tikzpicture}
    \hspace{0.3cm}
    \begin{tikzpicture}[scale=0.5]
        \draw(0,7.1) node(){};
        \draw(0,0.9) node(){};
        \draw [line width=3.5pt,black!40] (4,2)--(4,1)--(8,1)--(8,2)--(11,2)--(11,3)--(10,3)
        --(10,4)--(11,4)--(11,5)--(10,5)--(10,6)--(9,6)--(9,7)--(3,7)--(3,6)--(2,6)--(2,5)
        --(1,5)--(1,4)--(2,4)--(2,3)--(1,3)--(1,2)--(4,2);
        \draw [line width=3.5pt,black!40] (2,5)--(10,5);
        \draw [line width=3.5pt,black!40] (4,2)--(5,2)--(5,3)--(4,3)--(4,4)--(5,4)--(5,5);
        \draw [line width=3.5pt,black!40] (8,2)--(7,2)--(7,3)--(8,3)--(8,4)--(7,4)--(7,5);
        \draw [line width=3.5pt,black!40] (5,3)--(7,3);
        \foreach \x in {1,2,3,4,5,6,7,8,9,10,11}{
            \draw (\x,7) node[v](7h\x){};
            \draw (\x,6) node[v](6h\x){};
            \draw (\x,5) node[v](5h\x){};
            \draw (\x,4) node[v](4h\x){};
            \draw (\x,3) node[v](3h\x){};
            \draw (\x,2) node[v](2h\x){};
        }
        \foreach \x in {2,3,4,5,6,7,8,9,10}{
            \draw (\x,1) node[v](1h\x){};
        }
        \foreach \x in {1,3,5,7,9,11}{
            \draw (7h\x)--(6h\x);
            \draw (5h\x)--(4h\x);
            \draw (3h\x)--(2h\x);
        }
        \foreach \x in {2,4,6,8,10}{
            \draw (6h\x)--(5h\x);
            \draw (4h\x)--(3h\x);
            \draw (2h\x)--(1h\x);
        }
        \foreach \x in {2,3,4,5,6,7}{
            \draw (1,\x)--(11,\x);
        }
        \draw (2,1)--(10,1);
        \draw (4,2) node[b](a1){};
        \draw (5,3) node[b](b1){};
        \draw (7,3) node[b](b2){};
        \draw (7,5) node[b](b3){};
        \draw (5,5) node[b](b4){};
        \draw (8,2) node[b](a4){};
        \draw (2,5) node[b](a2){};
        \draw (10,5) node[b](a3){};
    \end{tikzpicture}
    \hspace{0.3cm}
    \begin{tikzpicture}[scale=0.5]
        \draw(0,6.6) node(){};
        \draw(0,0.4) node(){};
        \draw [line width=3.5pt,black!40] (2,1)--(10,1)--(10,2)--(11,2)--(11,3)--(10,3)--(10,4)
        --(11,4)--(11,5)--(10,5)--(10,6)--(2,6)--(2,5)--(1,5)--(1,4)--(2,4)--(2,3)
        --(1,3)--(1,2)--(2,2)--(2,1)--(3,1);
        \draw [line width=3.5pt,black!40] (2,2)--(7,2)--(7,3)--(9,3)--(9,2)--(8,2)--(8,1);
        \draw [line width=3.5pt,black!40] (10,5)--(5,5)--(5,4)--(3,4)--(3,5)--(4,5)--(4,6);
        \draw [line width=3.5pt,black!40] (5,2)--(5,3)--(4,3)--(4,4);
        \draw [line width=3.5pt,black!40] (7,5)--(7,4)--(8,4)--(8,3);
        \foreach \x in {1,2,3,4,5,6,7,8,9,10,11}{
            \draw (\x,5) node[v](5h\x){};
            \draw (\x,4) node[v](4h\x){};
            \draw (\x,3) node[v](3h\x){};
            \draw (\x,2) node[v](2h\x){};
        }
        \foreach \x in {2,3,4,5,6,7,8,9,10}{
            \draw (\x,1) node[v](1h\x){};
            \draw (\x,6) node[v](6h\x){};
        }
        \foreach \x in {1,3,5,7,9,11}{
            \draw (5h\x)--(4h\x);
            \draw (3h\x)--(2h\x);
        }
        \foreach \x in {2,4,6,8,10}{
            \draw (6h\x)--(5h\x);
            \draw (4h\x)--(3h\x);
            \draw (2h\x)--(1h\x);
        }
        \foreach \x in {2,3,4,5}{
            \draw (1,\x)--(11,\x);
        }
        \draw (2,1)--(10,1);
        \draw(2,6)--(10,6);
        \draw (2,2) node[b](){};
        \draw (5,2) node[b](){};
        \draw (7,5) node[b](){};
        \draw (4,4) node[b](){};
        \draw (4,6) node[b](){};
        \draw (8,3) node[b](){};
        \draw (8,1) node[b](){};
        \draw (10,5) node[b](){};
    \end{tikzpicture}
    \caption{An induced subgraph of the $7\times 7$ grid isomorphic to an \subd{1} of $Q_3$, an induced subgraph of the elementary $11\times 7$ wall isomorphic to an \subd{1} of $Q_3$, and a subgraph of the elementary $11\times 6$ wall isomorphic to an \subd{2} of $Q_3$.}
    \label{fig:cubic}
\end{figure}
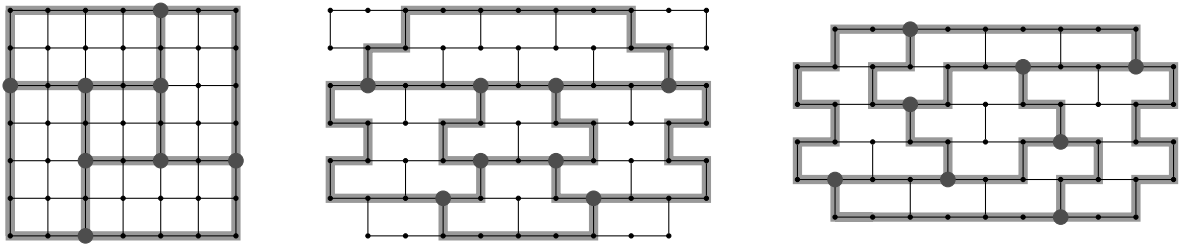

Proposition~\ref{prop:mainforward} is an important step towards understanding the structure of sparse graphs of twin-width at most $3$.
As an example, it allows us to determine the twin-width of large grids, large walls, and the line graphs of large walls by finding appropriate induced subgraphs.
Let $m$ and $n$ be positive integers.
The \emph{$m\times n$ grid} is the graph with vertex set $\{(i,j)\in\mathbb{Z}^2:1\leq i\leq m,~1\leq j\leq n\}$ such that $(i,j)$ and $(i',j')$ are adjacent if and only if $\abs{i-i'}+\abs{j-j'}=1$.
The \emph{elementary $m\times n$ wall} is the graph obtained from the $m\times n$ grid by removing all edges between $(i,j)$ and $(i,j+1)$ where~$i$ and~$j$ have the same parity, and then removing vertices of degree~$1$.
An \emph{$m\times n$ wall} is a graph isomorphic to a subdivision of the elementary $m\times n$ wall.

Utilising SAT solvers, Schidler and Szeider~\cite{SS2021} were able to compute the twin-width of various small graphs, including the $6\times 8$ grid which has twin-width~$3$.
Leading from this work, they raised the problem of determining the twin-width of the $m\times m$ grid for all positive integers $m$.
We completely solve the problem by showing that the $m\times m$ grid has twin-width~$4$ for all $m\geq 7$.
The upper bound was already known, and is straightforward to observe; see~\cite{gridtwinwidth}.
The lower bound follows from Proposition~\ref{prop:mainforward}, together with the observation that the $7\times7$ grid has an induced subgraph isomorphic to an \subd{1} of $Q_3$, as illustrated in Figure~\ref{fig:cubic}.
We summarise the state of knowledge of the twin-width of general $m\times n$ grids in Table~\ref{tab:grid}.
The case which remains open is determining the twin-width of the $6\times n$ grid for $n\geq 9$.

Similar to grids, large walls and their line graphs also have twin-width $4$.
We will prove the upper bounds in Section~\ref{sec:grid}.
The lower bounds come from Proposition~\ref{prop:mainforward} and the following observations.
We observe that the elementary $11\times7$ wall contains an induced subgraph isomorphic to an \subd{1} of~$Q_3$, and the elementary $11\times6$ wall contains a subgraph isomorphic to an \subd{2} of~$Q_3$; see Figure~\ref{fig:cubic}.
It follows that every $11\times7$ wall contains an induced subgraph isomorphic to an \subd{1} of~$Q_3$, and the line graph of every $11\times6$ wall contains an induced subgraph isomorphic to an \subd{2} of~$Q_3$.
Thus, we obtain the following corollary of Proposition~\ref{prop:mainforward}.

\begin{table}[t]
    \centering
    \newcommand\rem[1]{\text{ ({#1})}}
    \begin{tblr}{
        colspec={cccccccccc},
        column{2-11}={mode=math},
        cell{2-9}{1}={mode=math},
        vline{1,2}={1pt,solid},
        hline{1,2}={1pt,solid},
        hline{3}={2,4-11}{solid},
        hline{4}={3-11}{solid},
        hline{5}={4-5}{solid},
        hline{6}={5}{solid},
        hline{7}={6-11}{solid},
        hline{8}={7-11}{solid},
        vline{3}={3}{solid},
        vline{4}={3-4}{solid},
        vline{5}={5}{solid},
        vline{6}={6}{solid},
        vline{7}={7}{solid},
        vline{10}={7}{solid},
        vline{8}={8-9}{solid},
        vline{5}={2-3}{solid},
        vline{6}={4}{solid},
        cell{2}{5}={c=7}{l},
        cell{3,5}{5}={c=7}{l,gray!20},
        cell{4}{6}={c=6,r=3}{l,gray!20},
        cell{7}{7}={c=3}{l},
        cell{7}{10}={c=2}{l},
        cell{8}{8}={c=4}{l,gray!20}
        }
        \diagbox{$m$}{$n$} & 1& 2& 3 & 4 & 5 & 6 & 7 & 8 & \geq9 &  \\
        1 & 0 & & & 1 \rem{Bonnet et al.~\cite{twin-width1}} \\
        2 &   & &1& {\mathbf{2}} \rem{Proposition~\ref{prop:grid1}}\\  
        3 &   & &2& &  \\
        4 &   & & & {\mathbf{3}} \rem{Proposition~\ref{prop:grid2} and Schidler\textsuperscript{\textcolor{red!60!black}{1}}} \\
        5 &   & & &  & \\
        6 &   & & &  &  & 3\rem{Schidler and Szeider~\cite{SS2021}} & & &\le 4 \rem{Bonnet~\cite{gridtwinwidth}}\\
        \geq7 &   & & &  &  &   & {\mathbf{4}} \rem{Corollary~\ref{cor:grid}}  
    \end{tblr}
    \caption{Twin-width of $m\times n$ grids. New results are in gray-coloured cells.}
    \label{tab:grid}
\end{table}

\footnotetext[1]{Schidler informed the authors in 2023 that he computed the twin-width of the $3\times5$ grid and the $4\times4$ grid using the SAT solver by the method of Schidler and Szeider~\cite{SS2021}.}

\begin{corollary}\label{cor:grid}
    Let $m$ and $n$ be integers.
    Then the following hold.
    \begin{itemize}
        \item If $m\geq7$ and $n\geq7$, then the $m\times n$ grid has twin-width~$4$.
        \item If $m\geq11$ and $n\geq7$, then every $m\times n$ wall has twin-width~$4$.
        \item If $m\geq11$ and $n\geq6$, then the line graph of every $m\times n$ wall has twin-width~$4$.\qed
    \end{itemize}
\end{corollary}

The reason for focusing on walls and their line graphs is due to their relationship with tree-width.
The tree-width of a graph is an important structural parameter which has many algorithmic applications.
For completeness, we present a formal definition in Section~\ref{sec:prelim}.
Courcelle's theorem~\cite{Courcelle1990} famously shows that every monadic second-order formula can be decided in linear time on graphs of bounded tree-width.
In terms of the relationship with twin-width, Bonnet et al.~\cite{twin-width1} showed that every class of graphs of bounded tree-width has bounded twin-width, and Jacob and Pilipczuk~\cite{JP2022} improved the bound.
The converse does not hold, as large complete graphs have large tree-width while having twin-width~$0$.
Thus, a bound on the twin-width of a graph is not generally sufficient for applying Courcelle's theorem.
However the following recent breakthrough results of Aboulker, Adler, Kim, Sintiari, and Trotignon~\cite{evenholefree} and of Korhonen~\cite{gridinducedminor} enable us to bound the tree-width of a graph~$G$ of twin-width at most $3$, given either a bound on the maximum degree of~$G$ or a fixed graph~$H$ which~$G$ does not contain as a minor.

\begin{theorem}[Aboulker, Adler, Kim, Sintiari, and Trotignon~\cite{evenholefree}]\label{thm:evenholefree}
    For every graph  $H$, there is a function $f_H:\mathbb{N}\to\mathbb{N}$ such that every $H$-minor-free graph of tree-width at least $f_H(k)$ contains a $k\times k$-wall or the line graph of a $k\times k$-wall as an induced subgraph.
\end{theorem}

\begin{theorem}[Korhonen~\cite{gridinducedminor}]\label{thm:gridinducedminor}
    For every integer $d$, there is a function $f_d:\mathbb{N}\to\mathbb{N}$ such that every graph with maximum degree at most $d$ and tree-width at least $f_d(k)$ contains a $k\times k$-wall or the line graph of a $k\times k$ wall as an induced subgraph.
\end{theorem}

From Theorems~\ref{thm:evenholefree} and~\ref{thm:gridinducedminor}, and Corollary~\ref{cor:grid}, we deduce the following corollaries.

\begin{corollary}
    For every graph $H$, there exists an integer $w$ such that every $H$-minor-free graph of twin-width at most~$3$ has tree-width at most~$w$.
\end{corollary}

\begin{corollary}
    For every integer $\Delta$, there exists an integer $w$ such that graph of twin-width at most $3$ and maximum degree at most $\Delta$ has tree-width at most $w$.
\end{corollary}

We organise this paper as follows.
In Section~\ref{sec:prelim}, we present some terminology from graph theory and formally define twin-width.
In Section~\ref{sec:structures}, we investigate the structures of the graphs in $\mathcal{F}_3$.
In Sections~\ref{sec:forward},~\ref{sec:backward}, and~\ref{sec:2tree}, we show Theorem~\ref{thm:main} and Proposition~\ref{prop:mainforward}.
In Section~\ref{sec:grid}, we investigate the twin-width of $m\times n$ grids for small $m$, walls, and the line graphs of walls.
In Section~\ref{sec:extremal}, we determine the maximum number of edges of graphs having no \subd{1}s of twin-width at least $4$, and in Section~\ref{sec:conclu}, we present some open problems.

\section{Preliminaries}\label{sec:prelim}

For every integer $k$, let $[k]$ be the set of positive integers less than or equal to~$k$.
For a set~$S$ and an integer~$k$, we denote by ${S\choose k}$ the set of $k$-element subsets of~$S$.
Let $G$ be a multigraph.
For a set $X\subseteq V(G)$, let $N_G[X]:=\bigcup_{v\in X}N_G[v]$, called the \emph{closed neighbourhood} of~$X$, and $N_G(X):=N_G[X]\setminus X$.
The \emph{degree} of a vertex~$v$ of~$G$, denoted by $\deg_G(v)$, is the number of edges incident with~$v$, where each loop contributes~$2$.
The \emph{minimum degree} of~$G$, denoted by $\delta(G)$, is the minimum of the degree of a vertex, and the \emph{maximum degree} of~$G$, denoted by $\Delta(G)$, is the maximum of the degree of a vertex.
The \emph{distance} between~$v$ and~$w$ in~$G$, denoted by $\dist_G(v,w)$, is the minimum length of a path in~$G$ between them.
The \emph{second neighbourhood} of~$X$ is $N_G^2[X]:=N_G[N_G[X]]$.
If $X=\{v\}$, then we may write $N_G^2[v]$ for $N_G^2[\{v\}]$.
We may omit the subscripts in these notations if it is clear from the context.
The \emph{simplification} of a multigraph~$G$ is the  simple graph on~$V(G)$ such that two distinct vertices are adjacent if and only if they are adjacent in~$G$.

For a set $X$ of vertices, we denote by $G-X$ the multigraph obtained from $G$ by removing the vertices in~$X$ and the edges incident with vertices in $X$.
We sometimes abuse the notation to allow deleting vertices not in $V(G)$, which means that $G-X=G-(X\cap V(G))$.
If $X=\{v\}$, then we may write $G-v$ for $G-\{v\}$.
The \emph{induced subgraph} of $G$ on $X$ is the multigraph $G[X]:=G-(V(G)\setminus X)$.
For a set $Y$ of edges, we denote by $G-Y$ the multigraph obtained from $G$ by removing the edges of $Y$.

A \emph{clique} in~$G$ is a set of pairwise adjacent vertices of $G$, and an \emph{independent set} in~$G$ is a set of pairwise nonadjacent vertices of $G$.
A \emph{triangle} is a cycle of length $3$.
In a graph, we may write $uvw$ for a triangle on the vertex set $\{u,v,w\}$.
A \emph{star} is a complete bipartite graph $K_{1,s}$ for $s\geq1$.
The \emph{line graph} of a multi\-graph~$G$ is the graph~$L(G)$ with vertex set $E(G)$ such that distinct $e,e'\in E(G)$ are adjacent in $L(G)$ if and only if~$e$ and~$e'$ share an end in~$G$.
A \emph{matching} of~$G$ is an independent set of~$L(G)$.
For multigraphs $G$ and $G'$, the \emph{union} of~$G$ and~$G'$ is the multigraph $(V(G)\cup V(G'),E(G)\cup E(G'))$.
If $V(G)\cap V(G')=\emptyset$, then we call it the \emph{disjoint union} of~$G$ and~$G'$.
The \emph{join} of~$G$ and~$G'$ is the multigraph $G+G'$ obtained from the disjoint union of~$G$ and~$G'$ by adding one new edge between~$v$ and~$w$ for each pair $\{v,w\}$ of $v\in V(G)$ and $w\in V(G')$.
For a graph $F$, the \emph{complement} of~$F$ is the graph $\overline{F}$ with vertex set $V(F)$ such that two distinct vertices~$u$ and~$v$ are adjacent in $\overline{F}$ if and only if they are nonadjacent in~$F$.
An \emph{isomorphism} from a multigraph~$G$ to a multigraph~$G'$ is a pair $(\varphi_v,\varphi_e)$ of 
bijections $\varphi_v:V(G)\to V(G')$ and $\varphi_e:E(G)\to E(G')$ such that 
for every edge $f$ of $G$ with its set of ends $\{u,w\}$, 
the set of ends of $\varphi_e(f)$ is equal to $\{ \varphi_v(u),\varphi_v(w)\}$.
If $G$ is a simple graph and $u$ and $v$ are adjacent vertices, then we write $uv$ to denote the unique edge incident with both~$u$ and~$v$.

A multigraph $G$ is \emph{connected} if for every pair $\{v,w\}$ of distinct vertices, $G$ has a path between $v$ and~$w$.
A \emph{separator} of $G$ is a set $S\subseteq V(G)$ such that $G-S$ is disconnected.
A separator $S$ of $G$ is \emph{minimal} if no proper subset of $S$ is a separator of~$G$.
For an integer $k\geq1$, $G$ is \emph{$k$-connected} if $\abs{V(G)}>k$ and~$G$ has no separator of size less than $k$.
A \emph{separation} of $G$ is an ordered pair $(A,B)$ of subsets of $V(G)$ such that $A\cup B=V(G)$ and $G$ has no edge between $A\setminus B$ and $B\setminus A$.
For a minor~$H$ of~$G$, a \emph{minor model of~$H$ in~$G$} is a collection $(T_u)_{u\in V(H)}$
of vertex-disjoint subtrees $T_u$ of~$G$ indexed by the vertices~$u$ of~$H$ for which there is an injection $\varphi:E(H)\to E(G)\setminus\bigcup_{u\in V(H)}E(T_u)$ 
such that if~$u$ and~$v$ are the ends of~$e$ in~$H$, then a vertex in $T_u$ and a vertex in $T_v$ are the ends of $\varphi(e)$ in~$G$.
For a set $\mathcal{F}$ of multigraphs, $G$ is \emph{$\mathcal{F}$-minor-free} if no multigraph in $\mathcal{F}$ is a minor of $G$.
If $\mathcal{F}=\{F\}$, we may say that $G$ is \emph{$F$-minor-free}.
A subgraph of $G$ is \emph{spanning} if its vertex set is equal to~$V(G)$.

A \emph{tree decomposition} of a multigraph $G$ is a pair $(T,(B_t)_{t\in V(T)})$ 
of a tree~$T$ and a collection of subsets~$B_t$ of~$V(G)$ indexed by the vertices~$t$ of~$T$
such that
\begin{enumerate}[label=$(\mathrm{T}\arabic*)$]
    \item\label{def:td bag union} $V(G)=\bigcup_{t\in V(T)}B_t$,
    \item\label{def:td edge} for each $e\in E(G)$, there is a node $t\in V(T)$ such that both ends of $e$ are in $B_t$,
    \item\label{def:td vertex} for each $v\in V(G)$, 
    the set $\{t\in V(T):v\in B_t\}$ induces a connected subgraph of~$T$.
\end{enumerate}
The sets in $(B_t)_{t\in V(T)}$ are \emph{bags}.
For $vw\in E(T)$, the \emph{adhesion set} of $B_v$ and $B_w$ is $B_v\cap B_w$.

The \emph{tree-width} of a multigraph is the minimum $k$ such that it admits a tree decomposition such that every bag has at most $k+1$ vertices.
A \emph{$k$-tree} is either $K_{k+1}$ or a graph obtained from a $k$-tree by adding a new vertex and making it adjacent to every vertex in some clique of size~$k$.
A \emph{partial $k$-tree} is a subgraph of a $k$-tree.
It is well known that a graph has tree-width at most $k$ if and only if it is a partial $k$-tree, see Bodlaender~\cite{Bodlaender1998}.

\begin{lemma}[see~{\cite[Lemma~12.3.1]{Diestel2018}}]\label{lem:adhesion}
    Let $(T,(B_t)_{t\in V(T)})$ be a tree decomposition of a multigraph~$G$.
    For $vw\in E(T)$, let $T_1$ and $T_2$ be the components of $T-vw$ with $v\in V(T_1)$.
    For each $i\in[2]$, let $U_i:=\bigcup_{t\in V(T_i)}B_t$.
    If $U_i\setminus(B_v\cap B_w)\neq\emptyset$ for each $i\in[2]$, then $B_v\cap B_w$ is a separator of $G$.
\end{lemma}

\begin{lemma}[see~{\cite[Corollary~12.3.5]{Diestel2018}}]\label{lem:complete bag}
    For every tree decomposition of a multigraph~$G$, every clique in~$G$ is contained in some bag of the tree decomposition.
\end{lemma}

\subsection{Trigraphs and twin-width}\label{subsec:tww}

A \emph{trigraph} is a triple $H=(V(H),B(H),R(H))$ such that $(V(H),B(H)\cup R(H))$ is a graph and~$B(H)$ and~$R(H)$ are disjoint.
We define $E(H):=B(H)\cup R(H)$.
The elements of $E(H)$ are called the \emph{edges} of~$H$, the elements of~$B(H)$ are called the \emph{black edges} of~$H$, and the elements of~$R(H)$ are called the \emph{red edges} of~$H$.
We identify a graph $G=(V,E)$ with a trigraph $(V,E,\emptyset)$.

Let $H$ be a trigraph.
The \emph{underlying graph} of~$H$ is the graph $(V(H),E(H))$.
We identify $H$ with its underlying graph when we use standard graph-theoretic terms and notations such as adjacency, neighbours, $N_H(v)$, $N_H[X]$, $\deg_H(v)$, $\delta(H)$, $\Delta(H)$, $\dist_H(u,v)$, the second neighbourhood, connectedness, $k$-connectedness, a separator, and a separation.
We call a neighbour $w$ of $v$ a \emph{black neighbour} if $vw\in B(H)$ and a \emph{red neighbour} if $vw\in R(H)$.
The \emph{red degree} of~$v$, denoted by $\rdeg_H(v)$, is the number of red neighbours of $v$.
The \emph{maximum red degree} of $H$, denoted by $\DeltaR(H)$, the maximum of the red degree of a vertex.
For an integer $d$, a \emph{$d$-trigraph} is a trigraph with the maximum red degree at most~$d$.
Two distinct vertices~$v$ and~$w$ of $H$ are \emph{twins} in~$H$ if for every $u\in V(G)\setminus\{v,w\}$, $uv\in B(H)$ implies $uw\in B(H)$, $uv\in R(H)$ implies $uw\in R(H)$, and $uv\notin E(H)$ implies $uw\notin E(H)$.
We may omit the subscripts in these notations if it is clear from the context.

A trigraph $H'$ is a \emph{subgraph} of $H$ if $V(H')\subseteq V(H)$, $B(H')\subseteq B(H)$, and $R(H')\subseteq R(H)$.
For a set $X\subseteq V(G)$, we denote by $H-X$ the trigraph obtained from~$H$ by removing the vertices in~$X$ and the edges incident with vertices in~$X$.
If $X=\{v\}$, then we may write $H-v$ for~$H-\{v\}$.
The \emph{induced subgraph} of~$H$ on~$X$ is the trigraph $H[X]:=H-(V(H)\setminus X)$.
For a set $Y\subseteq E(H)$, we denote by~$H-Y$ the trigraph obtained from~$H$ by removing the edges in~$Y$.

An \emph{isomorphism} from~$H$ to~$H'$ is a bijection $\varphi:V(H)\to V(H')$ such that for every pair of distinct vertices of $H$,
\begin{itemize}
    \item if $v$ and $w$ are joined by a black edge in~$H$, then $\varphi(v)$ and $\varphi(w)$ are joined by a black edge in $H'$,
    \item if $v$ and $w$ are joined by a red edge in~$H$, then $\varphi(v)$ and $\varphi(w)$ are joined by a red edge in $H'$,
and
    \item if $v$ and $w$ are non-adjacent in~$H$, then $\varphi(v)$ and $\varphi(w)$ are non-adjacent in $H'$.
\end{itemize}
An \emph{automorphism} of $H$ is an isomorphism from~$H$ to itself.
Two sets $\mathcal{X},\mathcal{X}'$ of subsets of $V(H)$ are \emph{isomorphic} if $H$ has an automorphism $\rho$ such that $\mathcal{X}'=\{\rho(X):X\in\mathcal{X}\}$.

For distinct vertices $v$ and $w$ of $H$, we denote by $H/\{u,v\}$ the trigraph $H'$ obtained from $H-\{u,v\}$ by adding a vertex $x$ such that for every vertex $y\in V(H)\setminus\{u,v\}$, the following hold:
\begin{itemize}
    \item if both $u$ and $v$ are joined to $y$ by black edges in~$H$, then $x$ and $y$ are joined by a black edge in $H'$, 
    \item if both $u$ and $v$ are non-adjacent to $y$ in~$H$, then $x$ and $y$ are non-adjacent in $H'$, and
    \item otherwise, $x$ and $y$ are joined by a red edge in $H'$.
\end{itemize}
We call this operation \emph{contracting $u$ and $v$}, and we say that $u$ and $v$ are \emph{involved} in this contraction, and that $x$ is the vertex \emph{resulting} from this contraction.
We allow $x$ to be equal to one of the vertices $u$ or $v$, and if $x=u$, then we call this operation \emph{contracting $\{u,v\}$ to~$u$}.

A \emph{partial contraction sequence} from~$H$ to a trigraph~$H'$ is a sequence $H_1,\ldots,H_t$ of trigraphs such that $H=H_1$, $H'=H_t$, and for each $i\in[t-1]$, $H_{i+1}$ is obtained from~$H_i$ by contracting vertices.
If $\abs{V(H')}=1$, then we call it a \emph{contraction sequence} of $H$.
For an integer $d$, a (partial) contraction sequence $H_1,\ldots,H_t$ is a \emph{(partial) \cont{d}} if $\DeltaR(H_i)\leq d$ for every $i\in[t]$.
The \emph{twin-width} of $H$, denoted by $\tww(H)$, is the minimum integer~$d$ such that there is a \cont{d} of~$H$.
If~$H'$ is an induced subgraph of~$H$, then $\tww(H')\leq\tww(H)$, because every \cont{\tww(H)} of $H$ yields a \cont{\tww(H)} of $H'$ by ignoring contractions involving vertices not in $V(H')$.

\begin{lemma}\label{lem:cycle}
    The twin-width of a cycle is~$2$ if its length is at least~$5$, and otherwise~$0$.
\end{lemma}
\begin{proof}
    Cographs are known to have twin-width~$0$~\cite{twin-width1}, and cycles of length at most~$4$ are cographs.
    Thus, a cycle has twin-width~$0$ if its length is at most~$4$.
    Let $C$ be a cycle of length at least~$5$.
    If we contract any pair of vertices of $C$, then the new vertex created by the contraction already has the red degree at least~$2$, which means that $\tww(C)\geq2$.
    On the other hand, by iteratively contracting pairs of adjacent vertices, we can find a $2$-contraction sequence of~$C$.
    Therefore, $\tww(C)=2$.
\end{proof}

\section{Forbidden structures for twin-width at most {\boldmath$2$} or at most {\boldmath$3$}}\label{sec:structures}

\subsection{Segregated-models}

A \emph{refinement} of a trigraph $H$ is a trigraph $H'$ such that $B(H)\subseteq B(H')\subseteq E(H)$ and $R(H')\subseteq R(H)$.
A \emph{refined subgraph} of~$H$ is a refinement of an induced subgraph of~$H$.

\begin{lemma}[Bonnet, Kim, Thomass\'{e}, and Watrigant~\cite{twin-width1} and Berg\'{e}, Bonnet, and D\'{e}pr\'{e}s~{\cite[Observation~4]{BBD2021}}]\label{lem:refinedtww}
    If a trigraph $H$ is a refined subgraph of a trigraph $G$, then $\tww(H)\le \tww(G)$.
\end{lemma}

A \emph{proper refined subgraph} of $H$ is a refined subgraph that is not equal to~$H$, or equivalently a refined subgraph with fewer vertices or fewer red edges.
Since twin-width does not increase by taking induced subgraphs, every refined subgraph of~$H$ has twin-width at most $\tww(H)$.
It is easy to observe that if $H'$ is a refined subgraph of $H$, then every refined subgraph of $H'$ is a refined subgraph of $H$.

For a collection $\mathcal{X}$ of subsets of $V(H)$, an \emph{$\mathcal{X}$-path} in~$H$ is a path in~$H$ having two ends in distinct members of $\mathcal X$ such that none of its internal vertices are in $\bigcup\mathcal{X}$.
For an integer $t$ and a graph $F$, a \emph{$t$-\sgre{}} of~$F$ in~$H$ is a collection~$\mathcal{X}:=(X_v)_{v\in V(F)}$ of distinct pairwise disjoint nonempty subsets of~$V(H)$ such that for some refined subgraph~$H'$ of~$H$ the following hold.
\begin{enumerate}[label=(\Roman*)]
    \item\label{sgre:connected} For every $v\in V(F)$, $H[X_v]$ is a connected subgraph of $H'$.
        In particular, $X_v\subseteq V(H')$.
    \item\label{sgre:degree2} Every vertex in $V(H')\setminus\bigcup\mathcal{X}$ has degree $2$ in $H'$.
    \item\label{sgre:length} Every $\mathcal{X}$-path in $H'$ of length at most $t$ has a red edge.
    \item\label{sgre:isomorphic} For every pair of adjacent vertices $v$ and $w$ in~$F$, there is an $\mathcal{X}$-path between a vertex in $X_v$ and a vertex in $X_w$ in $H'$.
\end{enumerate}

Observe by \ref{sgre:degree2} that the $\mathcal{X}$-paths in $H'$ are pairwise internally disjoint.
We say that $H'$ \emph{witnesses} that~$\mathcal{X}$ is a $t$-\sgre{} of $F$ in~$H$.
If there exists a $t$-\sgre{} of~$F$ in~$H$, then we say that~$H$ \emph{contains} a $t$-\sgre{} of $F$.
For notational convenience, we call a $1$-\sgre{} a \emph{\sgre{}}.

\begin{lemma}\label{lem:detecting}
    Let $F$ be a graph.
    For every positive integer $t$, if a multigraph $G$ has $F$ as a minor, then all \subd{t}s of~$G$ and the line graphs of all \subd{t+1}s of $G$ contain a $t$-\sgre{} of~$F$.
\end{lemma}
\begin{proof}
    Let $(T_u)_{u\in V(F)}$ be a minor model of $F$ in~$G$.
    Let $H$ be an \subd{t} of~$G$, and for each $e\in E(G)$, let~$P_e$ be the path in~$H$ replacing~$e$.
    For each $u\in V(F)$, let $X_u:=V(T_u)\cup \bigcup_{e\in E(T_u)}V(P_e)$.
    For each $uv\in E(F)$, we choose an edge $f(uv)$ of~$G$ between a vertex in $X_u$ and a vertex in $X_v$.
    It is easy to see that $H':=H[\bigcup_{u\in V(F)}X_u\cup\bigcup_{e\in E(F)}V(P_{f(e)})]$ witnesses that $\mathcal{X}:=(X_u)_{u\in V(F)}$ is a $t$-\sgre{} of~$F$ in~$H$.

    For the line graph, we first observe that the lemma is trivial if $F$ has no edges, and therefore we may assume that $F$ has at least one edge, implying that $H'$ has an edge.
    Now consider $H^*:=L(H')$, and for each $u\in V(F)$, let $X^*_u$ be the set of edges of $H'$ which are incident with a vertex in $X_u$ and let $\mathcal{X}^*:=(X^*_u)_{u\in V(F)}$.
    If~$H'$ is an \subd{t+1} of $H$, then every $\mathcal{X}^*$-path in $H^*$ has length at least $t+1$, and so $H^*$ witnesses that $\mathcal{X}^*$ is a $t$-\sgre{} of $F$ in $L(H)$.
\end{proof}

\begin{lemma}\label{lem:minsgre}
    Let $t$ be a positive integer, let $F$ be a graph with $\delta(F)\geq 3$, and let $H$ be a trigraph containing a $t$-\sgre{} of $F$ such that no proper refined subgraph of $H$ contains a $t$-\sgre{} of $F$.
    If~$\mathcal{X}$ is a $t$-\sgre{} of $F$ which minimises $|\bigcup\mathcal{X}|$, then: 
    \begin{enumerate}[label=\rm(\alph*)]
        \item\label{minsgre:internal} every vertex in $V(H)\setminus\bigcup\mathcal{X}$ is on an $\mathcal{X}$-path in~$H$,
        \item\label{minsgre:uniqueness} for distinct vertices $p$ and $q$ of $F$, $H$ has a unique $\mathcal{X}$-path between $X_p$ and $X_q$ if $pq\in E(F)$, and otherwise $H$ has no $\mathcal{X}$-path between $X_p$ and $X_q$,
        \item \label{minsgre:delta}$\delta(H)\geq 2$, and
        \item \label{minsgre:deg2} for every $p\in V(F)$ and every $z\in X_p$, if $\deg_H(z)=2$, then $N_H(z)\subseteq X_p$.
    \end{enumerate}
\end{lemma}
\begin{proof}
    Note that any refined subgraph $H'$ of $H$ which witnesses that $\mathcal{X}$ is a $t$-\sgre{} of $F$ in~$H$ also witnesses that $\mathcal{X}$ is a $t$-\sgre{} of $F$ in $H'$, and is therefore equal to~$H$.
    By~\ref{sgre:isomorphic}, for every edge $xy\in E(F)$, we can select an $\mathcal{X}$-path $P_{xy}$ between a vertex in~$X_x$ and a vertex in~$X_y$ in~$H$.
    Let $W$ be the set of all internal vertices of paths in $\{P_{rs}:rs\in E(F)\}$
    and let $R$ be the set of edges of $H-E(\bigcup_{rs\in E(F)} E(P_{rs}))$ 
    having two ends in distinct members of $\mathcal X$.
    By \ref{sgre:length}, every edge in $R$ is red.
    Let us consider the refined subgraph $H^*:=H[W\cup \bigcup \mathcal{X}]-R$.
    Trivially, \ref{sgre:connected},~\ref{sgre:length}, and~\ref{sgre:isomorphic} are satisfied with $H':=H^*$.
    Additionally, every vertex in~$W$ has degree at least~$2$ in~$H^*$ since it is an internal vertex of a path, and hence has degree exactly~$2$ by~\ref{sgre:degree2} with $H':=H$.
    Therefore,~$H^*$ witnesses that $\mathcal{X}$ is a $t$-\sgre{} of $F$, and so~$H=H^*$.
    
    Now by~\ref{sgre:degree2} and~\ref{sgre:length}, the set of $\mathcal{X}$-paths of $H^*$ is exactly $\{P_{rs}:rs\in E(F)\}$, which completes the proof of~\ref{minsgre:internal} and~\ref{minsgre:uniqueness}.

    To prove~\ref{minsgre:delta}, suppose for contradiction $H$ has a vertex $z$ of degree at most $1$.
    By~\ref{sgre:degree2}, there is $p\in V(F)$ such that $X_p$ contains $z$.
    Since $\delta(F)\geq 3$,~\ref{sgre:isomorphic} implies that there are at least three~$\mathcal{X}$-paths with an end in~$X_p$, and so $X_p\setminus \{z\}$ is nonempty.
    Now by~\ref{sgre:connected}, $H[X_p]$ is connected, from which we deduce that $z$ has degree exactly $1$ and that the unique neighbour of $z$ is in $X_p$.
    Let $X'_p:=X_p\setminus\{z\}$, and for every $q\in V(F)\setminus\{p\}$, let $X'_q:=X_q$.
    It follows that $(X'_s)_{s\in V(F)}$ is a $t$-\sgre{} of~$F$ in~$H-z$, a contradiction.

    To prove~\ref{minsgre:deg2}, suppose for contradiction that there exist $p\in V(F)$ and $z\in X_p$ such that $\deg_H(z)=2$ and~$z$ has a neighbour which is not in $X_p$.
    To derive a contradiction, we show that $\abs{\bigcup\mathcal{X}}$ is not minimised.
    Since $H[X_p]$ is connected and $\deg_H(z)=2$, if $N_H(z)\cap X_p=\emptyset$, then $X_p=\{z\}$, so that $H$ has at most two $\mathcal{X}$-paths from~$X_p$, contradicting~\ref{sgre:isomorphic} and the assumption that $\delta(F)\ge 3$.
    Thus, $\deg_{H[X_p]}(z)=1$, and therefore $H[X_p\setminus\{z\}]$ is connected.
    Let $X'_p:=X_p\setminus\{z\}$, and for every $q\in V(F)\setminus\{p\}$, let $X'_q:=X_q$.
    Then $(X'_s)_{s\in V(F)}$ is a $t$-\sgre{} of~$F$ in~$H$, contradicting our choice of $\mathcal{X}$.
\end{proof}

\subsection{Properties of {\boldmath$\mathcal{F}_3$}}

We recall that $\mathcal{F}_3=\{K_6^-,\overline{C_7},C_5+\overline{K_2},K_{3,\widehat{1},3},Q_3,V_8\}$; see Figure~\ref{fig:forbidden}.
We have the following observation and lemmas about $\mathcal{F}_3$.

\begin{observation}\label{obs:forbidden}
    For every $G\in\mathcal{F}_3$, the following hold.
    \begin{enumerate}[label=(\roman*)]
        \item\label{obs:mindeg} $\delta(G)\geq3$.
        \item\label{obs:pairneighbour} For distinct vertices $u,v$ of $G$, $\abs{N_G(\{u,v\})}\geq4$.
        \item\label{obs:triangle mindeg} 
        Every vertex belonging to a triangle has degree at least~$4$ in~$G$.
    \end{enumerate}
\end{observation}
A \emph{$(\Delta,Y)$-operation} is an operation on a triangle $T$ of a graph $G$ that constructs a graph from $G-E(T)$ by adding a new vertex with neighbourhood $V(T)$.

\begin{lemma}\label{lem:DeltaY}
    Let $G\in\mathcal{F}_3$ and $G'$ be a graph obtained from $G$ by applying a $(\Delta,Y)$-operation.
    Then $G'$ has a spanning subgraph isomorphic to a graph in $\mathcal{F}_3$.
\end{lemma}
\begin{proof}
    Since neither $Q_3$ nor $V_8$ has a triangle, we may assume that $G\notin\{Q_3,V_8\}$.
    Observe that if $G=K_6^-$, then $G'$ has $K_{3,\widehat{1},3}$ as an induced subgraph.
    Thus, we may assume that $G\neq K_6^-$.
    We now show that $G'$ has a spanning subgraph isomorphic to~$Q_3$ or~$V_8$.
    We label the vertices of~$G$ as in Figure~\ref{fig:forbidden}.
    Let~$T$ be the triangle of~$G$ on which the $(\Delta,Y)$-operation is applied, and let $v_8$ be the new vertex of~$G'$ with neighbourhood $V(T)$.

    Suppose first that $G=\overline{C_7}$.
    By symmetry, we may assume that $T=v_1v_2v_3$.
    Then $G'$ has two disjoint cycles with vertex sets $\{v_8,v_1,v_7,v_2\}$ and $\{v_3,v_5,v_6,v_4\}$.
    Since $G'$ has edges $v_8v_3$, $v_1v_6$, $v_7v_5$, and $v_2v_4$, it has a spanning subgraph isomorphic to~$V_8$.
    
    Suppose that $G=C_5+\overline{K_2}$.
    Since every triangle of~$G$ contains exactly one of~$v_4$ and~$v_5$, which are twins in~$G$, we may assume that $T$ contains~$v_4$.
    By symmetry, we may assume that $T=v_1v_2v_4$.
    Then $G'$ has two disjoint cycles with vertex sets $\{v_8,v_4,v_7,v_1\}$ and $\{v_2,v_3,v_6,v_5\}$.
    Since $G'$ has edges $v_8v_2$, $v_4v_3$, $v_7v_6$, and~$v_1v_5$, it has a spanning subgraph isomorphic to~$Q_3$.
    
    Suppose now that $G=K_{3,\widehat{1},3}$.
    Since every triangle of $G$ contains $v_1$, by symmetry, we may assume that $T=v_1v_4v_5$.
    Then $G'$ has two disjoint cycles with vertex sets $\{v_1,v_3,v_5,v_8\}$ and $\{v_7,v_2,v_6,v_4\}$.
    Since $G'$ has edges $v_1v_7$, $v_3v_6$, $v_5v_2$, and~$v_8v_4$, it has a spanning subgraph isomorphic to~$V_8$.
\end{proof}

We remark that Observation~\ref{obs:forbidden} and Lemma~\ref{lem:DeltaY} are sufficient to prove Proposition~\ref{prop:mainforward}.
Combined with Theorem~\ref{thm:main}\ref{main:atmost3}, this implies that given any class~$\mathcal{F}$ of graphs which satisfies Observation~\ref{obs:forbidden} and Lemma~\ref{lem:DeltaY}, every graph in $\mathcal{F}$ contains some graph in $\mathcal{F}_3$ as a minor\footnote[2]{This fails if we allow graphs to be infinite.
For example, $\mathcal{F}$ may consist of a single infinite tree of minimum degree~$3$, which is $\mathcal{F}_3$-minor-free.}.
Since we may take $\mathcal{F}$ equal to $\mathcal{F}_3$, this uniquely describes the class $\mathcal{F}_3$.

An \emph{internally $4$-connected} graph is a $3$-connected graph $G$ on at least five vertices such that if $G$ has a separator $S$ of size~$3$, then $S$ is an independent set of $G$ and $G-S$ has a vertex $v$ whose neighbourhood is $S$.
For the backward direction of Theorem~\ref{thm:main}, we need the following additional property of the graphs in $\mathcal{F}_3$.

\begin{lemma}\label{lem:internally}
    Every graph $G\in\mathcal{F}_3$ is internally $4$-connected.
\end{lemma}
\begin{proof}
    If $G\in\{Q_3,V_8\}$, then for every pair $\{u,v\}$ of adjacent vertices, $G-\{u,v\}$ has a cycle.
    Thus, both~$Q_3$ and~$V_8$ are internally $4$-connected.
    We can easily see that each of $K_6^-$, $\overline{C_7}$, and $C_5+\overline{K_2}$ is $4$-connected.
    Thus, we may assume that $G=K_{3,\widehat{1},3}$ with vertex-labelling as in Figure~\ref{fig:forbidden}.
    If $S$ is a separator of $G$ having size $3$, then $S$ contains the common neighbourhood of some pair of nonadjacent vertices, and therefore either $\{v_3,v_4\}\subseteq S$ or $\{v_5,v_6\}\subseteq S$.
    Observe that $G-\{v_3,v_4\}$ is isomorphic to $G-\{v_5,v_6\}$, and has a unique separator $\{v_2\}$ of size $1$.
    Thus, we may assume that $S=\{v_2,v_3,v_4\}$.
    Since $S$ is an independent set and $N_G(v_7)=S$, we conclude that $K_{3,\widehat{1},3}$ is internally $4$-connected.
\end{proof}

\section{Subdivisions of twin-width at least {\boldmath$4$}}\label{sec:forward}

Proposition~\ref{prop:mainforward} directly follows from Lemma~\ref{lem:detecting} and the following proposition.

\begin{proposition}\label{prop:segregated}
    If a trigraph $H$ contains a \sgre{} of some graph in $\mathcal{F}_3$, then $\tww(H)\geq4$.
\end{proposition}

We first prove Proposition~\ref{prop:segregated} by using the following lemma.

\begin{lemma}\label{lem:segregated}
    Let $F\in\mathcal{F}_3$ and let $H$ be a trigraph containing a \sgre{} of $F$ such that no proper refined subgraph of~$H$ 
    contains a \sgre{} of $F$.
    For every pair $\{u,v\}$ of distinct vertices of $H$, either $\DeltaR(H/\{u,v\})\geq4$, or $H/\{u,v\}$ contains a \sgre{} of some graph in $\mathcal{F}_3$.
\end{lemma}

\begin{proof}[Proof of Proposition~\ref{prop:segregated} assuming Lemma~\ref{lem:segregated}]
    Suppose that $H$ is a counterexample to Proposition~\ref{prop:segregated} which minimises $\abs{V(H)}$, and subject to this minimum $\abs{R(H)}$.
    Since $\tww(H)\leq3$,
    there exists a pair $\{u,v\}$ of distinct vertices of $H$ such that $H/\{u,v\}$ has twin-width at most~$3$.
    Since the twin-width of a refined subgraph of $H$ is no larger than the twin-width of~$H$, by the assumption, $H$ has no proper refined subgraph containing a \sgre{} of a graph in $\mathcal{F}_3$.
    Since $\DeltaR(H/\{u,v\})\leq3$, by Lemma~\ref{lem:segregated}, $H/\{u,v\}$ contains a \sgre{} of some graph in $\mathcal{F}_3$, so by the assumption on $H$, we deduce that $\tww(H/\{u,v\})\geq4$, a contradiction.
\end{proof}

We now prove Lemma~\ref{lem:segregated}.

\begin{proof}[Proof of Lemma~\ref{lem:segregated}]
    Let $\mathcal{X}:=(X_s)_{s\in V(F)}$ be a \sgre{} of $F$ in~$H$ such that $\abs{\bigcup\mathcal{X}}$ is minimised.
    Let $u$ and $v$ be distinct vertices of~$H$.
    Suppose that $\DeltaR(H/\{u,v\})\leq3$.
    By Lemma~\ref{lem:minsgre}\ref{minsgre:delta}, we have that $N_H[u]\cap N_H[v]\neq\emptyset$.
    We are going to show that $H/\{u,v\}$ contains a \sgre{} of some graph in~$\mathcal{F}_3$.

    By the assumption and Lemma~\ref{lem:minsgre}, $H$ and $\mathcal{X}$ satisfy~\ref{minsgre:internal}--\ref{minsgre:deg2}.
    Let $\mathcal{P}$ be the set of $\mathcal{X}$-paths in~$H$, let $H':=H/\{u,v\}$, and let~$x$ be the new vertex of~$H'$ obtained by contracting~$u$ and~$v$.
    We consider three cases depending on the size of $\{u,v\}\cap(\bigcup\mathcal{X})$.
    Unless specified otherwise, whenever we verify \ref{sgre:connected}--\ref{sgre:isomorphic} for a \sgre{} of a graph in $\mathcal{F}_3$ in $H'$, the \sgre{} will be witnessed by $H'$.
    
    \medskip
    \noindent\textbf{Case 1.} $\abs{\{u,v\}\cap(\bigcup\mathcal{X})}=0$.

    Let $P_u$ and $P_v$ be the $\mathcal{X}$-paths in~$H$ containing $u$ and $v$, respectively.
    If $P_u=P_v$, then for the set~$W$ of internal vertices of the subpath of $P_u$ between~$u$ and~$v$, the refined subgraph $H'-W$ witnesses that $\mathcal{X}$ is a \sgre{} of $F$ in $H'$.
    Thus, we may assume that $P_u\neq P_v$.
    By~\ref{sgre:degree2}, $P_u$ and $P_v$ are internally disjoint in~$H$, and $N[u]\cap N[v]\subseteq V(P_u)\cap V(P_v)$.
    Since $N_H[u]\cap N_H[v]\neq\emptyset$, $P_u$ and $P_v$ share an end~$w$ which is a common neighbour of~$u$ and~$v$ in~$H$.
    By Lemma~\ref{lem:minsgre}\ref{minsgre:uniqueness}, $w$ is the unique common neighbour of~$u$ and~$v$.
    Note that $w\in\bigcup\mathcal{X}$.
    Let $p$ be the vertex of $F$ with $w\in X_p$.
    Let $X'_p:=X_p\cup\{x\}$ and for every $q\in V(F)\setminus\{p\}$, let $X'_q:=X_q$.
    
    We show that $\mathcal{X}':=(X'_s)_{s\in V(F)}$ is a \sgre{} of $F$ in $H'$.
    Since $x$ is adjacent to $w$ in~$H'$, we have that $H'[X'_p]$ is connected, so $\mathcal{X}'$ satisfies~\ref{sgre:connected}.
    For each $y\in\{u,v\}$, let $P'_y:=H'[(V(P_y)\setminus\{y,w\})\cup\{x\}]$.
    Note that each $P'_y$ is an $\mathcal{X}'$-path in~$H'$ with at least one red edge.
    Since $H'-x=H-\{u,v\}$, we have that $(\mathcal{P}\setminus\{P_u,P_v\})\cup\{P'_u,P'_v\}$ is the set of $\mathcal{X}'$-paths in $H'$ and every vertex in $V(H')\setminus \bigcup \mathcal{X}'$ is an internal vertex of one of these paths, and therefore $\mathcal{X}'$ satisfies~\ref{sgre:degree2}, \ref{sgre:length}, and~\ref{sgre:isomorphic}.
    Hence, $\mathcal{X}'$ is a \sgre{} of~$F$ in~$H'$.

    \medskip
    \noindent\textbf{Case 2.} $\abs{\{u,v\}\cap(\bigcup\mathcal{X})}=1$.

    By symmetry, we may assume that $u\in\bigcup\mathcal{X}$.
    Let $p$ be the vertex of $F$ with $u\in X_p$, and let $P_v$ be the $\mathcal{X}$-path in~$H$ containing $v$.
    Suppose first that $P_v$ has an end $w\in X_p$.
    Let $W$ be the set of internal vertices of the subpath $P_v$ between $v$ and $w$.
    Let $X'_p:=(X_p\setminus\{u\})\cup W\cup\{x\}$, and for every $q\in V(F)\setminus\{p\}$, let $X'_q:=X_q$.
    
    We show that $\mathcal{X}':=(X'_s)_{s\in V(F)}$ is a \sgre{} of $F$ in $H'$.
    Since $H[X_p\cup W\cup\{v\}]$ is connected, by the definition of a contraction, $H'[X'_p]$ is connected.
    Thus, $\mathcal{X}'$ satisfies~\ref{sgre:connected}.
    Every vertex $y\in V(H')\setminus\bigcup\mathcal{X}'$ has degree~$2$ in $H'$ because because $\deg_H(y)=2$ and $N_H(y)\neq \{u,v\}$, proving \ref{sgre:degree2} for $\mathcal{X}'$ with~$H'$.
    Let $P'_v:=H'[(V(P_v)\setminus(W\cup\{v,w\}))\cup\{x\}]$.
    Note that~$P'_v$ is an $\mathcal{X}'$-path in $H'$ with at least one red edge.
    Since $\deg_{H'}(y)=2$ for every $y\in W$, no $\mathcal{X}'$-path in~$H'$ has an end in $W$.
    Thus, if $P'$ is an $\mathcal{X}'$-path in $H'$, then $P'=P_v'$ or there is an $\mathcal{X}$-path $P\neq P_v$ in $H$ such that $P'=P$ if $u\notin V(P)$ and $P'=H'[(V(P)\setminus\{u\})\cup \{x\}]$ otherwise.
    Therefore, $\mathcal{X}'$ satisfies~\ref{sgre:length}.
    Conversely, for every $\mathcal{X}$-path~$P\neq P_v$ of $H$ between a vertex in $X_s$ and a vertex in $X_t$, 
    there is an $\mathcal{X}$-path $P'$ in $H'$ between a vertex in $X_s'$ and a vertex in $X_t'$ 
    by $P'=P$ if $u\notin V(P)$
    and 
    $P'=H'[(V(P)\setminus\{u\})\cup \{x\}]$ otherwise,
    verifying~\ref{sgre:isomorphic} for $\mathcal{X}'$ with $H'$.
    Hence, $\mathcal{X}'$ is a \sgre{} of $F$ in~$H'$.

    Thus, we may assume that $P_v$ has no vertex in $X_p$.
    Since $u$ and $v$ are nonadjacent in~$H$ and $N_H[u]\cap N_H[v]\neq\emptyset$, $u$ and~$v$ have a common neighbour in~$H$.
    Since $P_v$ has no vertex in $X_p$, every common neighbour of $u$ and $v$ in~$H$ is an end of~$P_v$.
    
    We show that $u$ and $v$ have no common black neighbour in~$H$.
    Suppose for contradiction that~$u$ and~$v$ have a common black neighbour~$w$ in~$H$.
    Let $q$ be the vertex of $F$ with $w\in X_q$.
    Note that $X_p\neq X_q$.
    Since $u$ is adjacent to~$w$, $H[\{u,w\}]$ is an $\mathcal{X}$-path in~$H$ between $X_p$ and $X_q$, so by~\ref{sgre:length}, $uw$ is a red edge of~$H$, a contradiction.
    Hence,~$u$ and $v$ have no common black neighbour in~$H$, and therefore $\{xy:y\in N_H(u)\cup N_H(v)\}\subseteq R(H')$.
    
    Since $\rdeg_{H'}(x)\leq3$, by Lemma~\ref{lem:minsgre}\ref{minsgre:delta}, we have that $2\leq\deg_H(u)\leq3$.
    Since $P_v$ has no vertex in~$X_p$, and $N_H[u]\cap N_H[v]\neq\emptyset$, by Lemma~\ref{lem:minsgre}\ref{minsgre:deg2}, $u$ has degree~$3$ in~$H$.
    Since $\deg_H(v)=2$ and $\rdeg_{H'}(x)\leq3$, $u$ and $v$ have exactly two common neighbours~$w$ and~$w'$ in~$H$.
    Let~$u'$ be the other neighbour of~$u$ in~$H$.
    Let $q$ and $q'$ be the vertices of $F$ such that $w\in X_q$ and $w'\in X_{q'}$.
    Note that $pqq'$ is a triangle of $F$.
    
    Let $F'$ be the graph obtained from $F$ by applying a $(\Delta,Y)$-operation on $pqq'$.
    By Lemma~\ref{lem:DeltaY}, $F'$ has a spanning subgraph $F''$ isomorphic to a graph in $\mathcal{F}_3$.
    Let $z$ be the vertex in $V(F')\setminus V(F)$.
    Since $H[X_p]$ is connected, if $u'\notin X_p$, then $X_p=\{u\}$ so that $H$ has at most three $\mathcal{X}$-paths from $X_p$, contradicting~\ref{sgre:isomorphic} and Observation~\ref{obs:forbidden}\ref{obs:triangle mindeg}.
    Thus, $u'\in X_p$.
    Let $X'_p:=X_p\setminus\{u\}$, $X'_z:=\{x\}$, and for every $q\in V(F'')\setminus\{p,z\}$, let $X'_q:=X_q$.

    We show that $\mathcal{X}':=(X'_s)_{s\in V(F)}$ is a \sgre{} of~$F''$ in~$H'$.
    Since $u'$ is the only neighbour of~$u$ in~$X_p$, $H[X'_p]$ is connected, so by the definition of a contraction, $H'[X'_p]$ is also connected.
    Thus,~$\mathcal{X}'$ satisfies~\ref{sgre:connected}.
    Since $H'-x=H-\{u,v\}$, by the construction of~$\mathcal{X}'$, we have that
    \[
        (\mathcal{P}\setminus\{P_v,H[\{u,w\}],H[\{u,w'\}]\})\cup\{H'[\{x,w\}],H'[\{x,w'\}],H'[x,u']\}
    \]
    is the set of $\mathcal{X}'$-paths in $H'$ and every vertex in $V(H')\setminus \bigcup \mathcal{X}'$ is an internal vertex of one of these paths, and therefore $\mathcal{X}'$ satisfies~\ref{sgre:degree2}.
    Since $\{xy:y\in N_H(u)\cup N_H(v)\}\subseteq R(H')$, $\mathcal{X}'$ satisfies~\ref{sgre:length}.
    Since $H'$ has an $\mathcal{X}'$-path of length~$1$ from $x\in X'_z$ to each of $u'\in X'_p$, $w\in X'_q$, and $w'\in X'_{q'}$, $\mathcal{X}'$ satisfies~\ref{sgre:isomorphic}.
    Hence,~$\mathcal{X}'$ is a \sgre{} of~$F''$ in $H'$.

    \medskip
    \noindent\textbf{Case 3.} $\abs{\{u,v\}\cap(\bigcup\mathcal{X})}=2$.

    Let $p$ and $q$ be the vertices of $F$ such that $u\in X_p$ and $v\in X_q$.
    If $X_p=X_q$, then
    by Lemma~\ref{lem:minsgre}\ref{minsgre:internal}, no vertex in $V(H)\setminus \bigcup \mathcal{X}$ is adjacent to both $u$ and $v$, and therefore     
    we obtain a \sgre{} of $F$ in $H'$ by replacing $X_p$ with $(X_p\setminus \{u,v\})\cup \{x\}$ in $\mathcal{X}$.
    Thus, we may assume that $X_p\neq X_q$.

    Suppose for contradiction that there is no $\mathcal{X}$-path of~$H$ from $u$ to $v$.
    Since $X_p\neq X_q$, $u$ and~$v$ are nonadjacent in~$H$
    and every common neighbour of~$u$ and~$v$ is in~$\bigcup\mathcal{X}$.
    
    We show that $\deg_H(u)=\deg_H(v)=3$.
    Since $\mathcal{X}$ satisfies~\ref{sgre:length}, we have that $\{xy:y\in N_H(u)\cup N_H(v)\}\subseteq R(H')$.
    Thus, each of $u$ and $v$ has degree at most $3$ in~$H$.
    Since $N_H[u]\cap N_H[v]\neq\emptyset$, by Lemma~\ref{lem:minsgre}\ref{minsgre:delta} and~\ref{minsgre:deg2},~$u$ or~$v$ has degree~$3$ in~$H$.
    If the other vertex has degree~$2$ in~$H$, then by Lemma~\ref{lem:minsgre}\ref{minsgre:uniqueness} and~\ref{minsgre:deg2},~$u$ and~$v$ have at most one common neighbour in~$H$, so $\rdeg_{H'}(x)\geq4$, a contradiction.
    Therefore, $\deg_H(u)=\deg_H(v)=3$.

    Since $u$ and $v$ are nonadjacent and $\rdeg_{H'}(x)\leq3$, $u$ and~$v$ have the same neighbourhood in~$H$.
    By Lemma~\ref{lem:minsgre}\ref{minsgre:uniqueness}, $u$ and $v$ have at most one common neighbour in $X_p\cup X_q$, and therefore $X_p$ or $X_q$ is a singleton since $\mathcal{X}$ satisfies~\ref{sgre:connected}.
    If both $X_p$ and $X_q$ are singletons, then $p$ and $q$ have the same neighbourhood in~$F$ by Lemma~\ref{lem:minsgre}\ref{minsgre:uniqueness}, contradicting Observation~\ref{obs:forbidden}\ref{obs:pairneighbour}.
    By symmetry, we may assume that $X_p=\{u\}$, and that $u$ and $v$ have a common neighbour in $X_q$.
    Let $w$ be a vertex in $(N_H(u)\cap N_H(v))\setminus X_q$, and let $r$ be the vertex of $F$ with $w\in X_r$.
    Thus, by Lemma~\ref{lem:minsgre}\ref{minsgre:uniqueness}, $pqr$ is a triangle of~$F$ and $\deg_F(p)=3$, contradicting Observation~\ref{obs:forbidden}\ref{obs:triangle mindeg}.
    
    Thus, there is an $\mathcal{X}$-path $P$ in~$H$ from $u$ to $v$.
    By Lemma~\ref{lem:minsgre}\ref{minsgre:uniqueness}, $u$ and $v$ have no common neighbour in $X_p\cup X_q$.
    Since~$u$ and~$v$ are the ends of~$P$, by Lemma~\ref{lem:minsgre}\ref{minsgre:delta} and \ref{minsgre:deg2}, each of~$u$ and~$v$ has degree at least~$3$ in~$H$.
    Since $\rdeg_{H'}(x)\leq3$, $u$ and $v$ have at least one common neighbour in $V(H)\setminus V(P)$.
    Let $N_{uv}:=(N_H(u)\cap N_H(v))\setminus V(P)$.

    By Lemma~\ref{lem:minsgre}\ref{minsgre:uniqueness}, $N_{uv}\subseteq(\bigcup\mathcal{X})\setminus (X_p\cup X_q)$.
    For every $y\in N_{uv}$, let $y^*$ be the vertex of $F$ with $y\in X_{y^*}$, and note that $pqy^*$ is a triangle of~$F$.
    Additionally note that both $H[\{u,y\}]$ and $H[\{v,y\}]$ are $\mathcal{X}$-paths in~$H$ of length~$1$, so by~\ref{sgre:length}, both $uy$ and $vy$ are red edges of $H$.
    This implies that $\{xy:y\in (N_H(u)\cup N_H(v))\setminus V(P)\}\subseteq R(H')$, and both~$u$ and~$v$ have degree at most~$4$ in~$H$ because $\rdeg_{H'}(x)\le 3$.

    We are going to show that both~$u$ and~$v$ have degree~$3$ in~$H$.
    First, suppose for contradiction that both~$u$ and~$v$ have degree $4$ in~$H$.
    Since $\rdeg_{H'}(x)\leq3$, we have that $N_H(u)\setminus V(P)=N_H(v)\setminus V(P)=N_{uv}$.
    Since $N_{uv}\subseteq(\bigcup\mathcal{X})\setminus (X_p\cup X_q)$, we deduce that both $X_p$ and $X_q$ are singletons, so that $\abs{N_F(\{p,q\})}\leq3$, contradicting~\ref{sgre:isomorphic} and Observation~\ref{obs:forbidden}\ref{obs:pairneighbour}.
    Hence, $u$ or $v$ has degree~$3$ in~$H$.
    By symmetry, we may assume that $v$ has degree $3$ in~$H$.

    Suppose for contradiction that $u$ has degree $4$ in~$H$.
    Since $\{xy:y\in(N_H(u)\cup N_H(v))\setminus V(P)\}\subseteq R(H')$ and $\deg_H(v)=3$, $N_{uv}$ contains at least two vertices, and therefore $X_q$ is a singleton.
    Thus, $\deg_F(q)=3$.
    However, since $pqy^*$ is a triangle of $F$ for each $y\in N_{uv}$, $q$ is a vertex of some triangle of $F$, contradicting~\ref{sgre:isomorphic} and Observation~\ref{obs:forbidden}\ref{obs:triangle mindeg}.
    Hence, both $u$ and $v$ have degree $3$ in~$H$.

    Since $F$ has a triangle containing the edge $pq$, by~\ref{sgre:isomorphic} and Observation~\ref{obs:forbidden}\ref{obs:triangle mindeg}, both $p$ and $q$ have degree at least~$4$ in~$F$, and therefore neither $X_p$ nor $X_q$ is a singleton as $\deg_H(u)=\deg_H(v)=3$.
    Since $N_{uv}$ is nonempty, both $u$ and $v$ have neighbours in $P$, and 
    $\deg_{H'}(x)\le 3$, we deduce that $\abs{N_{uv}}=1$, $\abs{V(P)}\leq3$, and $\deg_{H[X_p]}(u)=\deg_{H[X_q]}(v)=1$.
    Let $w$ be the vertex in $N_{uv}$, and $r$ be the vertex of $F$ with $w\in X_r$.
    Note that $pqr$ is a triangle of $F$.
    Let $u'\in N_H(u)\cap X_p$ and $v'\in N_H(v)\cap X_q$.
    Note that $H[X_p\setminus\{u\}]$ and $H[X_q\setminus\{v\}]$ are connected, and so are $H'[X_p\setminus\{u\}]$ and $H'[X_q\setminus\{v\}]$.

    Let $F'$ be the graph obtained from $F$ by applying a $(\Delta,Y)$-operation on $pqr$.
    By Lemma~\ref{lem:DeltaY}, $F'$ has a spanning subgraph $F''$ isomorphic to a graph in~$\mathcal{F}_3$.
    Let $z$ be the vertex in $V(F')\setminus V(F)$.
    Let $X'_p:=X_p\setminus\{u\}$, $X'_q:=X_q\setminus\{v\}$, $X'_z:=(V(P)\cup\{x\})\setminus\{u,v\}$, and for every $t\in V(F'')\setminus\{p,q,z\}$, let $X'_t:=X_t$.
    
    We show that $\mathcal{X}':=(X'_s)_{s\in V(F)}$ is a \sgre{} of~$F''$ in~$H'$.
    Since $H[X'_p]$ and $H[X'_q]$ are connected, by the definition of a contraction, $H'[X'_p]$ and $H'[X'_q]$ are connected.
    Since $P$ is connected, $H'[X'_z]$ is connected.
    Since $H'-x=H-\{u,v\}$, we have that
    \[
        (\mathcal{P}\setminus\{V(P),H[\{u,w\}],H[\{v,w\}]\})\cup\{H'[\{u',x\}],H'[\{v',x\}],H'[\{w,x\}]\}
    \]
    is the set of $\mathcal{X}'$-paths in $H'$ and every vertex in $V(H')\setminus \bigcup \mathcal{X}'$ is an internal vertex of one of these paths, and therefore $\mathcal{X}'$ satisfies~\ref{sgre:degree2}.
    Since every $\mathcal{X}'$-path containing $x$ has a red edge incident with $x$, $\mathcal{X}'$ satisfies~\ref{sgre:length}.
    Since $H'$ has an $\mathcal{X}'$-path from $X'_z$ to each of $X'_p$, $X'_q$, and $X'_r$, $\mathcal{X}'$ satisfies~\ref{sgre:isomorphic}.
    Hence, $\mathcal{X}'$ is a \sgre{} of~$F''$ in~$H'$, and this completes the proof.
\end{proof}

\section{Subdivisions of twin-width at most {\boldmath$3$}}\label{sec:backward}

In this section, we complete the proof of Theorem~\ref{thm:main}\ref{main:atmost3}.
In fact, we prove the following stronger result.

\begin{proposition}\label{prop:mainbackward}
    Let $G$ be a multigraph and let $H$ be a trigraph with $\DeltaR(H)\leq3$ whose underlying graph is an \subd{2} of~$G$.
    Then $\tww(H)\leq3$ if and only if $G$ is $\mathcal{F}_3$-minor-free.
\end{proposition}

To prove Proposition~\ref{prop:mainbackward}, we first observe that it suffices to consider multigraphs $G$ obtained from edge-maximal $\mathcal{F}_3$-minor-free simple graphs by adding parallel edges or loops, since an \subd{2} of a subgraph of~$G$ is an induced subgraph of an \subd{2} of~$G$.
This allows us to construct a special type of a tree decomposition $(T,(B_t)_{t\in V(T)})$ of~$G$, called a $3$-contractible tree decomposition, where every bag induces one of a small set of specific graphs and every adhesion set is a clique of size at most~$3$.
We then prove the statement by induction on $\abs{V(T)}$.

We organise this section as follows.
In Subsection~\ref{subsec:psi}, we introduce \sena{}s of trigraphs and present specific partial contraction sequences.
In Subsection~\ref{subsec:treedecomp}, we fully describe the specific graphs that may be induced by a bag of a $3$-contractible tree decomposition.
In Subsection~\ref{subsec:construct}, we construct a $3$-contractible tree decomposition, and in Subsection~\ref{subsec:backwardproof}, we prove Proposition~\ref{prop:mainbackward}.

\subsection{{\boldmath$\nabla$}-separations and reduction to smaller trigraphs}\label{subsec:psi}

For a graph $G$, we denote by $\theta(G)$ the union of a $2$-subdivision~$G_1$ of~$G$ and a $3$-subdivision~$G_2$ of~$G$ with $V(G_1)\cap V(G_2)=V(G)$.
If $G'$ is a subgraph of~$G$, then we assume that $\theta(G')$ is a subgraph of $\theta(G)$.
Note that for adjacent vertices $u$ and $v$ in~$G$, $\theta(G[\{u,v\}])$ is a cycle of length~$7$ in $\theta(G)$ containing~$u$ and~$v$.

For an \subd{2} of a multigraph~$G$, the following lemma presents a partial contraction sequence from the subdivision to a trigraph whose underlying graph is $\theta(G_0)$ for the simplification~$G_0$ of~$G$.

\begin{lemma}\label{lem:theta}
    Let $G$ be a multigraph, let $G_0$ be the simplification of $G$, and let $H$ be a trigraph whose underlying graph is an \subd{2} of~$G$.
    Let $\mu:=1$ if at least one vertex in $V(G)$ has red degree at least~$1$ in~$H$, and otherwise let~$\mu:=0$.
    For $d:=\max\{\DeltaR(H),2+\mu\}$, there exists a partial \cont{d} from~$H$ to a trigraph $H'$ such that the underlying graph of~$H'$ is an induced subgraph of $\theta(G_0)$, and $\rdeg_{H'}(u)\leq\rdeg_H(u)$ for every $u\in V(G)$.
\end{lemma}

Before presenting the proof of this lemma, we sketch the idea of the proof.
First, observe that for a loop~$e$ of~$G$ incident with a vertex~$v$, it is possible to perform contractions to convert the cycle of~$H$ corresponding to~$e$ to a path of length at most~$2$, without increasing the red degree of~$v$.
For a non-loop edge~$e'$ of~$G$ with ends~$v$ and~$w$, if the path corresponding to~$e'$ in $H$ has length at least~$4$ then it can be contracted down to a path of length exactly~$4$ without increasing the red degree of~$v$ or~$w$.

After performing these reductions, we can proceed to merge a short path~$P$ which we obtain from the cycle corresponding to any loop~$e$ with any path corresponding to a non-loop edge incident with~$e$ whose length is at least~$2$ more than the length of~$P$.
We can also merge two paths corresponding to parallel edges of~$G$, provided that they have the same length.
Finally, suppose that~$G$ has loops~$e_1$ and~$e_2$ incident with distinct adjacent vertices~$v$ and~$w$, and we have contracted the cycles corresponding to these loops down to paths~$P_1$ and~$P_2$, each of length~$2$.
We can contract an end of~$P_1$ with an end of~$P_2$, to form a path of length~$4$ from~$v$ to~$w$.

When it is no longer possible to perform any of these reductions, we must have an induced subgraph of $\theta(G_0)$.
We now provide a formal proof of Lemma~\ref{lem:theta}.

\begin{proof}[Proof of Lemma~\ref{lem:theta}]
    We proceed by induction on $t:=\abs{E(G)}$.
    The statement vacuously holds for $t=0$.
    Thus, we may assume that $t>0$.
    
    We may assume that $G$ is connected.
    For each edge~$e$ of~$G$, we denote by $P_e$ the path or the cycle in~$H$ replacing $e$, and by $\ell(e)$ the length of~$P_e$.
    Since the underlying graph of $H$ is an \subd{2} of~$G$, for every edge $e$ of $G$, we have that $\ell(e)\geq3$.
    For every edge~$e$ of~$G$, we arbitrarily select an end $w_0^e$ of $e$ and denote by $w_0^e,w_1^e,\ldots,w_{\ell(e)}^e$ the vertices of $P_e$ such that for each $i\in[\ell(e)]$, $w_i^e$ is adjacent to $w_{i-1}^e$ in~$P_e$.
    Thus, $\{w_0^e,w_{\ell(e)}^e\}$ is the set of ends of $e$.
    Let $L$ be the set of loops of~$G$, let $E_1:=\{e\in E(G):\ell(e)=3\}$, and let $E_2:=\{e\in E(G):\ell(e)\geq4\}$.
    For each $i\in[2]$, let $G_i:=G-E_{3-i}$.

    Suppose first that $\abs{V(G)}=1$.
    If $E_2=\emptyset$, then we obtain a \cont{d} from~$H$ by iteratively contracting pairs of adjacent degree-$2$ vertices or pairs of degree-$1$ vertices, until only the vertex of $G$ remains.
    If $E_2\neq\emptyset$, then we arbitrarily choose $e\in E_2$, and for each $e'\in E(G)\setminus\{e\}$, apply the following partial \cont{d}.
    \begin{enumerate}[label=\bf{Step \arabic*.},leftmargin=*]
        \item Contract $\{w_1^{e'},w_{\ell(e')-1}^{e'}\}$ to~$w_1^{e'}$. Observe that the red degree of $w_1^{e'}$ in the resulting trigraph is $2$ if both $w_1^{e'}$ and $w_{\ell(e')-1}^{e'}$ are black neighbours of $w_0^{e'}$ in~$H$, and $3$ otherwise.
        \item If $\ell(e')\geq5$, then for~$j$ from~$3$ to~$\ell(e')-2$, contract $\{w_2^{e'},w_j^{e'}\}$ to~$w_2^{e'}$.
        \item Contract $\{w_1^e,w_1^{e'}\}$ to $w_1^e$, and then contract $\{w_2^e,w_2^{e'}\}$ to $w_2^e$.
    \end{enumerate}
    Finally, we iteratively select a neighbour~$x$ of~$w_0^e$ and contract $\{w_0^e,x\}$ to $w_0^e$ to obtain a \cont{d} of the resulting trigraph.

    We now assume that $\abs{V(G)}\geq2$.
    Since $G$ is connected, we can construct a function $\eta:L\to E(G)\setminus L$ such that for every $e\in L$, $\eta(e)$ is incident with $e$, and $\eta(e)\in E_2$ unless no edge in $E_2\setminus L$ is incident with $e$.

    \medskip
    \noindent\textbf{Case 1.} $G_1$ has parallel edges~$e$ and~$e'$.
    
    Without loss of generality, we may assume that $w_0^e=w_0^{e'}$.
    We contract $\{w_1^e,w_1^{e'}\}$ to $w_1^{e'}$, and then contract $\{w_2^e,w_2^{e'}\}$ to $w_2^{e'}$.
    Let $H_e$ be the resulting trigraph.
    Observe that the underlying graph of~$H_e$ is an \subd{2} of~$G-\{e\}$ and $\rdeg_{H_e}(u)\leq\rdeg_H(u)$ for every $u\in V(G)$.
    Since $G-\{e\}$ and $G$ have the same simplification, the result now follows by induction.

    \medskip
    \noindent\textbf{Case 2.} $G_1$ has a loop $e$.
    
    Without loss of generality, we may assume that $w_0^e=w_0^{\eta(e)}$.
    We contract $\{w_1^e,w_2^e\}$ to~$w_1^e$, and then contract $\{w_1^e,w_1^{\eta(e)}\}$ to~$w_1^{\eta(e)}$.
    Let $H_e$ be the resulting trigraph.
    Again, the underlying graph of~$H_e$ is an \subd{2} of~$G-\{e\}$ and $\rdeg_{H_e}(u)\leq\rdeg_H(u)$ for every $u\in V(G)$.
    Since $G-\{e\}$ and $G$ have the same simplification, the result now follows by induction.

    \medskip
    \noindent\textbf{Case 3.} $G_2$ has parallel edges $e$ and $e'$.
    
    Without loss of generality, we may assume that $w_0^e=w_0^{e'}$.
    We apply the following partial \cont{d} from~$H$.
    \begin{enumerate}[label=\bf{Step \arabic*.},leftmargin=*]
        \item For each $f\in\{e,e'\}$, if $\ell(f)\geq5$, then for $i$ from $3$ to $\ell(f)-2$, contract $\{w_2^f,w_i^f\}$ to $w_2^f$.
        \item Contract the following three pairs of vertices in order: $\{w_1^e,w_1^{e'}\}$ to $w_1^{e'}$, $\{w_{\ell(e)-1}^e,w_{\ell(e')-1}^{e'}\}$ to~$w_{\ell(e')-1}^{e'}$, and $\{w_2^e,w_2^{e'}\}$ to~$w_2^{e'}$.
    \end{enumerate}
    Let $H_e$ be the resulting trigraph.
    Again, the underlying graph of~$H_e$ is an \subd{2} of~$G-\{e\}$, and $\rdeg_{H_e}(u)\leq\rdeg_H(u)$ for every $u\in V(G)$.
    Since $G-\{e\}$ and $G$ have the same simplification, the result now follows by induction.

    \medskip
    \noindent\textbf{Case 4.} $G_2$ has a loop $e$ with $\eta(e)\in E_2$.
    
    Without loss of generality, we may assume that $w_0^e=w_0^{\eta(e)}$.
    We apply the following partial \cont{d} from~$H$.
    \begin{enumerate}[label=\bf{Step \arabic*.},leftmargin=*]
        \item Contract $\{w_1^e,w_{\ell(e)-1}^e\}$ to $w_1^e$.
        \item If $\ell(e)\geq5$, then for $i$ from $3$ to $\ell(e)-2$, contract $\{w_2^e,w_i^e\}$ to $w_2^e$.
        \item Contract $\{w_1^e,w_1^{\eta(e)}\}$ to $w_1^{\eta(e)}$, and then contract $\{w_2^e,w_2^{\eta(e)}\}$ to~$w_2^{\eta(e)}$.
    \end{enumerate}
    Let $H_e$ be the resulting trigraph.
    Again, the underlying graph of~$H_e$ is an \subd{2} of~$G-\{e\}$, and $\rdeg_{H_e}(u)\leq\rdeg_H(u)$ for every $u\in V(G)$.
    Since $G-\{e\}$ and $G$ have the same simplification, the result now follows by induction.

    \medskip
    \noindent\textbf{Case 5.} $G_1$ and $G_2-L$ are both simple, and for every $e\in L$, $\eta(e)\in E_1$.
    
    Note that no loop of $G$ is incident with a non-loop edge of $G_2$.
    By Case 3, we may assume that every vertex of~$G$ is incident with at most one loop.
    For each $e\in L$, we apply the following partial \cont{d}.
    \begin{enumerate}[label=\bf{Step \arabic*.},leftmargin=*]
        \item Contract $\{w_1^e,w_{\ell(e)-1}^e\}$ to~$w_1^e$.
        \item If $\ell(e)\geq5$, then for~$j$ from~$3$ to~$\ell(e)-2$, contract $\{w_2^e,w_j^e\}$ to~$w_2^e$.
    \end{enumerate}
    Let $H^*$ be the resulting trigraph.
    For each edge $f$ in $E(G)\setminus L$, we now apply the following partial \cont{d}.
    \begin{enumerate}[label=\bf{Step \arabic*.},leftmargin=*]
        \setcounter{enumi}{2}
        \item If $\ell(f)\geq5$, then for $i$ from $3$ to $\ell(f)-2$, contract $\{w_2^f,w_i^f\}$ to $w_2^f$.
        \item If there exist distinct loops $e_1$ and $e_2$ of~$G$ such that $\eta(e_1)=\eta(e_2)=f$, then contract $w_2^{e_1}$ and $w_2^{e_2}$.
    \end{enumerate}
    Let $H'$ be the resulting trigraph.
    Note that $\rdeg_{H'}(u)\leq\rdeg_H(u)$ for every $u\in V(G)$.
    Since there is no loop~$e$ of $G$ such that $\eta(e)$ is parallel to an edge in $E_2$, the underlying graph of~$H'$ is an induced subgraph of~$\theta(G_0)$.
    Hence, the resulting partial \cont{d} is as desired.
\end{proof}

A \emph{\sena{}} of a trigraph $H$ is a separation $(U,U')$ of $H$ such that $\abs{U\cap U'}\geq2$, $H[U]$ is a cycle, $H[U']$ is connected, and $\dist_{H[U]}(v,w)\geq3$ for distinct $v,w\in U\cap U'$.
Note that for a \sena{} $(U,U')$ of~$H$, every vertex in $U\setminus U'$ has degree~$2$ in~$H$.

Let $H_0$ be the underlying graph of a trigraph~$H$.
For a \sena{} $S=(U,U')$ of $H$, let $\psi^*(H,S)$ be the graph such that for some set $W_S:=\{x_S\}\cup\{y_{u,S}:u\in U\cap U'\}$ disjoint from $V(H)$,
\begin{linenomath}
\begin{align*}
    V(\psi^*(H,S))&:=U\cup W_S,\\
    E(\psi^*(H,S))&:=E(H_0[U])\cup\{uy_{u,S}:u\in U\cap U'\}\cup\{x_Sy_{u,S}:u\in U\cap U'\}.
\end{align*}
\end{linenomath}
In other words, $\psi^*(H,S)$ is a graph obtained from $H_0[U]$ by adding a new vertex~$x_S$ and a length-$2$ path from~$x_S$ to every vertex in~$U\cap U'$.
For a set~$\mathcal{S}$ of \sena{}s of~$H$, let 
\[
    \psi(H,\mathcal{S}):=H_0\cup\bigcup_{S\in\mathcal{S}}\psi^*(H,S).
\]
If $\mathcal{S}:=\{S\}$, then we may write $\psi(H,S)$ for $\psi(H,\{S\})$.

For a $3$-vertex clique $S$ in a graph~$G$, let $C_{G,S}$ be the unique length-$9$ cycle in~$\theta(G)$ with $V(C_{G,S})\cap V(G)=S$, let $C'_{G,S}$ be the unique length-$12$ cycle in~$\theta(G)$ with $V(C'_{G,S})\cap V(G)=S$, and let
\[
    \sigma_G(S):=(V(C_{G,S}),(V(\theta(G))\setminus V(C_{G,S}))\cup S),
\]
which is a \sena{} of~$\theta(G)$ if $G$ is connected.
For a set $\mathcal{A}$ of $3$-vertex cliques in~$G$, we denote by~$\sigma_G(\mathcal{A})$ the set $\{\sigma_G(S):S\in\mathcal{A}\}$.

For a trigraph $H$ and a set $X\subseteq V(H)$, we say that a partial contraction sequence from~$H$ is \emph{$X$-fixing} if no vertex in $X$ is contracted in the sequence, and is \emph{$X$-stable} if it is $X$-fixing and each vertex $v\in X$ has red degree at most $\rdeg_H(v)$ in every trigraph of the sequence.

We often use the following partial contraction sequences of Lemma~\ref{lem:yoperation}--\ref{lem:routine}.

\begin{lemma}\label{lem:yoperation}
    Let $\mathcal{A}$ be a set of $3$-vertex cliques in a connected graph~$G$ and let $H$ be a trigraph with underlying graph~$\psi(\theta(G),\sigma_G(\mathcal{A}))$.
    Let
    \[
        W:=\bigcup_{S\in\mathcal{A}}(V(C'_{G,S})\setminus S)\quad\text{ and }\quad Z:=V(\theta(G))\setminus W.
    \]
    For $d:=\max\{\DeltaR(H),3\}$, there exists a $Z$-stable partial \cont{d} from~$H$ to a trigraph with underlying graph $\psi(\theta(G),\sigma_G(\mathcal{A}))-W$.
\end{lemma}
\begin{proof}
    Note that every component $P$ of $H[W]$ is a path of length $2$.
    For every component $P$ of $H[W]$, we denote by $z_P$ its internal vertex, and choose $S(P):=\sigma_G(S)$ for some $S\in\mathcal{A}$ with $V(P)\subseteq V(C'_{G,S})$.
    If there is more than one such $S$, then we arbitrarily select one of them to define $S(P)$.
    For each component~$P$ of~$H[W]$, we do the following in order.
    \begin{enumerate}[label=\bf{Step \arabic*.},leftmargin=*]
        \item For each end $p$ of $P$, contract $\{p,y_{u,S(P)}\}$ to $y_{u,S(P)}$, where $u$ is the neighbour of $p$ in~$V(G)$.
        \item Contract $\{x_{S(P)},z_P\}$ to $x_{S(P)}$.
    \end{enumerate}

    Note that the underlying graph of the resulting trigraph is $\psi(\theta(G),\sigma_G(\mathcal{A}))-W$.
    Throughout the resulting partial contraction sequence, we always keep the following.
    \begin{itemize}
        \item For each contraction of a pair $\{a,b\}$, we have that $\{a,b\}\subseteq V(H)\setminus Z$ and $N_H(a)\cap Z=N_H(b)\cap Z$.
        This implies that the sequence is $Z$-stable.
        \item No vertex resulting from a contraction has degree more than $3$, and the red edges which are newly created are not incident with any vertices of red degree more than $d$.
        This implies that the sequence is a partial \cont{d}.
    \end{itemize}
    Thus, the partial contraction sequence is as desired.
\end{proof}

\begin{lemma}\label{lem:protection1}
    Let $\{t_1,t_2,t_3\}$ be a clique of size $3$ in a graph $G$, let $F$ be a trigraph with underlying graph $\psi^*(\theta(G),\sigma_G(\{t_1,t_2,t_3\}))$, and let $t_4$ be the unique degree-$3$ vertex in $F-\{t_1,t_2,t_3\}$.
    For all distinct~$i$ and~$j$ in~$[3]$, let $t_{i,j}$ and $t_{j,i}$ be the internal vertices of the unique length-$3$ path in $F$ between~$t_i$ and~$t_j$ such that $t_it_{i,j}\in E(F)$, and let $y_i$ be the common neighbour of $t_i$ and $t_4$ in $F$.
    Let $t'_1:=t_{1,2}$ if $t_1t_{1,3}\in B(F)$, and otherwise let $t'_1:=t_{1,3}$.
    If $t_1t_{1,3}\in R(F)$ or $t_2t_{2,3}\in B(F)$, then there is a $\{t_3\}$-stable $\{t_1,t_2\}$-fixing partial \cont{3} $H_1,\ldots,H_9$ with $H_1=F$ such that
    \begin{itemize}
        \item $H_2$ is obtained from $H_1$ by contracting $\{t'_1,y_1\}$ to $y_1$,
        \item for each $i\in\{2,3\}$, $H_{i+1}$ is obtained from $H_i$ by contracting $\{t_{i,1},y_i\}$ to $y_i$,
        \item for each $i\in [9]$ and each $j\in [2]$, we have that $\rdeg_{H_i}(t_j)\leq\max\{\rdeg_F(t_j),1\}$, and
        \item the underlying graph of~$H_9$ is the graph on $\{t_1,t_2,t_3,t_4,y_3\}$ with the edges $t_1t_4$, $t_2t_4$, $t_3y_3$, and $t_4y_3$.
    \end{itemize}
\end{lemma}
\begin{proof}
    We present a desired partial \cont{d} from~$H_4$ to~$H_9$ as follows.
    
    If $t'_1=t_{1,3}$, then we contract the following five pairs of vertices in order: $\{t_4,y_1\}$ to~$t_4$, $\{t_{1,2},t_4\}$ to~$t_4$, $\{t_{3,2},y_3\}$ to~$y_3$, $\{t_{2,3},y_2\}$ to~$y_2$, and $\{t_4,y_2\}$ to~$t_4$.
    
    If $t'_1=t_{1,2}$, then we contract the following five pairs of vertices in order: $\{t_4,y_1\}$ to~$t_4$, $\{t_{1,3},t_4\}$ to~$t_4$, $\{t_4,y_2\}$ to~$t_4$, $\{t_{3,2},y_3\}$ to $y_3$, and $\{t_{2,3},t_4\}$ to~$t_4$.

    It is straightforward to deduce that the partial contraction sequence is as desired.
\end{proof}

For convenience, we often use the following corollary of Lemma~\ref{lem:protection1}.

\begin{corollary}\label{cor:protection1}
    Let $\{t_1,t_2,t_3\}$ be a clique of size $3$ in a graph $G$, let $F$ be a trigraph with underlying graph $\psi^*(\theta(G),\sigma_G(\{t_1,t_2,t_3\}))$, and let $t_4$ be the unique degree-$3$ vertex in $F-\{t_1,t_2,t_3\}$.
    Let $y_3$ be the common neighbour of~$t_3$ and~$t_4$ in~$F$.
    There is a $\{t_3\}$-stable $\{t_1,t_2\}$-fixing partial \cont{3} from $F$ to a trigraph whose underlying graph is the graph on $\{t_1,t_2,t_3,t_4,y_3\}$ with the edges $t_1t_4$, $t_2t_4$, $t_3y_3$, and $t_4y_3$.
\end{corollary}
\begin{proof}
    Let $t_{1,3}$ be the internal vertices of the unique length-$3$ path in $F$ between~$t_1$ and~$t_3$ such that $t_1t_{1,3}\in E(F)$.
    If $t_1t_{1,3}$ is a red edge, then we apply Lemma~\ref{lem:protection1}.
    If $t_1t_{1,3}$ is a black edge, then we apply Lemma~\ref{lem:protection1} after swapping $t_1$ and $t_2$.
\end{proof}

\begin{lemma}\label{lem:protection2}
    Let $\{t_1,t_2,t_3\}$ be a clique of size~$3$ in a graph $G$, let $F$ be a trigraph with underlying graph $\psi^*(\theta(G),\sigma_G(\{t_1,t_2,t_3\})$, and let $t_4$ be the unique degree-$3$ vertex in $F-\{t_1,t_2,t_3\}$.
    For all distinct~$i$ and~$j$ in~$[3]$, let $t_{i,j}$ and $t_{j,i}$ be the internal vertices of the unique length-$3$ path in~$F$ between~$t_i$ and~$t_j$ such that $t_it_{i,j}\in E(F)$, and let $y_i$ be the common neighbour of $t_i$ and $t_4$ in~$F$.
    There is a $\{t_1,t_2\}$-stable $\{t_3\}$-fixing partial \cont{3} $H_1,\ldots,H_5$ with $H_1=F$ such that
    \begin{itemize}
        \item for each $i\in[5]$, $\rdeg_{H_i}(t_3)\leq\max\{\rdeg_F(t_3),1\}$, and
        \item the underlying graph of $H_5$ is the graph obtained from the underlying graph of $F-\{t_{1,3},t_{3,1},t_{2,3},t_{3,2},y_3\}$ by adding an edge $t_3t_4$.
    \end{itemize}
\end{lemma}
\begin{proof}
    By symmetry, we may assume that $t_3t_{3,1}\in R(F)$ or $t_3t_{3,2}\in B(F)$.
    We present a desired partial \cont{3} by contracting the following five pairs of vertices in order: $\{t_{1,3},y_1\}$ to~$y_1$, $\{t_{3,1},y_3\}$ to~$y_3$, $\{y_3,t_4\}$ to~$t_4$, $\{t_{2,3},y_2\}$ to~$y_2$, and $\{t_{3,2},t_4\}$ to~$t_4$.
\end{proof}

\begin{lemma}\label{lem:routine}
    Let $H$ be a trigraph with underlying graph $H'$, let $H_0$ be an induced subgraph of~$H'$ isomorphic to the $1$-subdivision of $K_{1,3}$, and let $L$ be the set of leaves of $H_0$.
    If $H'$ is either an \subd{1} of $K_{3,3}$ or a subdivision of $K_{3,3}$ such that $H'-V(H_0)$ contains a subdivision of $K_{2,3}$ as a subgraph, then there exists an $L$-stable partial \cont{3} from~$H$ to a trigraph with underlying graph $H_0$.
\end{lemma}
\begin{proof}
    Throughout the following process, we will not contract any pair of vertices in $V(H_0)$, and whenever we contract a pair $\{u,v\}$ with $u\in V(H_0)$, we contract $\{u,v\}$ to $u$.
    We first iteratively contract pairs consisting of a degree-$2$ vertex and one of its neighbours such that both of them are not in $N_H[L]$, until no such pairs remain.
    In the resulting trigraph, every degree-$2$ vertex is in the closed neighbourhood of $L$.
    For every $v\in L$ with degree~$3$ in~$H$, we then contract the two neighbours of $v$ which are not in $H_0$.
    At this point, if we were to delete $V(H_0)$ and every vertex of degree $2$ in the neighbourhood of $L$ from the resulting trigraph, we would have a trigraph $H^*$ whose underlying graph is isomorphic to $K_{2,3}$.
    Let $u$ and $v$ be the degree-$3$ vertices in~$H^*$.
    We contract $u$ and $v$ to obtain a trigraph $H''$ whose underlying graph is isomorphic to an \subd{1} of $K_{2,3}$.

    Let $x$ and $y$ be the degree-$3$ vertices in $H''$ where $x\in V(H_0)$ and $y\notin V(H_0)$.
    Observe that in $H''$, the distance from each vertex of $L$ to $x$
    is $2$ and 
    the distance from each vertex of $L$ to $y$ is $2$ or $3$.
    For each length-$3$ path from $y$ to $L$, we contract $y$ and its neighbour on the path to $y$.
    Then we obtain a trigraph in which the distance between every vertex of $L$ and every degree-$3$ vertex is exactly $2$. 
    We now contract the two neighbours of each vertex in~$L$, and then contract $\{x,y\}$ to $x$.
\end{proof}

\subsection{Addable sets and {\boldmath$3$}-contractible tree decompositions}\label{subsec:treedecomp}

\begin{figure}[t]
    \centering
    \tikzstyle{v}=[circle, draw, solid, fill=black, inner sep=0pt, minimum width=3pt]
    \begin{tikzpicture}
        \draw (0,1.8);
        \draw (0,-1.8);
        \draw (1*60:1.1) node[v,label={[xshift=0.0mm, yshift=0.0mm]$v_6$}](u1){};
        \draw (2*60:1.1) node[v,label={[xshift=0.0mm, yshift=0.0mm]$v_1$}](u2){};
        \draw (3*60:1.1) node[v,label={[xshift=-2.7mm, yshift=-3.0mm]$v_2$}](u3){};
        \draw (4*60:1.1) node[v,label={[xshift=0.0mm, yshift=-6.0mm]$v_3$}](u4){};
        \draw (5*60:1.1) node[v,label={[xshift=0.0mm, yshift=-6.0mm]$v_4$}](u5){};
        \draw (6*60:1.1) node[v,label={[xshift=2.7mm, yshift=-3.0mm]$v_5$}](u6){};
        \draw (u1)--(u3)--(u2);
        \draw (u1)--(u4)--(u2);
        \draw (u1)--(u5)--(u2);
        \draw (u1)--(u6)--(u2);
        \draw (u4)--(u3);
        \draw (u6)--(u5);
        \draw (u3)--(u5);
        \draw (u4)--(u6);
        \draw (0,-1.5) node[label=below:$K_6^\equiv$](){};
    \end{tikzpicture}
    \hspace{0.5cm}
    \begin{tikzpicture}
        \draw (0,1.8);
        \draw (0,-1.8);
        \draw (1*60:1.1) node[v,label={[xshift=0.0mm, yshift=0.0mm]$v_6$}](u1){};
        \draw (2*60:1.1) node[v,label={[xshift=0.0mm, yshift=0.0mm]$v_1$}](u2){};
        \draw (3*60:1.1) node[v,label={[xshift=-2.7mm, yshift=-3.0mm]$v_2$}](u3){};
        \draw (4*60:1.1) node[v,label={[xshift=0.0mm, yshift=-6.0mm]$v_3$}](u4){};
        \draw (5*60:1.1) node[v,label={[xshift=0.0mm, yshift=-6.0mm]$v_4$}](u5){};
        \draw (6*60:1.1) node[v,label={[xshift=2.7mm, yshift=-3.0mm]$v_5$}](u6){};
        \foreach \x in {3,4,5,6}{
            \draw (u1)--(u\x)--(u2);
        }
        \draw (u4)--(u3)--(u6)--(u5);
        \draw (u3)--(u5);
        \draw (u4)--(u6);
        \draw (0,-1.5) node[label=below:$K_6^{=}$](){};
    \end{tikzpicture}
    \hspace{0.5cm}
    \begin{tikzpicture}
        \draw (0,0) circle (1.414);
        \draw (0,0) node[v,label={[xshift=2.0mm, yshift=0.0mm]$v_7$}](o){};
        \draw (150:1.414) node[v,label={[xshift=-2.5mm, yshift=-0.5mm]$v_4$}](v1){};
        \draw (90:1) node[v,label={[xshift=2.5mm, yshift=-2mm]$v_1$}](w1){};
        \draw (270:1.414) node[v,label={[xshift=2.5mm, yshift=-0.5mm]$v_5$}](v2){};
        \draw (210:1) node[v,label={[xshift=-1.0mm, yshift=0.0mm]$v_2$}](w2){};
        \draw (30:1.414) node[v,label={[xshift=2.0mm, yshift=-0.5mm]$v_6$}](v3){};
        \draw (330:1) node[v,label={[xshift=0.0mm, yshift=-5.5mm]$v_3$}](w3){};
        \foreach \x in {1,2,3}{
            \draw (o)--(v\x)--(w\x);
	}
	\draw (w1)--(w2)--(w3)--(w1);
	\draw (w1)--(o)--(w2);
	\draw (o)--(w3);
        \draw (0,-1.5) node[label=below:$\overline{C_6}+K_1$](){};
    \end{tikzpicture}
    \hspace{0.5cm}
    \begin{tikzpicture}
        \draw (0,0) circle (1.414);
        \draw (0,0) circle (1.414);
        \draw (-1,1) node[v,label={[xshift=-1.8mm, yshift=-0.5mm]$v_1$}](a1){};
        \draw (0,1) node[v,label={[xshift=0.0mm, yshift=-0.7mm]$v_2$}](a2){};
        \draw (1,1) node[v,label={[xshift=1.8mm, yshift=-0.5mm]$v_3$}](a3){};
        \draw (-1,0) node[v,label={[xshift=-2.3mm, yshift=-3.0mm]$v_4$}](b1){};
        \draw (0,0) node[v,label={[xshift=2.0mm, yshift=-5.0mm]$v_5$}](b2){};
        \draw (1,0) node[v,label={[xshift=2.3mm, yshift=-3.0mm]$v_6$}](b3){};
	\draw (-1,-1) node[v,label={[xshift=-1.8mm, yshift=-5.5mm]$v_7$}](c1){};
        \draw (0,-1) node[v,label={[xshift=0.0mm, yshift=-5.3mm]$v_8$}](c2){};
        \draw (1,-1) node[v,label={[xshift=1.8mm, yshift=-5.5mm]$v_9$}](c3){};
        \draw (b1) arc (135:45:1.414);
        \draw (c2) arc (225:135:1.414);
	\draw (a1)--(a2)--(a3);
        \draw (b1)--(b2)--(b3);
        \draw (c1)--(c2)--(c3);
        \draw (a1)--(b1)--(c1);
        \draw (a2)--(b2)--(c2);
        \draw (a3)--(b3)--(c3);
        \draw (0,1.8);
        \draw (0,-1.8);
        \draw (0,-1.5) node[label=below:$L(K_{3,3})$](){};
    \end{tikzpicture}
    \caption{Graphs which can be induced by a bag of a $3$-contractible tree decomposition.}
    \label{fig:bagtypes}
\end{figure}
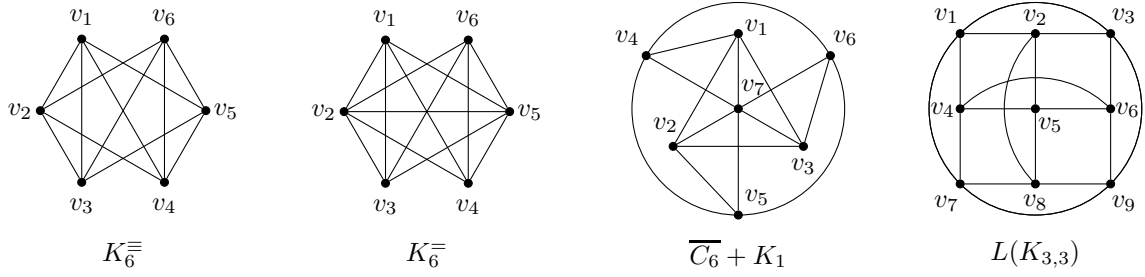

Let $K_6^\equiv$ be the graph obtained from~$K_6$ by removing a matching of size $3$, and $K_6^{=}$ be the graph obtained from~$K_6$ by removing a matching of size~$2$; see Figure~\ref{fig:bagtypes}.
Let $\mathcal{K}:=\{K_n:n\in[5]\}\cup\{K_6^\equiv,K_6^{=},{\overline{C_6}+K_1},L(K_{3,3})\}$.

Let $G$ be a multigraph, and let $\mathcal{A}$ be a set of $3$-vertex cliques in~$G$.
We introduce the following definitions.
Let $Y(G,\mathcal{A})$ be the multigraph obtained from~$G$ by adding a new degree-$3$ vertex $v_A$ with neighbourhood $A$ for each $A\in\mathcal{A}$.
For a set $X\subseteq V(G)$, $\mathcal{A}$ is \emph{addable} to $X$ in~$G$ if $Y(G[X],\{A\in\mathcal{A}:A\subseteq X\})$ is $\mathcal{F}_3$-minor-free.
An \emph{addable set} of~$G$ is a set of $3$-vertex cliques in~$G$ which is addable to $V(G)$ in~$G$.
A tree decomposition $(T,(B_t)_{t\in V(T)})$ of~$G$ is \emph{$3$-contractible} if it satisfies the following conditions.
\begin{enumerate}[label=$(\mathrm{C}\arabic*)$]
    \item\label{def:td contractible edge} For each $vw\in E(T)$, $B_v\cap B_w$ is a minimal separator of~$G$ and a clique of size at most~$3$ in~$G$.
    \item\label{def:td contractible vertex} For each $u\in V(T)$, the simplification of $G[B_u]$ is isomorphic to a graph in $\mathcal{K}$, and $\{B_u\cap B_{u'}:u'\in N_T(u),\abs{B_u\cap B_{u'}}=3\}$ is addable to~$B_u$ in~$G$.
\end{enumerate}

In proving Proposition~\ref{prop:mainbackward}, we will show that every edge-maximal $\mathcal{F}_3$-minor-free graph admits a $3$-contractible tree-decomposition, and that all graphs admitting $3$-contractible tree-decompositions are $\mathcal{F}_3$-minor-free.
To prove Proposition~\ref{prop:mainbackward}, we simultaneously show that $\mathcal{F}_3$-minor-free graphs are subgraphs of graphs which admit $3$-contractible tree-decompositions, and demonstrate that \subd{2}s of these graphs admit $3$-contraction sequences.

We summarise our strategy for finding these contraction sequences as follows.
We will contract the bags of the tree-decomposition one at a time, starting from the leaves.
When we contract a bag, we contract it down to the graph $\psi^*(\theta(G),\sigma_G(X))$, where~$X$ is the adhesion of the leaf bag to its parent in the tree decomposition.
Halfway through this process, the bag of a node whose children have all been contracted will have a subdivision of~$K_4$ attached to each of its adhesion sets.
Thus, for each $G\in\mathcal{K}$, each addable set~$\mathcal{A}$ of~$G$, and each $X\in A$, we require an algorithm for contracting $\psi(\theta(G),\sigma_G(\mathcal{A})$ down to $\psi^*(\theta(G),\sigma_G(X))$, with the added technicality that we must also keep track of the red edges we accumulate throughout the process.
In the following propositions, we present such an algorithm.
The general strategy will be to apply Lemmas~\ref{lem:yoperation}, \ref{lem:protection1}, and~\ref{lem:protection2} to simplify the graph, to the point where we can either apply Lemma~\ref{lem:routine} or explicitly complete the contraction sequence.
To simplify the case analysis, we will repeatedly use the following lemma.

\begin{lemma}\label{lem:repeat}
    Let $G$ be a graph, let $\mathcal{S}$ be a set of $3$-vertex cliques in $G$, let $\mathcal{A}$ be a subset of $\mathcal{S}$, let $X\in\mathcal{A}$, let $H_\mathcal{S}$ be a trigraph with underlying graph $\psi(\theta(G),\sigma_G(\mathcal{S}))$, and let $H$ be the induced subgraph of $H_\mathcal{S}$ with underlying graph $\psi(\theta(G),\sigma_G(\mathcal{A}))$.
    If there exists an $X$-stable partial \cont{3} from $H_\mathcal{S}$ to a trigraph with underlying graph $\psi^*(\theta(G),\sigma_G(X))$, then there exists an $X$-stable partial \cont{3} from $H$ to a trigraph with underlying graph $\psi^*(\theta(G),\sigma_G(X))$.
\end{lemma}
\begin{proof}
    In the $X$-stable partial \cont{3} from~$H_{\mathcal S}$ to a trigraph $H'$ with underlying graph $\psi^*(\theta(G),\sigma_G(X))$, we ignore contractions with vertices in $V(H_\mathcal{S})\setminus V(H)$ to obtain a partial \cont{3} from~$H$ to a trigraph $H''$ that is a refined subgraph of $H'$.
    Since the underlying graph of~$H$ contains $\psi^*(\theta(G),\sigma_G(X))$ as an induced subgraph, $H''$ contains all edges of~$H'$, and therefore the underlying graph of $H''$ is precisely $\psi^*(\theta(G),\sigma_G(X))$.
\end{proof}

\begin{figure}[t]
    \centering
    \tikzstyle{v}=[circle, draw, solid, fill=black, inner sep=0pt, minimum width=1.5pt]
    \tikzstyle{w}=[circle, draw, solid, fill=black, inner sep=0pt, minimum width=4pt]
    \tikzstyle{u}=[rectangle, draw, fill=black, inner sep=2pt, minimum width=0.5pt]
    \begin{tikzpicture}[scale=0.8]
        \draw (330:4) node[w,label=below:$v_1$](1){};
        \draw (90:4) node[w,label=above:$v_2$](2){};
        \draw (210:4) node[w,label=below:$v_3$](3){};
        \draw (0,0) node[w,label={[xshift=3mm,yshift=-5mm]$v_4$}](4){};
        \draw (150:1.3) node[v,label=left:$x_1$](x1){};
        \draw (270:1.3) node[v,label=below:$x_2$](x2){};
        \draw (30:1.3) node[v,label=right:$x_3$](x3){};

        \path (x1)--(4) node[u,pos=0.5](y14){};
        \draw (x1)--(y14);
        \draw (y14)--(4);        
        \path (x1)--(2) node[u,pos=0.5](y12){};
        \draw (x1)--(y12);
        \draw (y12)--(2);
        \path (x1)--(3) node[u,pos=0.5](y13){};
        \draw (x1)--(y13);
        \draw (y13)--(3);

        \path (x2)--(4) node[u,pos=0.5](y24){};
        \draw (x2)--(y24);
        \draw (y24)--(4);        
        \path (x2)--(3) node[u,pos=0.5](y23){};
        \draw (x2)--(y23);
        \draw (y23)--(3);
        \draw (x2)--(1);

        \path (x3)--(4) node[u,pos=0.5](y34){};
        \draw (x3)--(y34);
        \draw (y34)--(4);        
        \path (x3)--(2) node[u,pos=0.5](y32){};
        \draw (x3)--(y32);
        \draw (y32)--(2);
        \path (x3)--(1) node[u,pos=0.5](y31){};
        \draw (x3)--(y31);
        \draw (y31)--(1);

        \path (3)--(4) node[v,pos=0.35](){};
        \path (3)--(4) node[v,pos=0.65](){};
        \draw (3)--(4);
        \path (1)--(2) node[v,pos=0.35](){};
        \path (1)--(2) node[v,pos=0.65](){};
        \draw (1)--(2);
        \path (2)--(4) node[v,pos=0.35](){};
        \path (2)--(4) node[v,pos=0.65](){};
        \draw (2)--(4);
        \path (2)--(3) node[v,pos=0.35](){};
        \path (2)--(3) node[v,pos=0.65](){};
        \draw (2)--(3);
    \end{tikzpicture}
    \caption{The underlying graph $G^*_2$ in the proof of Proposition~\ref{prop:K_4 addable} where square vertices are $y_{i,j}$ vertices.}
    \label{fig:K_4 addable}
\end{figure}

\begin{proposition}\label{prop:K_4 addable}
    Let $G:=K_4$, let $\mathcal{A}$ be a set of~$3$-vertex cliques in~$G$, and let $H$ be a trigraph with underlying graph~$\psi(\theta(G),\sigma_G(\mathcal{A}))$ such that $\DeltaR(H)\leq3$.
    The following are equivalent.
    \begin{enumerate}[label=\rm(\alph*)]
        \item\label{cond:K_4 addable1} $\mathcal{A}$ is an addable set of $G$.
        \item\label{cond:K_4 addable2} $\abs{\mathcal{A}}\leq3$.
        \item\label{cond:K_4 addable3} For every $X\in\mathcal{A}$, there exists a trigraph $H'$ with underlying graph $\psi^*(\theta(G),\sigma_G(X))$ and an $X$-stable partial \cont{3} from~$H$ to $H'$.
    \end{enumerate}
\end{proposition}
\begin{proof}
    Let $v_1$, $v_2$, $v_3$, and $v_4$ be the vertices of~$G$.
    
    \ref{cond:K_4 addable1}$\Rightarrow$\ref{cond:K_4 addable2}:
    If $\abs{\mathcal{A}}=4$, then $Y(G,\mathcal{A})-E(G)$ is isomorphic to~$Q_3$.

    \ref{cond:K_4 addable2}$\Rightarrow$\ref{cond:K_4 addable3}:
    If $\mathcal{A}=\emptyset$, then the statement trivially holds, so we may assume that $\mathcal{A}$ is nonempty.
    We may assume that $\mathcal{A}\subseteq\mathcal{S}:=\{\{v_2,v_3,v_4\},\{v_1,v_3,v_4\},\{v_1,v_2,v_4\}\}$ and $X=\{v_2,v_3,v_4\}\in \mathcal{A}$.
    There exists a trigraph $H_\mathcal{S}$ with underlying graph $\psi(\theta(G),\sigma_G(\mathcal{S}))$ such that~$H$ is an induced subgraph of~$H_\mathcal{S}$ and $\DeltaR(H_\mathcal{S})=\DeltaR(H)\leq3$.

    By Lemma~\ref{lem:repeat}, it suffices to find a trigraph~$H'$ with underlying graph $\psi^*(\theta(G),\sigma_G(X))$ and an $X$-stable partial \cont{3} from~$H_\mathcal{S}$ to~$H'$.
    Let
    \[
        W:=\bigcup_{S\in\mathcal{S}}(V(C'_{G,S})\setminus S)\quad\text{ and }\quad Z:=V(\theta(G))\setminus W.
    \]
    In other words, $W$ is the set of all internal vertices of length-$4$ paths of~$\theta(G)$ between two vertices of~$G$.
    Our first step is to remove~$W$ while keeping all other vertices to have low red degree.
    For that, we use Lemma~\ref{lem:yoperation} to a $Z$-stable partial \cont{3} from~$H_\mathcal{S}$ to a trigraph~$H^*_1$ with underlying graph $G^*_1:=\psi(\theta(G),\sigma_G(\mathcal{S}))-W$.
    For distinct $i,j\in[4]$, let $w_{i,j}$ and $w_{j,i}$ be the internal vertices of the unique length-$3$ path in $H^*_1$ between~$v_i$ and~$v_j$ such that $v_iw_{i,j}\in E(H^*_1)$.
    For each $i\in [3]$ and $j\in [4]\setminus\{i\}$, let $S_i:=\sigma_G(\{v_1,v_2,v_3,v_4\}\setminus\{v_i\})$, let
    $x_i$ be the unique vertex of degree~$3$ in $\psi^*(\theta(G),S_i)-V(G)$, and let $y_{i,j}$ be the common neighbour of~$x_i$ and~$v_j$ in~$H^*_1$.

    Since $\DeltaR(H^*_1)\leq\DeltaR(H_\mathcal{S})\leq3$ and $\deg_{H^*_1}(v_1)=5$, by Lemma~\ref{lem:protection2} with $(t_1,t_2,t_3,t_4):=(v_3,v_4,v_1,x_2)$, there exists a $(V(H^*_1)\setminus V(\psi^*(\theta(G),S_2)))\cup\{v_3,v_4\}$-stable $\{v_1\}$-fixing partial \cont{3} from~$H^*_1$ to a trigraph whose underlying graph $G^*_2$ is the graph obtained from $G^*_1-\{w_{1,3},w_{3,1},w_{1,4},w_{4,1},y_{2,1}\}$ by adding an edge $v_1x_2$; see Figure~\ref{fig:K_4 addable}.
    We then contract the following ten pairs of vertices in order.
    \[
        \begin{array}{llll}
            1)\ \{w_{1,2},y_{3,1}\}\text{ to }y_{3,1}, 
            &2)\ \{w_{2,1},y_{3,2}\}\text{ to }y_{3,2},
            &3)\ \{x_3,y_{3,1}\}\text{ to }x_3,
            &4)\ \{y_{2,4},y_{3,4}\}\text{ to }y_{2,4},\\
            5)\ \{v_1,x_2\}\text{ to }v_1,
            &6)\ \{v_1,x_3\}\text{ to }v_1,
            &7)\ \{y_{1,2},y_{3,2}\}\text{ to }y_{1,2},
            &8)\ \{y_{1,3},y_{2,3}\}\text{ to }y_{1,3},\\
            9)\ \{y_{1,4},y_{2,4}\}\text{ to }y_{1,4},
            &10)\ \{v_1,x_1\}\text{ to }x_1.
        \end{array}
    \]
    Thus, we find an $X$-stable partial \cont{3} from~$H_\mathcal{S}$ to a trigraph~$H'$ with underlying graph $\psi^*(\theta(G),\sigma_G(X))$.

    \ref{cond:K_4 addable3}$\Rightarrow$\ref{cond:K_4 addable1}:
    If $\mathcal{A}$ is empty, then \ref{cond:K_4 addable1} holds because $K_4$ is $\mathcal{F}_3$-minor-free.
    Thus, we may assume that there is $X\in\mathcal{A}$.
    By~\ref{cond:K_4 addable3}, there exists a trigraph~$H'$ with underlying graph $\psi^*(\theta(G),\sigma_G(X))$ and an $X$-stable partial \cont{3} from~$H$ to~$H'$.
    Since the underlying graph of $H'$ is a subdivision of $K_4$, it is readily seen that $\tww(H')\leq3$, so that $\tww(H)\leq3$.
    Since the underlying graph of $H$ has an \subd{1} of~$Y(G,\mathcal{A})$ as an induced subgraph, by Proposition~\ref{prop:mainforward}, $Y(G,\mathcal{A})$ is $\mathcal{F}_3$-minor-free, and therefore~$\mathcal{A}$ is an addable set of~$G$.
\end{proof}

\begin{proposition}\label{prop:K_5 addable}
    Let $G:=K_5$, let $\mathcal{A}$ be a set of $3$-vertex cliques in~$G$, and let $H$ be a trigraph with underlying graph $\psi(\theta(G),\sigma_G(\mathcal{A}))$ such that $\DeltaR(H)\leq3$.
    The following are equivalent.
    \begin{enumerate}[label=\rm(\alph*)]
        \item\label{cond:K_5 addable1} $\mathcal{A}$ is an addable set of $G$.
        \item\label{cond:K_5 addable2} There exists $B\subseteq V(G)$ of size~$2$ such that for all $Q\in \mathcal A$, $\abs{Q\cap B}$ is even.
        \item\label{cond:K_5 addable3} For every $X\in\mathcal{A}$, there exists a trigraph $H'$ with underlying graph $\psi^*(\theta(G),\sigma_G(X))$ and an $X$-stable partial \cont{3} from~$H$ to $H'$.
    \end{enumerate}
\end{proposition}
\begin{proof}
    Let $v_1$, $v_2$, $v_3$, $v_4$, and $v_5$ be the vertices of~$G$.
    
    \ref{cond:K_5 addable1}$\Rightarrow$\ref{cond:K_5 addable2}:
    If $\abs{\mathcal{A}}\leq2$, then we can easily find the set $B$, so we may assume that $\abs{\mathcal A}\ge 3$.
    Note that $\mathcal{A}$ contains two cliques $A_1$ and $A_2$ with $\abs{A_1\cap A_2}=2$.
    We claim that $B:=A_1\cap A_2$ has the desired property.
    Suppose for contradiction that there is $Q\in\mathcal A$ such that $\abs{Q\cap B}=1$.
    By relabelling, we may assume that $A_1=\{v_1,v_2,v_3\}$, $A_2=\{v_1,v_2,v_4\}$, and
    $Q\in\{\{v_1,v_3,v_4\},\{v_1,v_3,v_5\}\}$.

    If $Q=\{v_1,v_3,v_4\}$, then $Y(G,\mathcal A)$ has two disjoint cycles with vertex sets $\{v_Q,v_3,v_5,v_4\}$ and $\{v_1,v_{A_1},v_2,v_{A_2}\}$.
    Since $Y(G,\mathcal{A})$ has edges $v_Qv_1$, $v_3v_{A_1}$, $v_5v_2$, and $v_4v_{A_2}$, it has a subgraph isomorphic to~$Q_3$, contradicting~\ref{cond:K_5 addable1}.
    
    Hence, we may assume that $Q=\{v_1,v_3,v_5\}$. Now $Y(G,\mathcal{A})$ has two disjoint cycles with vertex sets $\{v_1,v_Q,v_3,v_{A_1}\}$ and $\{v_{A_2},v_4,v_5,v_2\}$.
    Since $Y(G,\mathcal{A})$ has edges $v_1v_{A_2}$, $v_Qv_5$, $v_3v_4$, and~$v_{A_1}v_2$, it has a subgraph isomorphic to~$V_8$, contradicting~\ref{cond:K_5 addable1}.

    \ref{cond:K_5 addable2}$\Rightarrow$\ref{cond:K_5 addable3}:
    If $\mathcal{A}=\emptyset$, then the statement trivially holds, so we may assume that $\mathcal{A}$ is nonempty.
    Suppose without loss of generality that $B=\{v_1,v_2\}$.
    Note that
    \[
        \mathcal{A}\subseteq\mathcal{S}:=\{\{v_1,v_2,v_3\},\{v_1,v_2,v_4\},\{v_1,v_2,v_5\},\{v_3,v_4,v_5\}\}.
    \]
    There exists a trigraph $H_\mathcal{S}$ with underlying graph $\psi(\theta(G),\sigma_G(\mathcal{S}))$ such that~$H$ is an induced subgraph of~$H_\mathcal{S}$ and $\DeltaR(H_\mathcal{S})=\DeltaR(H)\leq3$.
    
    By Lemma~\ref{lem:repeat}, it suffices to find a trigraph~$H'$ with underlying graph $\psi^*(\theta(G),\sigma_G(X))$ and an $X$-stable partial \cont{3} from~$H_\mathcal{S}$ to~$H'$.
    Let
    \[
        W:=\bigcup_{S\in\mathcal{S}}(V(C'_{G,S})\setminus S)\quad\text{ and }\quad Z:=V(\theta(G))\setminus W.
    \]
    In other words, $W$ is the set of all internal vertices of length-$4$ paths of~$\theta(G)$ between two vertices of~$G$.
    We first apply Lemma~\ref{lem:yoperation} to obtain a $Z$-stable partial \cont{3} from~$H_\mathcal{S}$ to a trigraph~$H^*_1$ with underlying graph $G^*_1:=\psi(\theta(G),\sigma_G(\mathcal{S}))-W$.
    For distinct $i,j\in[5]$, let~$w_{i,j}$ and~$w_{j,i}$ be the internal vertices of the unique length-$3$ path in $H^*_1$ between~$v_i$ and~$v_j$ such that $v_iw_{i,j}\in E(H^*_1)$.
    Let $S_0$ be the \sena{} $\sigma_G(\{v_3,v_4,v_5\})$ of $\theta(G)$ and 
    let $S_i$ be the \sena{} $\sigma_G(\{v_1,v_2,v_i\})$ of $\theta(G)$ for each $i\in\{3,4,5\}$.
    For each $i\in\{0,3,4,5\}$, let $x_i$ be the unique vertex of degree~$3$ in $\psi^*(\theta(G),S_i)-V(G)$. 
    For each $i\in\{3,4,5\}$ and each $j\in \{1,2,i\}$, let~$y_{0,i}$ be the common neighbour of~$x_0$ and~$v_i$ in~$H^*_1$, and let $y_{i,j}$ be the common neighbour of $x_i$ and $v_j$ in~$H^*_1$.
    For each $i\in[2]$ and $j\in\{3,4,5\}$, let $q(i,j):=\abs{\{v_iw_{i,j},v_iy_{j,i}\}\cap R(H^*_1)}$.
    Since $\DeltaR(H^*_1)\leq\DeltaR(H_{\mathcal S})\leq3$, 
    we have that for each $i\in[2]$, 
    $q(i,3)+q(i,4)+q(i,5)\le 3$ 
    and therefore there is at most one $j\in \{3,4,5\}$ such that $q(i,j)=2$.
    
    Suppose that $X=\{v_3,v_4,v_5\}$.
    By symmetry, we may assume that $\max\{q(1,4),q(1,5),q(2,5)\}\leq1$.
    Let~$H^*_2$ be the trigraph obtained from~$H^*_1$ by contracting $\{w_{2,4},y_{4,2}\}$ to $y_{4,2}$, and let $G^*_2$ be the underlying graph of~$H^*_2$.
    Let $F$ be the subgraph of $H_{\mathcal{S}}$ with underlying graph $\psi^*(\theta(G),S_3)$.
    Note that
    \[
        \DeltaR(H^*_2)\leq3,\quad 
        \rdeg_{H^*_2-E(F)}(v_1)=q(1,4)+q(1,5)\le 2,\ \text{ and}\ 
        \rdeg_{H^*_2-E(F)}(v_2)\le q(2,5)+1\le 2.
    \]
    Applying Corollary~\ref{cor:protection1} to $F$ with $(t_1,t_2,t_3,t_4):=(v_1,v_2,v_3,x_3)$, we find a $(V(H^*_2)\setminus V(F))\cup\{v_3\}$-stable $\{v_1,v_2\}$-fixing partial \cont{3} from~$H^*_2$ to a trigraph $H^*_3$ whose underlying graph is the graph obtained from $G^*_2-\{w_{1,3},w_{3,1},w_{2,3},w_{3,2},y_{3,1},y_{3,2}\}$ by adding edges $v_1x_3$ and $v_2x_3$.
    Let $H_4^*$ be the trigraph obtained from $H_3^*$ by applying the following sequence of contractions.
    \begin{enumerate}[label=\bf{Step \arabic*.},leftmargin=*]
        \item Contract the following five pairs of vertices in order: $\{w_{4,2},y_{4,4}\}$ to~$y_{4,4}$, $\{x_4,y_{4,2}\}$ to~$x_4$, $\{w_{1,4},y_{4,1}\}$ to~$y_{4,1}$, $\{w_{4,1},y_{4,4}\}$ to $y_{4,4}$, and $\{x_4,y_{4,1}\}$ to~$x_4$.
        \item For each $i\in[2]$, contract the following three pairs of vertices in order: $\{w_{i,5},y_{5,i}\}$ to~$y_{5,i}$, $\{w_{5,i},y_{5,5}\}$ to~$y_{5,5}$, $\{x_5,y_{5,i}\}$ to~$x_5$.
    \end{enumerate}
    Since $\max\{q(1,4),q(1,5),q(2,5)\}\leq1$, 
    this partial contraction sequence from $H_3^*$ to $H_4^*$ is an $X$-stable partial \cont{3}.
    Let $G^*_4$ be the underlying graph of $H^*_4$ and let $H_0:=G^*_4[N^2_{G^*_4}[x_0]]$, 
    which is isomorphic to the $1$-subdivision of $K_{1,3}$. 
    Note that 
    $G_0:=G^*_4-\{w_{3,4},w_{4,3},w_{3,5},w_{5,3},w_{4,5},w_{5,4}\}$ is a subdivision of $K_{3,3}$ and that in $G_0-V(H_0)$, 
    vertices $v_1$, $v_2$ and their neighbours induce a subgraph isomorphic to $K_{2,3}$.
    By Lemma~\ref{lem:routine} 
    applied to $H_4^*-\{w_{3,4},w_{4,3},w_{3,5},w_{5,3},w_{4,5},w_{5,4}\}$ with $H_0$, there exists an $X$-stable partial \cont{3} from~$H^*_4$ to a trigraph $H'$ with underlying graph $\psi^*(\theta(G),\sigma_G(X))$.
    Thus, we find an $X$-stable partial \cont{3} from~$H_\mathcal{S}$ to~$H'$.   

    Hence, we may assume that $X\neq\{v_3,v_4,v_5\}$.
    By symmetry, we may assume that $X=\{v_1,v_2,v_3\}$ and that $v_4w_{4,3}\in R(H^*_1)$ or $v_5w_{5,3}\in B(H^*_1)$.

    The following two contractions are the first two contractions 
    given by Lemma~\ref{lem:protection1} applied with $(t_1,t_2,t_3,t_4):=(v_4,v_5,v_3,x_0)$.
    Let $t'_1:=w_{4,5}$ if $v_4w_{4,3}\in B(H^*_1)$, and otherwise let $t'_1:=w_{4,3}$.
    Let~$H^*_2$ be the trigraph obtained from $H^*_1$ by contracting $\{t'_1,y_{0,4}\}$ to~$y_{0,4}$, let $H^*_3$ be the trigraph obtained from $H^*_2$ by contracting $\{w_{5,4},y_{0,5}\}$ to~$y_{0,5}$, and let $G^*_3$ be the underlying graph of~$H^*_3$.
    Note that $H^*_1,H^*_2,H^*_3$ is an $X$-stable partial \cont{3}, and $\deg_{H^*_3}(v_4)=\deg_{H^*_3}(v_5)=5$.

    Thus, for each $i\in \{4,5\}$, $v_i$ has at most two incident red edges in $H_3^*-E(\psi^*(\theta(G),S_i))$.
    For an integer $i$ from~$4$ to~$5$, by Lemma~\ref{lem:protection2} with $(t_1,t_2,t_3,t_4):=(v_1,v_2,v_i,x_i)$, there exists a $(V(H^*_{i-1})\setminus V(\psi^*(\theta(G),S_i)))\cup \{v_1,v_2\}$-stable $\{v_i\}$-fixing partial \cont{3} from~$H^*_{i-1}$ to a trigraph~$H^*_i$ whose underlying graph $G^*_i$ is the graph obtained from $G^*_{i-1}-\{w_{1,i},w_{i,1},w_{2,i},w_{i,2},y_{i,i}\}$ by adding an edge~$v_ix_i$.
    Note that $\deg_{H^*_5}(v_4)=\deg_{H^*_5}(v_5)=3$.
    We now apply all but the first contractions given by Lemma~\ref{lem:protection1} applied with $(t_1,t_2,t_3,t_4):=(v_4,v_5,v_3,x_0)$
    to obtain a $(V(H^*_5)\setminus V(\psi^*(\theta(G),S_0)))\cup \{v_3\}$-stable $\{v_4,v_5\}$-fixing partial \cont{3} from~$H^*_5$ to a trigraph~$H^*_6$ whose underlying graph~$G^*_6$ is the graph obtained from
    \[
        G^*_5-\{w_{3,4},w_{4,3},w_{3,5},w_{5,3},w_{4,5},y_{0,4},y_{0,5}\}
    \]
    by adding edges $v_4x_0$ and $v_5x_0$.
    Recall that we have already performed the first two contractions given by Lemma~\ref{lem:protection1}.
    Since $G^*_6-\{w_{1,2},w_{2,1},w_{1,3},w_{3,1},w_{2,3},w_{3,2}\}$ is an \subd{1} of $K_{3,3}$, 
    we can apply Lemma~\ref{lem:routine}
    to $H^*_6-\{w_{1,2},w_{2,1},w_{1,3},w_{3,1},w_{2,3},w_{3,2}\}$ with $H_0:=G^*_6[N_{G^*_6}^2[x_3]]$
    and obtain an $X$-stable partial \cont{3} from~$H^*_6$ to a trigraph~$H'$ with underlying graph $\psi^*(\theta(G),\sigma_G(X))$.
    Thus, we find an $X$-stable partial \cont{3} from~$H_\mathcal{S}$ to~$H'$.
    
    \ref{cond:K_5 addable3}$\Rightarrow$\ref{cond:K_5 addable1}:
    If $\mathcal{A}$ is empty, then \ref{cond:K_4 addable1} holds because $K_5$ is $\mathcal{F}_3$-minor-free.
    Thus, we may assume that there is $X\in\mathcal{A}$.
    By \ref{cond:K_5 addable3}, there exists a trigraph $H'$ with underlying graph $\psi^*(\theta(G),\sigma_G(X))$ and an $X$-stable partial \cont{3} from~$H$ to $H'$.
    Since the underlying graph of $H'$ is a subdivision of $K_4$, $\tww(H')\leq3$, so that $\tww(H)\leq3$.
    Since the underlying graph of $H$ has an \subd{1} of~$Y(G,\mathcal{A})$ as an induced subgraph, by Proposition~\ref{prop:mainforward}, $Y(G,\mathcal{A})$ is $\mathcal{F}_3$-minor-free, and therefore~$\mathcal{A}$ is an addable set of~$G$.
\end{proof}

\begin{figure}[t]
    \centering
    \tikzstyle{v}=[circle, draw, solid, fill=black, inner sep=0pt, minimum width=1.5pt]
    \tikzstyle{w}=[circle, draw, solid, fill=black, inner sep=0pt, minimum width=4pt]
    \tikzstyle{u}=[rectangle, draw, fill=black, inner sep=2pt, minimum width=0.5pt]
    \begin{tikzpicture}[scale=0.8]
        \draw (120+60*0:4) node[w,label={[xshift=0.0mm,yshift=0.0mm]$v_1$}](1){};
        \draw (120+60*1:4) node[w,label={[xshift=-2.0mm,yshift=0.0mm]$v_2$}](2){};
        \draw (120+60*2:4) node[w,label={[xshift=0.0mm,yshift=-6.0mm]$v_3$}](3){};
        \draw (120+60*3:4) node[w,label={[xshift=0.0mm,yshift=-6.0mm]$v_4$}](4){};
        \draw (120+60*4:4) node[w,label={[xshift=0.5mm,yshift=0.0mm]$v_5$}](5){};
        \draw (120+60*5:4) node[w,label={[xshift=0.0mm,yshift=0.0mm]$v_6$}](6){};
        \draw (2)+(1,-0.8) node[w,label={[xshift=3.2mm,yshift=-2.0mm]$x_0$}](123){};
        \draw (2)+(3,0) node[w,label={[xshift=4.2mm,yshift=-3.0mm]$x_2$}](246){};
        \draw (6)+(0,-1.5) node[w,label={[xshift=3.2mm,yshift=-2.0mm]$x_3$}](356){};
        \draw (4)+(0,1.5) node[w,label={[xshift=3.2mm,yshift=-3.5mm]$x_1$}](145){};

        \path (1)--(2) node[v,pos=0.35](w12){};
        \path (1)--(2) node[v,pos=0.65](w21){};
        \path (1)--(3) node[v,pos=0.35](w13){};
        \path (1)--(3) node[v,pos=0.65](w31){};
        \path (2)--(3) node[v,pos=0.35](w23){};
        \path (2)--(3) node[v,pos=0.65](w32){};
        \draw (1)--(2)--(3)--(1);

        \path (123)--(1) node[u,pos=0.5](x1){};
        \draw (123)--(x1);
        \draw (x1)--(1);
        \path (123)--(2) node[u,pos=0.5](x2){};
        \draw (123)--(x2);
        \draw (x2)--(2);
        \path (123)--(3) node[u,pos=0.5](x3){};
        \draw (123)--(x3);
        \draw (x3)--(3);

        \path (246)--(2) node[u,pos=0.5](x4){};
        \draw (246)--(x4);
        \draw (x4)--(2);
        \draw (246)--(4);
        \draw (246)--(6);
        
        \path (356)--(3) node[u,pos=0.7](x7){};
        \draw (356)--(x7);
        \draw (x7)--(3);
        \draw (356)--(5);
        \draw (356)--(6);

        \path (145)--(1) node[u,pos=0.7](x10){};
        \draw (145)--(x10);
        \draw (x10)--(1);
        \draw (145)--(4);
        \draw (145)--(5);

        \path (2)--(5) node[pos=0.3](y1){};
        \path (2)--(5) node[pos=0.5](y2){};
        \path (2)--(5) node[pos=0.7](y3){};
        \draw (y1)+(0,0.5) node[v,label={[xshift=0.0mm,yshift=-5.0mm]$z_1$}](z1){};
        \draw (y2)+(0,0.5) node[v,label={[xshift=0.0mm,yshift=-5.0mm]$z_2$}](z2){};
        \draw (y3)+(0,0.5) node[v,label={[xshift=0.0mm,yshift=-5.0mm]$z_3$}](z3){};
        \draw (2)--(z1);
        \draw (z1)--(z2);
        \draw (z2)--(z3);
        \draw (z3)--(5);

        \path (2)--(5) node[pos=0.35](z4){};
        \path (2)--(5) node[pos=0.65](z5){};
        \draw (z4)+(0,-0.5) node[v,label={[xshift=0.0mm,yshift=-5.0mm]$w_{2,5}$}](y4){};
        \draw (z5)+(0,-0.5) node[v,label={[xshift=3.0mm,yshift=-5.0mm]$w_{5,2}$}](y5){};
        \draw (2)--(y4)--(y5)--(5);        
    \end{tikzpicture}
    \caption{The underlying graph $G^*_6$ in the proof of Proposition~\ref{prop:K_6 addable} where square vertices are~$y_{i,j}$ vertices.}
    \label{fig:K_6 addable}
\end{figure}
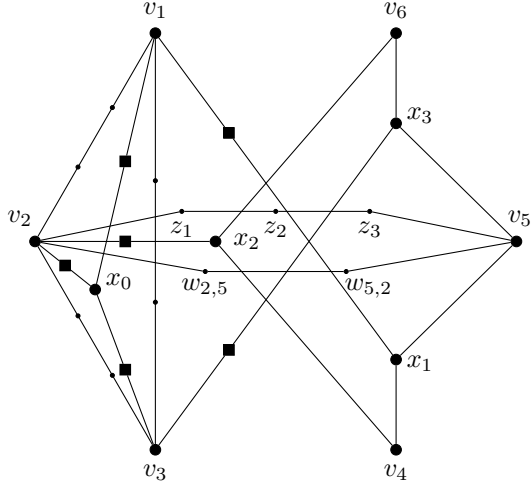

\begin{proposition}\label{prop:K_6 addable}
    Let $G\in\{K_6^\equiv,K_6^{=}\}$, let $\mathcal{A}$ be a set of $3$-vertex cliques in~$G$, and let $H$ be a trigraph with underlying graph $\psi(\theta(G),\sigma_G(\mathcal{A}))$ such that $\DeltaR(H)\leq3$.
    The following are equivalent.
    \begin{enumerate}[label=\rm(\alph*)]
        \item\label{cond:K_6 addable1} $\mathcal{A}$ is an addable set of $G$.
        \item\label{cond:K_6 addable2} There exists $B\subseteq V(G)$ such that both $B$ and $V(G)\setminus B$ are cliques of size $3$ in $G$ and for all $Q\in\mathcal{A}$, $\abs{Q\cap B}$ is even and $Q$ contains at most one degree-$5$ vertex of $G$.
        \item\label{cond:K_6 addable3} For every $X\in\mathcal{A}$, there exists a trigraph $H'$ with underlying graph $\psi^*(\theta(G),\sigma_G(X))$ and an $X$-stable partial \cont{3} from~$H$ to $H'$.
    \end{enumerate}
\end{proposition}
\begin{proof}
    Assume that the vertices of~$G$ are labelled as in Figure~\ref{fig:bagtypes}.

    \ref{cond:K_6 addable1}$\Rightarrow$\ref{cond:K_6 addable2}:
    If $\mathcal{A}$ contains a clique~$T$ with $\{v_2,v_5\}\subseteq T$, then it is easy to check that $Y(G,\{T\})$ has a subgraph isomorphic to $K_{3,\widehat{1},3}$, where $T$ corresponds to one of the independent sets of size~$3$, and so $\mathcal{A}$ is not an addable set of $G$.
    Thus, $\mathcal A$ has no cliques containing both~$v_2$ and~$v_5$, and therefore each clique in $\mathcal A$ is also a clique in $K_6^\equiv$.
    Since $K_6^\equiv=K_6^{=}-\{v_2v_5\}$ and $Y(K_6^\equiv,\mathcal A)$ is a subgraph of $Y(K_6^=,\mathcal A)$, it is enough to prove this for the case that $G=K_6^\equiv$.

    We may assume that $\mathcal A\neq\emptyset$.
    Let $X\in\mathcal A$.
    By symmetry between $v_i$ and $v_{7-i}$ for each $i\in\{1,2,3\}$, we may assume that $X=\{v_1,v_2,v_3\}$.
    We claim that $B:=\{v_4,v_5,v_6\}$ satisfies the required properties.
    Suppose for contradiction that~$\mathcal{A}$ contains $Q$ such that $\abs{Q\cap B}$ is odd.
    If $Q=B$, then $Y(G,\mathcal A)$ has two disjoint cycles with vertex sets $\{v_1,v_X,v_3,v_5\}$ and $\{v_4,v_2,v_6,v_Q\}$.
    Since $Y(G,\mathcal{A})$ has edges $v_1v_4$, $v_Xv_2$, $v_3v_6$, and $v_5v_Q$, it has a subgraph isomorphic to~$Q_3$, contradicting~\ref{cond:K_6 addable1}.
    
    Hence, $\abs{Q\cap B}=1$.
    By symmetry in $\{v_4,v_5,v_6\}$, we may assume that $v_5\in Q$, and therefore $Q=\{v_1,v_3,v_5\}$.
    Then $Y(G,\mathcal A)$ has two disjoint cycles with vertex sets $\{v_1,v_X,v_3,v_Q\}$ and $\{v_4,v_2,v_6,v_5\}$.
    Since $Y(G,\mathcal{A})$ has edges $v_1v_4$, $v_Xv_2$, $v_3v_6$, and $v_Qv_5$, it has a subgraph isomorphic to~$Q_3$, contradicting~\ref{cond:K_6 addable1}.
    
    \ref{cond:K_6 addable2}$\Rightarrow$\ref{cond:K_6 addable3}:
    If $\mathcal{A}=\emptyset$, then the statement trivially holds, so we assume that $\mathcal{A}$ is nonempty.
    Let $X\in \mathcal A$.
    Since $X$ contains exactly one of $v_i$ and $v_{7-i}$ for each $i\in\{1,2,3\}$ by \ref{cond:K_6 addable2}, we may assume without loss of generality that $X=\{v_1,v_2,v_3\}$.
    Then $B$ is equal to one of $\{v_4,v_5,v_6\}$, $\{v_1,v_2,v_4\}$, 
    $\{v_1,v_3,v_5\}$, and $\{v_2,v_3,v_6\}$.
    In all cases of~$B$, we deduce that 
    \[
        \mathcal{A}\subseteq\mathcal{S}:=\{\{v_1,v_2,v_3\},\{v_2,v_4,v_6\},\{v_3,v_5,v_6\},\{v_1,v_4,v_5\}\}.
    \]
    Let $\Gamma:=K_6^{=}$.
    There exists a trigraph $H_\mathcal{S}$ with underlying graph $\psi(\theta(\Gamma),\sigma_{\Gamma}(\mathcal{S}))$ such that~$H$ is an induced subgraph of $H_\mathcal{S}$ and $\DeltaR(H_\mathcal{S})=\DeltaR(H)\leq3$.

    By Lemma~\ref{lem:repeat}, it suffices to find a trigraph~$H'$ with underlying graph $\psi^*(\theta(\Gamma),\sigma_G(X))$ and an $X$-stable partial \cont{3} from~$H_\mathcal{S}$ to~$H'$.
    Let
    \[
        W:=\bigcup_{S\in\mathcal{S}}(V(C'_{\Gamma,S})\setminus S)\quad\text{ and }\quad Z:=V(\theta(\Gamma))\setminus W.
    \]
    In other words, $W$ is the set of all internal vertices of length-$4$ paths of~$\theta(\Gamma)$ between two vertices of~$\Gamma$, except for the length-$4$ path of $\theta(\Gamma)$ between~$v_2$ and~$v_5$.
    We first apply Lemma~\ref{lem:yoperation} to obtain a $Z$-stable partial \cont{3} from~$H_\mathcal{S}$ to a trigraph~$H^*_1$ with underlying graph $G^*_1:=\psi(\theta(\Gamma),\sigma_\Gamma(\mathcal{S}))-W$.
    For each edge $v_iv_j$ of $\Gamma$, let $w_{i,j}$ and $w_{j,i}$ be the internal vertices of the unique length-$3$ path in $H^*_1$ between $v_i$ and $v_j$ such that $v_iw_{i,j}\in E(H^*_1)$.
    Let $z_1$, $z_2$, and $z_3$ be the internal vertices of the unique length-$4$ path in $H^*_1$ between~$v_2$ and~$v_5$ such that $\{v_2z_1,v_5z_3\}\subseteq E(H^*_1)$.
    Let $x_0$ be the unique vertex of degree~$3$ in $\psi^*(\theta(\Gamma),\sigma_\Gamma(X))-V(\Gamma)$.
    For each $i\in[3]$, let $y_{0,i}$ be the common neighbour of~$x_0$ and~$v_i$ in~$H^*_1$, let $S_i:=\sigma_\Gamma(\{v_i,v_4,v_5,v_6\}\setminus\{v_{7-i}\})$, let $x_i$ be the unique vertex of degree~$3$ in $\psi^*(\theta(\Gamma),S_i)-V(\Gamma)$, and for each $j\in\{i,4,5,6\}\setminus\{7-i\}$, let~$y_{i,j}$ be the common neighbour of $x_i$ and~$v_j$ in~$H^*_1$.

    By swapping $v_1$ and $v_3$ and swapping $v_4$ and $v_6$ if necessary, we may assume that $v_4w_{4,2}\in R(H^*_1)$ or $v_6w_{6,2}\in B(H^*_1)$.
    The following two contractions are the first and the second contractions given by Lemma~\ref{lem:protection1} with $(t_1,t_2,t_3,t_4):=(v_4,v_6,v_2,x_2)$.
    Let $t'_1:=w_{4,6}$ if $v_4w_{4,2}\in B(H^*_1)$, and otherwise let $t'_1:=w_{4,2}$.
    Let $H^*_2$ be the trigraph obtained from $H^*_1$ by contracting $\{t'_1,y_{2,4}\}$ to~$y_{2,4}$, let~$H^*_3$ be the trigraph obtained from $H^*_2$ by contracting $\{w_{6,4},y_{2,6}\}$ to~$y_{2,6}$, and let $G^*_3$ be the underlying graph of~$H^*_3$.
    Note that $H^*_1,H^*_2,H^*_3$ is an $X$-stable partial \cont{3}, and $\deg_{H^*_3}(v_4)=\deg_{H^*_3}(v_6)=5$.
    
    For each $i\in\{1,3\}$, let $F_i$ be the subgraph of $H_\mathcal{S}$ with underlying graph $\psi^*(\theta(\Gamma),S_i)$.
    Let $\ell_1$ be an integer in $\{1,3\}$ such that $\rdeg_{F_{\ell_1}}(v_5)\geq\rdeg_{F_{4-{\ell_1}}}(v_5)$, and let $\ell_2:=4-\ell_1$.
    Note that if $\rdeg_{F_{\ell_1}}(v_5)=0$, then $\rdeg_{H^*_3}(v_5)\leq2$, because $w_{5,2}$ and $z_3$ are the only neighbours of $v_5$ in $H^*_3$ which are not in $V(F_1)\cup V(F_3)$.
    Therefore, $v_5$ has at most two incident red edges in $H^*_3-E(F_{\ell_1})$.
    For $i$ from $1$ to~$2$, by applying Corollary~\ref{cor:protection1} to $F:=F_{\ell_i}$ with $(t_1,t_2,t_3,t_4):=(v_{\ell_i+3},v_5,v_{\ell_i},x_{\ell_i})$, we find a $(V(H^*_{i+2})\setminus V(F_{\ell_i}))\cup\{v_{\ell_i}\}$-stable $(\{v_4,v_5,v_6\}\setminus\{v_{7-\ell_i}\})$-fixing partial \cont{3} from~$H^*_{i+2}$ to a trigraph~$H^*_{i+3}$ whose underlying graph $G^*_{i+3}$ is the graph obtained from
    \[
        G^*_{i+2}-\{w_{\ell_i,\ell_i+3},w_{\ell_i+3,\ell_i},w_{\ell_i,5},w_{5,\ell_i},w_{\ell_i+3,5},w_{5,\ell_i+3},y_{\ell_i,\ell_i+3},y_{\ell_i,5}\}
    \]
    by adding edges $v_{\ell_i+3}x_{\ell_i}$ and $v_5x_{\ell_i}$.
    
    We now apply Lemma~\ref{lem:protection1} with $(t_1,t_2,t_3,t_4):=(v_4,v_6,v_2,x_2)$.
    Note that we have already performed the first two contractions of the partial \cont{3} given by Lemma~\ref{lem:protection1}.
    By performing the remaining contractions, we find a $(V(H^*_5)\setminus V(F_2))\cup\{v_2\}$-stable $\{v_4,v_6\}$-fixing partial \cont{3} from $H^*_5$ to a trigraph~$H^*_6$ whose underlying graph $G^*_6$ is the graph obtained from
    \[
        G^*_5-\{w_{2,4},w_{4,2},w_{2,6},w_{6,2},w_{4,6},y_{2,4},y_{2,6}\}
    \]
    by adding edges $v_4x_2$ and $v_6x_2$; see Figure~\ref{fig:K_6 addable}.
    We then contract the following fourteen pairs of vertices in order.
    \[
        \begin{array}{llll}
            1)\ \{v_4,x_2\}\text{ to }x_2,
            &2)\ \{v_6,x_2\}\text{ to }x_2,
            &3)\ \{w_{2,5},z_1\}\text{ to }w_{2,5},
            &4)\ \{w_{5,2},z_3\}\text{ to }w_{5,2},\\
            5)\ \{w_{5,2},z_2\}\text{ to }w_{5,2},
            &6)\ \{w_{5,2},v_5\}\text{ to }v_5,
            &7)\ \{w_{2,5},y_{2,2}\}\text{ to }y_{2,2},
            &8)\ \{x_2,v_5\}\text{ to }v_5,\\
            9)\ \{v_5,x_1\}\text{ to }v_5,
            &10)\ \{v_5,x_3\}\text{ to }v_5,
            &11)\ \{y_{0,1},y_{1,1}\}\text{ to }y_{0,1},
            &12)\ \{y_{0,2},y_{2,2}\}\text{ to }y_{0,2},\\
            13)\ \{y_{0,3},y_{3,3}\}\text{ to }y_{0,3},
            &14)\ \{x_0,v_5\}\text{ to }x_0.
        \end{array}
    \]
    Thus, we find an $X$-stable partial \cont{3} from~$H_\mathcal{S}$ to a trigraph~$H'$ with underlying graph $\psi^*(\theta(G),\sigma_G(X))$.

    \ref{cond:K_6 addable3}$\Rightarrow$\ref{cond:K_6 addable1}:
    If $\mathcal{A}$ is empty, then \ref{cond:K_4 addable1} holds because $K_6^{=}$ is $\mathcal{F}_3$-minor-free.
    Thus, we may assume that there is $X\in\mathcal{A}$.
    By~\ref{cond:K_6 addable3}, there exists a trigraph~$H'$ with underlying graph $\psi^*(\theta(G),\sigma_G(X))$ and an $X$-stable partial \cont{3} from~$H$ to~$H'$.
    Since the underlying graph of $H'$ is a subdivision of $K_4$, $\tww(H')\leq3$, so that $\tww(H)\leq3$.
    Since the underlying graph of $H$ has an \subd{1} of~$Y(G,\mathcal{A})$ as an induced subgraph, by Proposition~\ref{prop:mainforward}, $Y(G,\mathcal{A})$ is $\mathcal{F}_3$-minor-free, and therefore~$\mathcal{A}$ is an addable set of~$G$.
\end{proof}

\begin{proposition}\label{prop:spinner addable}
    Let $G:=\overline{C_6}+K_1$, let $\mathcal{A}$ be a set of $3$-vertex cliques in~$G$, and let $H$ be a trigraph with underlying graph $\psi(\theta(G),\sigma_G(\mathcal{A}))$ such that $\DeltaR(H)\leq3$.    
    The following are equivalent.
    \begin{enumerate}[label=\rm(\alph*)]
        \item\label{cond:spinner addable1} $\mathcal{A}$ is an addable set of $G$.
        \item\label{cond:spinner addable2} $\abs{Q\cap B}$ is even
        for all $Q\in \mathcal{A}$
        and all $B\subseteq V(G)$ for which $G[B]$ is a cycle of length $4$.
        \item\label{cond:spinner addable3} For every $X\in\mathcal{A}$, there exists a trigraph $H'$ with underlying graph $\psi^*(\theta(G),\sigma_G(X))$ and an $X$-stable partial \cont{3} from~$H$ to $H'$.
    \end{enumerate}
\end{proposition}
\begin{proof}
    Assume that the vertices of~$G$ are labelled as in Figure~\ref{fig:bagtypes}.

    \ref{cond:spinner addable1}$\Rightarrow$\ref{cond:spinner addable2}:
    Suppose for contradiction that $\abs{Q\cap B}$ is odd for some $Q\in \mathcal A$ and some $B\subseteq V(G)$ such that that $G[B]$ is a cycle of length $4$.
    Since no subset of $B$ is a clique of size $3$ in $G$, we have $\abs{Q\cap B}=1$.
    By symmetry, we may assume that $B=\{v_1,v_4,v_6,v_3\}$ and that $Q=\{v_1,v_2,v_7\}$.
    Then $Y(G,\mathcal A)$ has two disjoint cycles with vertex sets $\{v_1,v_Q,v_2,v_3\}$ and $\{v_4,v_7,v_5,v_6\}$.
    Since $Y(G,\mathcal{A})$ has edges $v_1v_4$, $v_Qv_7$, $v_2v_5$, and $v_3v_6$, it has a subgraph isomorphic to~$Q_3$, contradicting~\ref{cond:spinner addable1}.

    \ref{cond:spinner addable2}$\Rightarrow$\ref{cond:spinner addable3}:
    If $\mathcal{A}=\emptyset$, then the statement trivially holds, so we may assume that $\mathcal{A}$ is nonempty.
    We can deduce that $\mathcal{A}\subseteq\mathcal{S}:=\{\{v_1,v_2,v_3\},\{v_4,v_5,v_6\},\{v_1,v_4,v_7\},\{v_2,v_5,v_7\},\{v_3,v_6,v_7\}\}$ by considering the length-$4$ induced cycles on $\{v_1,v_4,v_6,v_3\}$ and $\{v_2,v_5,v_4,v_1\}$.
    There exists a trigraph $H_\mathcal{S}$ with underlying graph $\psi(\theta(G),\sigma_G(\mathcal{S}))$ such that~$H$ is an induced subgraph of~$H_\mathcal{S}$ and $\DeltaR(H_\mathcal{S})=\DeltaR(H)\leq3$.

    By Lemma~\ref{lem:repeat}, it suffices to find a trigraph~$H'$ with underlying graph $\psi^*(\theta(G),\sigma_G(X))$ and an $X$-stable partial \cont{3} from~$H_\mathcal{S}$ to~$H'$.
    Let
    \[
        W:=\bigcup_{S\in\mathcal{S}}(V(C'_{G,S})\setminus S)\quad\text{ and }\quad Z:=V(\theta(G))\setminus W.
    \]
    In other words, $W$ is the set of all internal vertices of length-$4$ paths of~$\theta(G)$ between two vertices of~$G$.
    We first apply Lemma~\ref{lem:yoperation} to obtain a $Z$-stable partial \cont{3} from~$H_\mathcal{S}$ to a trigraph~$H^*_1$ with underlying graph $G^*_1:=\psi(\theta(G),\sigma_G(\mathcal{S}))-W$.
    For each edge $v_iv_j$ of~$G$, let~$w_{i,j}$ and~$w_{j,i}$ be the internal vertices of the unique length-$3$ path in $H^*_1$ between $v_i$ and $v_j$ such that $v_iw_{i,j}\in E(H^*_1)$.
    For each $i\in[3]$, let $S_i:=\sigma_G(\{v_i,v_{i+3},v_7\})$.
    Let $S_4:=\sigma_G(\{v_1,v_2,v_3\})$ and let $S_5:=\sigma_G(\{v_4,v_5,v_6\})$.
    For each $i\in[5]$, let~$F_i$ be the subgraph of~$H_\mathcal{S}$ with underlying graph $\psi^*(\theta(G),S_i)$ and let $x_i$ be the unique vertex of degree~$3$ in $F_i-V(G)$.
    For each $i\in[3]$ and each $j\in\{i,i+3,7\}$, let~$y_{i,j}$ be the common neighbour of~$x_i$ and~$v_j$ in~$H^*_1$.
    For each $j\in[3]$, let~$y_{4,j}$ be the common neighbour of~$x_4$ and~$v_j$ in~$H^*_1$.
    For each $j\in\{4,5,6\}$, let~$y_{5,j}$ be the common neighbour of~$x_5$ and~$v_j$.
    For each $i\in[3]$, let $q(i)$ be the red degree of~$v_7$ in~$F_i$.
    Note that $q(1)+q(2)+q(3)=\rdeg_{H^*_1}(v_7)\leq3$.

    By symmetry, we may assume that $X$ is either $\{v_1,v_2,v_3\}$ or $\{v_1,v_4,v_7\}$, depending on whether $X$ contains~$v_7$ or not.
    By swapping~$v_2$ and $v_3$ and swapping $v_5$ and $v_6$ if necessary, we may assume that $v_5w_{5,4}\in R(H^*_1)$ or $v_6w_{6,4}\in B(H^*_1)$.
    The following three contractions are the first three contractions given by Lemma~\ref{lem:protection1} with $(t_1,t_2,t_3,t_4):=(v_5,v_6,v_4,x_5)$.
    Let $t'_1:=w_{5,6}$ if $v_5w_{5,4}\in B(H^*_1)$, and otherwise let $t'_1:=w_{5,4}$.
    Let $H^*_2$ be the trigraph obtained from~$H^*_1$ by contracting $\{t'_1,y_{5,5}\}$ to $y_{5,5}$, let $H^*_3$ be the trigraph obtained from $H^*_2$ by contracting $\{w_{6,5},y_{5,6}\}$ to $y_{5,6}$, and let $H^*_4$ be the trigraph obtained from $H^*_3$ by contracting $\{w_{4,5},y_{5,4}\}$ to~$y_{5,4}$.
    Note that $H^*_1,\ldots,H^*_4$ is an $X$-stable partial \cont{3}, and $\deg_{H^*_4}(v_j)=5$ for every $j\in\{4,5,6\}$.
    We are going to find an $X$-stable partial \cont{3} from~$H^*_4$ to a trigraph $H'$ with underlying graph $\psi^*(\theta(G),\sigma_G(X))$.

    Suppose first that $X=\{v_1,v_2,v_3\}$.
    Let $\ell_1$, $\ell_2$, and $\ell_3$ be distinct integers in $[3]$ with $q(\ell_1)\geq q(\ell_2)\geq q(\ell_3)$.
    For~$i$ from~$1$ to~$3$, by applying Corollary~\ref{cor:protection1} to $F:=F_{\ell_i}$ with $(t_1,t_2,t_3,t_4):=(v_{\ell_i+3},v_7,v_{\ell_i},x_{\ell_i})$, we find a $(V(H^*_{i+3})\setminus V(F_{\ell_i}))\cup\{v_{\ell_i}\}$-stable $\{v_{\ell_i+3},v_7\}$-fixing partial \cont{3} from~$H^*_{i+3}$ to a trigraph~$H^*_{i+4}$ whose underlying graph $G^*_{i+4}$ is the graph obtained from
    \[
        G^*_{i+3}-\{w_{\ell_i,\ell_i+3},w_{\ell_i+3,\ell_i},w_{\ell_i,7},w_{7,\ell_i},w_{\ell_i+3,7},w_{7,\ell_i+3},y_{\ell_i,\ell_i+3},y_{\ell_i,7}\}
    \]
    by adding edges $v_{\ell_i+3}x_{\ell_i}$ and $v_7x_{\ell_i}$.
    
    We apply Lemma~\ref{lem:protection1} to $F:=F_5$ with $(t_1,t_2,t_3,t_4):=(v_5,v_6,v_4,x_5)$.
    Note that we have already performed the first three contractions of the partial \cont{3} given by Lemma~\ref{lem:protection1}.
    By performing the remaining contractions, we find a $(V(H^*_7)\setminus V(F_5))\cup\{v_4\}$-stable $\{v_5,v_6\}$-fixing partial \cont{3} from $H^*_7$ to a trigraph~$H^*_8$ whose underlying graph $G^*_8$ is the graph obtained from
    \[
        G^*_7-\{w_{5,4},w_{4,6},w_{6,4},w_{5,6},y_{5,5},y_{5,6}\}
    \]
    by adding edges $v_5x_5$ and $v_6x_5$.
    Let $H_0:=G^*_8[N_{G^*_8}^2[x_4]]$.
    Note that $G_0:=G^*_8-\{w_{1,2},w_{2,1},w_{1,3},w_{3,1},w_{2,3},w_{3,2}\}$ is a subdivision of~$K_{3,3}$ and that $G_0-V(H_0)$ contains a subdivision of~$K_{2,3}$ as a subgraph.
    By Lemma~\ref{lem:routine} with $H_0$, there exists an $X$-stable partial \cont{3} from~$H^*_8$ to a trigraph $H'$ with underlying graph $\psi^*(\theta(G),\sigma_G(X))$.
    Thus, we find an $X$-stable partial \cont{3} from~$H_\mathcal{S}$ to~$H'$.

    We now assume that $X=\{v_1,v_4,v_7\}$.
    Let $k:=2$ if $v_2w_{2,1}\in R(H^*_4)$, and otherwise let $k:=3$.
    We first contract the first two contractions given by Lemma~\ref{lem:protection1} with $(t_1,t_2,t_3,t_4):=(v_k,v_{5-k},v_1,x_4)$.
    Let $t'_1:=w_{k,5-k}$ if $v_kw_{k,1}\in B(H^*_4)$, and otherwise let $t'_1:=w_{k,1}$.
    Let $H^*_5$ be the trigraph obtained from $H^*_4$ by contracting $\{t'_1,y_{4,k}\}$ to~$y_{4,k}$, let $H^*_6$ be the trigraph obtained from $H^*_5$ by contracting $\{w_{5-k,k},y_{4,5-k}\}$ to~$y_{4,5-k}$, and let $G^*_6$ be the underlying graph of $H^*_6$. 
    Note that $H^*_1,\ldots,H^*_6$ is an $X$-stable partial \cont{3} and $\deg_{H^*_6}(v_2)=\deg_{H^*_6}(v_3)=5$.

    Thus, for each $i\in\{2,3\}$, $v_i$ has at most two incident red edges in $H^*_6-E(F_i)$.
    For~$i$ from $2$ to $3$, by applying Corollary~\ref{cor:protection1} to $F:=F_i$ with $(t_1,t_2,t_3,t_4):=(v_i,v_{i+3},v_7,x_i)$, we find a $(V(H^*_{i+4})\setminus V(F_i))\cup\{v_7\}$-stable $\{v_i,v_{i+3}\}$-fixing partial \cont{3} from~$H^*_{i+4}$ to a trigraph~$H^*_{i+5}$ whose underlying graph~$G^*_{i+5}$ is the graph obtained from
    \[
        G^*_{i+4}-\{w_{i,i+3},w_{i+3,i},w_{i,7},w_{7,i},w_{i+3,7},w_{7,i+3},y_{i,i},y_{i,i+3}\}
    \]
    by adding edges $v_ix_i$ and $v_{i+3}x_i$.
    
    We apply Lemma~\ref{lem:protection1} to $F:=F_4$ with $(t_1,t_2,t_3,t_4):=(v_k,v_{5-k},v_1,x_4)$.
    Note that we have already performed the first two contractions of the partial \cont{3} given by Lemma~\ref{lem:protection1}.
    By performing the remaining contractions, we find a $(V(H^*_8)\setminus V(F_4))\cup\{v_1\}$-stable $\{v_2,v_3\}$-fixing partial \cont{3} from~$H^*_8$ to a trigraph~$H^*_9$ whose underlying graph $G^*_9$ is the graph obtained from
    \[
        G^*_8-\{w_{1,2},w_{2,1},w_{1,3},w_{3,1},w_{2,3},w_{3,2},y_{4,2},y_{4,3}\}
    \]
    by adding edges $v_2x_4$ and $v_3x_4$.
    
    We now apply Lemma~\ref{lem:protection1} to $F:=F_5$ with $(t_1,t_2,t_3,t_4):=(v_5,v_6,v_4,x_5)$.
    We have already performed the first three contractions of the partial \cont{3} given by Lemma~\ref{lem:protection1}.
    By performing the remaining five contractions, we find a $(V(H^*_9)\setminus V(F_5))\cup\{v_4\}$-stable $\{v_5,v_6\}$-fixing partial \cont{3} from $H^*_9$ to a trigraph~$H^*_{10}$ whose underlying graph $G^*_{10}$ is the graph obtained from
    \[
        G^*_9-\{w_{5,4},w_{4,6},w_{6,4},w_{5,6},y_{5,5},y_{5,6}\}
    \]
    by adding edges $v_5x_5$ and $v_6x_5$.
    Note that $G^*_{10}-\{w_{1,4},w_{4,1},w_{1,7},w_{7,1},w_{4,7},w_{7,4}\}$ is an \subd{1} of~$K_{3,3}$.
    By Lemma~\ref{lem:routine} applied to $H^*_{10}-\{w_{1,4},w_{4,1},w_{1,7},w_{7,1},w_{4,7},w_{7,4}\}$ with $H_0:=G^*_{10}[N_{G^*_{10}}^2[x_1]]$, there exists an $X$-stable partial \cont{3} from~$H^*_{10}$ to a trigraph~$H'$ with underlying graph $\psi^*(\theta(G),\sigma_G(X))$.
    Thus, we find an $X$-stable partial \cont{3} from~$H_\mathcal{S}$ to~$H'$.
    
    \ref{cond:spinner addable3}$\Rightarrow$\ref{cond:spinner addable1}:
    If $\mathcal{A}$ is empty, then \ref{cond:K_4 addable1} holds because $\overline{C_6}+K_1$ is $\mathcal{F}_3$-minor-free.
    Thus, we may assume that there is $X\in\mathcal{A}$.
    By~\ref{cond:spinner addable3}, there exists a trigraph~$H'$ with underlying graph $\psi^*(\theta(G),\sigma_G(X))$ and an $X$-stable partial \cont{3} from~$H$ to~$H'$.
    Since the underlying graph of $H'$ is a subdivision of $K_4$, $\tww(H')\leq3$, so that $\tww(H)\leq3$.
    Since the underlying graph of $H$ has an \subd{1} of~$Y(G,\mathcal{A})$ as an induced subgraph, by Proposition~\ref{prop:mainforward}, $Y(G,\mathcal{A})$ is $\mathcal{F}_3$-minor-free, and therefore~$\mathcal{A}$ is an addable set of~$G$.
\end{proof}

\begin{proposition}\label{prop:line addable}
    Let $G:=L(K_{3,3})$, let $\mathcal{A}$ be a set of $3$-vertex cliques in~$G$, and let $H$ be a trigraph with underlying graph $\psi(\theta(G),\sigma_G(\mathcal{A}))$ such that $\DeltaR(H)\leq3$.
    The following hold.
    \begin{enumerate}[label=\rm(\roman*)]
        \item\label{cond:line addable1} $\mathcal{A}$ is an addable set of $G$.
        \item\label{cond:line addable2} For every $X\in\mathcal{A}$, there exists a trigraph $H'$ with underlying graph $\psi^*(\theta(G),\sigma_G(X))$ and an $X$-stable partial \cont{3} from~$H$ to $H'$.
    \end{enumerate}
\end{proposition}
\begin{proof}
    We first show~\ref{cond:line addable2}.
    Assume that the vertices of~$G$ are labelled as in Figure~\ref{fig:bagtypes}.
    Suppose without loss of generality that $X=\{v_1,v_2,v_3\}$.
    Let $\mathcal{S}$ be the set of $3$-vertex cliques in~$G$.
    There exists a trigraph $H_\mathcal{S}$ with underlying graph $\psi(\theta(G),\sigma_G(\mathcal{S}))$ such that~$H$ is an induced subgraph of~$H_\mathcal{S}$ and $\DeltaR(H_\mathcal{S})=\DeltaR(H)\leq3$.

    By Lemma~\ref{lem:repeat}, it suffices to find a trigraph~$H'$ with underlying graph $\psi^*(\theta(G),\sigma_G(X))$ and an $X$-stable partial \cont{3} from~$H_\mathcal{S}$ to~$H'$.
    Let
    \[
        W:=\bigcup_{S\in\mathcal{S}}(V(C'_{G,S})\setminus S)\quad\text{ and }\quad Z:=V(\theta(G))\setminus W.
    \]
    In other words, $W$ is the set of all internal vertices of length-$4$ paths of~$\theta(G)$ between two vertices of~$G$.
    We first apply Lemma~\ref{lem:yoperation} to obtain a $Z$-stable partial \cont{3} from~$H_\mathcal{S}$ to a trigraph~$H^*_1$ with underlying graph $G^*_1:=\psi(\theta(G),\sigma_G(\mathcal{S}))-W$.
    For each edge $v_iv_j$ of~$G$, let $w_{i,j}$ and $w_{j,i}$ be the internal vertices of the unique length-$3$ path in $H^*_1$ between $v_i$ and $v_j$ such that $v_iw_{i,j}\in E(H^*_1)$.
    For each $i\in[3]$, let $S_i:=\sigma_G(\{v_i,v_{i+3},v_{i+6}\})$ and let $S_{i+3}:=\sigma_G(\{v_{3i-2},v_{3i-1},v_{3i}\})$.
    For each $\ell\in[6]$, let $F_\ell$ be the subgraph of~$H_\mathcal{S}$ with underlying graph $\psi^*(\theta(G),S_\ell)$, and let $x_\ell$ be the unique vertex of degree~$3$ in~$F_\ell-V(G)$.
    
    For each $i\in[3]$ and each $j\in\{i,i+3,i+6\}$, let $y_{i,j}$ be the common neighbour of~$x_i$ and~$v_j$ in~$H^*_1$.
    For each $i\in[3]$ and each $j\in\{3i-2,3i-1,3i\}$, let $y_{i+3,j}$ be the common neighbour of~$x_{i+3}$ and~$v_{j}$ in~$H^*_1$.

    By swapping $v_4$ and $v_5$ if necessary, we may assume that $v_4w_{4,6}\in R(H^*_1)$ or $v_5w_{5,6}\in B(H^*_1)$.
    We are going to apply the first three contractions given by Lemma~\ref{lem:protection1} with $(t_1,t_2,t_3,t_4):=(v_4,v_5,v_6,x_5)$.
    Let $z_1:=w_{4,5}$ if $v_4w_{4,6}\in B(H^*_1)$, and otherwise let $z_1:=w_{4,6}$.
    Let $H^*_2$ be the trigraph obtained from $H^*_1$ by contracting $\{y_{5,4},z_1\}$ to $y_{5,4}$, let $H^*_3$ be the trigraph obtained from $H^*_2$ by contracting $\{w_{5,4},y_{5,5}\}$ to~$y_{5,5}$, and let~$H^*_4$ be the trigraph obtained from $H^*_3$ by contracting $\{w_{6,4},y_{5,6}\}$ to~$y_{5,6}$.

    Let $\ell_1:=7$ if $v_7w_{7,9}\in R(H^*_4)$, and otherwise let $\ell_1:=8$.
    Let $\ell_2:=15-\ell_1$.
    The following are the first three contractions given by Lemma~\ref{lem:protection1} with $(t_1,t_2,t_3,t_4):=(v_{\ell_1},v_{\ell_2},v_9,x_6)$.
    Let $z_2:=w_{\ell_1,\ell_2}$ if $v_{\ell_1}w_{\ell_1,9}\in B(H^*_4)$, and otherwise let $z_2:=w_{\ell_1,9}$.
    Let $H^*_5$ be the trigraph obtained from~$H^*_4$ by contracting $\{y_{6,\ell_1},z_2\}$ to $y_{6,\ell_1}$, let~$H^*_6$ be the trigraph obtained from $H^*_5$ by contracting $\{w_{\ell_2,\ell_1},y_{6,\ell_2}\}$ to~$y_{6,\ell_2}$, let $H^*_7$ be the trigraph obtained from $H^*_6$ by contracting $\{w_{9,\ell_1},y_{6,9}\}$ to~$y_{6,9}$, and let $G^*_7$ be the underlying graph of $H^*_7$.
    Note that $H^*_1,\ldots,H^*_7$ is an $X$-stable partial \cont{3}, and $\deg_{H^*_7}(v_\ell)=5$ for every $\ell\in[9]\setminus[3]$.

    Thus, for each $i\in[3]$ and each $j\in\{i+3,i+6\}$, $v_j$ has at most two incident red edges in $H^*_7-E(F_i)$.
    For $i$ from $1$ to $3$, by applying Corollary~\ref{cor:protection1} to $F:=F_i$ with $(t_1,t_2,t_3,t_4):=(v_{i+3},v_{i+6},v_i,x_i)$, we find a $(V(H^*_{i+6})\setminus V(F_i))\cup\{v_i\}$-stable $\{v_{i+3},v_{i+6}\}$-fixing partial \cont{3} from $H^*_{i+6}$ to a trigraph~$H^*_{i+7}$ whose underlying graph $G^*_{i+7}$ is the graph obtained from
    \[
        G^*_{i+6}-\{w_{i,i+3},w_{i+3,i},w_{i,i+6},w_{i+6,i},w_{i+3,i+6},w_{i+6,i+3},y_{i,i+3},y_{i,i+6}\}
    \]
    by adding edges $v_{i+3}x_i$ and $v_{i+6}x_i$.

    We apply Lemma~\ref{lem:protection1} to $F:=F_5$ with $(t_1,t_2,t_3,t_4):=(v_4,v_5,v_6,x_5)$.
    Note that we have already performed the first three contractions of the partial \cont{3} given by Lemma~\ref{lem:protection1}.
    By performing the remaining contractions, we find a $(V(H^*_{10})\setminus V(F_5))\cup\{v_6\}$-stable $\{v_4,v_5\}$-fixing partial \cont{3} from~$H^*_{10}$ to a trigraph~$H^*_{11}$ whose underlying graph $G^*_{11}$ is the graph obtained from
    \[
        G^*_{10}-\{w_{4,5},w_{4,6},w_{5,6},w_{6,5},y_{5,4},y_{5,5}\}
    \]
    by adding edges $v_4x_5$ and $v_5x_5$.
    
    We now apply Lemma~\ref{lem:protection1} to $F:=F_6$ with $(t_1,t_2,t_3,t_4):=(v_{\ell_1},v_{\ell_2},v_9,x_6)$.
    We have already performed the first three contractions of the partial \cont{3} given by Lemma~\ref{lem:protection1}.
    By performing the remaining five contractions, we find a $(V(H^*_{11})\setminus V(F_6))\cup\{v_9\}$-stable $\{v_7,v_8\}$-fixing partial \cont{3} from~$H^*_{11}$ to a trigraph~$H^*_{12}$ whose underlying graph $G^*_{12}$ is the graph obtained from
    \[
        G^*_{11}-\{w_{\ell_1,\ell_2},w_{\ell_1,9},w_{\ell_2,9},w_{9,\ell_2},y_{6,\ell_1},y_{6,\ell_2}\}
    \]
    by adding edges $v_7x_6$ and $v_8x_6$.
    Note that $G^*_{12}-\{w_{1,2},w_{2,1},w_{1,3},w_{3,1},w_{2,3},w_{3,2}\}$ is an \subd{1} of $K_{3,3}$.
    By Lemma~\ref{lem:routine} applied to $H^*_{12}-\{w_{1,2},w_{2,1},w_{1,3},w_{3,1},w_{2,3},w_{3,2}\}$ with $H_0:=G^*_{12}[N_{G^*_{12}}^2[x_4]]$, there exists an $X$-stable partial \cont{3} from~$H^*_{12}$ to a trigraph~$H'$ with underlying graph $\psi^*(\theta(G),\sigma_G(X))$.
    Thus, we find an $X$-stable partial \cont{3} from~$H_\mathcal{S}$ to~$H'$.

    We now show~\ref{cond:line addable1}.
    If $\mathcal{A}$ is empty, then \ref{cond:line addable1} holds because $L(K_{3,3})$ is $\mathcal{F}_3$-minor-free.
    Thus, we may assume that there is $X\in\mathcal{A}$.
    By~\ref{cond:line addable2}, there exists a trigraph~$H'$ with underlying graph $\psi^*(\theta(G),\sigma_G(X))$ and an $X$-stable partial \cont{3} from~$H$ to~$H'$.
    Since the underlying graph of $H'$ is a subdivision of~$K_4$, $\tww(H')\leq3$, so that $\tww(H)\leq3$.
    Since the underlying graph of $H$ has an \subd{1} of~$Y(G,\mathcal{A})$ as an induced subgraph, by Proposition~\ref{prop:mainforward}, $Y(G,\mathcal{A})$ is $\mathcal{F}_3$-minor-free, and thus~$\mathcal{A}$ is an addable set of~$G$.
\end{proof}

We have the following corollary of Propositions~\ref{prop:K_4 addable}--\ref{prop:line addable}.

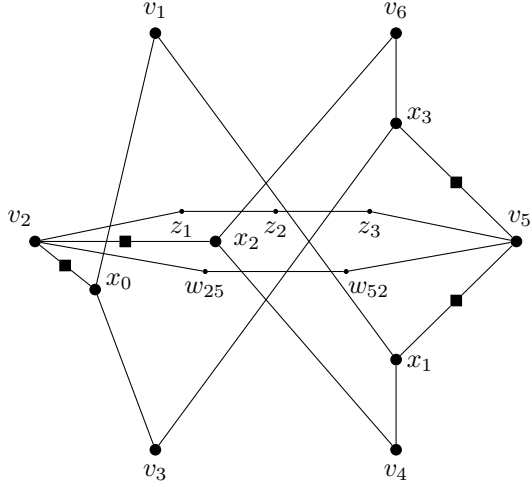
\begin{figure}[t]
    \centering
    \tikzstyle{v}=[circle, draw, solid, fill=black, inner sep=0pt, minimum width=1.5pt]
    \tikzstyle{w}=[circle, draw, solid, fill=black, inner sep=0pt, minimum width=4pt]
    \tikzstyle{u}=[rectangle, draw, fill=black, inner sep=2pt, minimum width=0.5pt]
    \begin{tikzpicture}[scale=0.8]
        \draw (120+60*0:4) node[w,label={[xshift=0.0mm,yshift=0.0mm]$v_1$}](1){};
        \draw (120+60*1:4) node[w,label={[xshift=-2.0mm,yshift=0.0mm]$v_2$}](2){};
        \draw (120+60*2:4) node[w,label={[xshift=0.0mm,yshift=-6.0mm]$v_3$}](3){};
        \draw (120+60*3:4) node[w,label={[xshift=0.0mm,yshift=-6.0mm]$v_4$}](4){};
        \draw (120+60*4:4) node[w,label={[xshift=0.5mm,yshift=0.0mm]$v_5$}](5){};
        \draw (120+60*5:4) node[w,label={[xshift=0.0mm,yshift=0.0mm]$v_6$}](6){};
        \draw (2)+(1,-0.8) node[w,label={[xshift=3.2mm,yshift=-2.0mm]$x_0$}](123){};
        \draw (2)+(3,0) node[w,label={[xshift=4.2mm,yshift=-3.0mm]$x_2$}](246){};
        \draw (6)+(0,-1.5) node[w,label={[xshift=3.2mm,yshift=-2.0mm]$x_3$}](356){};
        \draw (4)+(0,1.5) node[w,label={[xshift=3.2mm,yshift=-3.5mm]$x_1$}](145){};

        \path (123)--(2) node[u,pos=0.5](x1){};
        \draw (123)--(x1);
        \draw (x1)--(2);
        \draw (123)--(1);
        \draw (123)--(3);

        \path (246)--(2) node[u,pos=0.5](x4){};
        \draw (246)--(x4);
        \draw (x4)--(2);
        \draw (246)--(4);
        \draw (246)--(6);
        
        \path (356)--(5) node[u,pos=0.5](x7){};
        \draw (356)--(x7);
        \draw (x7)--(5);
        \draw (356)--(3);
        \draw (356)--(6);

        \path (145)--(5) node[u,pos=0.5](x10){};
        \draw (145)--(x10);
        \draw (x10)--(5);
        \draw (145)--(1);
        \draw (145)--(4);

        \path (2)--(5) node[pos=0.3](y1){};
        \path (2)--(5) node[pos=0.5](y2){};
        \path (2)--(5) node[pos=0.7](y3){};
        \draw (y1)+(0,0.5) node[v,label={[xshift=0.0mm,yshift=-5.0mm]$z_1$}](z1){};
        \draw (y2)+(0,0.5) node[v,label={[xshift=0.0mm,yshift=-5.0mm]$z_2$}](z2){};
        \draw (y3)+(0,0.5) node[v,label={[xshift=0.0mm,yshift=-5.0mm]$z_3$}](z3){};
        \draw (2)--(z1);
        \draw (z1)--(z2);
        \draw (z2)--(z3);
        \draw (z3)--(5);

        \path (2)--(5) node[pos=0.35](z4){};
        \path (2)--(5) node[pos=0.65](z5){};
        \draw (z4)+(0,-0.5) node[v,label={[xshift=0.0mm,yshift=-5.0mm]$w_{25}$}](y4){};
        \draw (z5)+(0,-0.5) node[v,label={[xshift=3.0mm,yshift=-5.0mm]$w_{52}$}](y5){};
        \draw (2)--(y4)--(y5)--(5);        
    \end{tikzpicture}
    \caption{The underlying graph $G'_9$ in the proof of Corollary~\ref{cor:base} where square vertices are~$y_{i,j}$ vertices.}
    \label{fig:base addable}
\end{figure}

\begin{corollary}\label{cor:base}
    Let $G\in\mathcal{K}\setminus\{K_1\}$, let $\mathcal{A}$ be an addable set of $G$, and let~$H$ be a trigraph with underlying graph $\psi(\theta(G),\sigma_G(\mathcal{A}))$ such that $\DeltaR(H)\leq3$.
    If $X\in\mathcal{A}\cup\{\{u,v\}:uv\in E(G)\}$, then there exists an $X$-stable partial \cont{3} from $H$ to a trigraph whose underlying graph is
    \begin{itemize}
        \item $\psi^*(\theta(G),\sigma_G(X))$ if $X\in\mathcal{A}$, and
        \item $\theta(G[X])$ otherwise.
    \end{itemize}
\end{corollary}

To prove the corollary, we use the following lemma whose proof is similar to that of Lemma~\ref{lem:repeat}.

\begin{lemma}\label{lem:repeat2}
    Let $G$ be a graph, let $\mathcal{S}$ be a set of $3$-vertex cliques in $G$, let $\mathcal{A}$ be a subset of $\mathcal{S}$, let $Y\in\mathcal{S}$, let~$X$ be a subset of $Y$ with $\abs{X}=2$, let $H_\mathcal{S}$ be a trigraph with underlying graph $\psi(\theta(G),\sigma_G(\mathcal{S}))$, and let~$H$ be the induced subgraph of $H_\mathcal{S}$ with underlying graph $\psi(\theta(G),\sigma_G(\mathcal{A}))$.
    If there exists an $X$-stable partial \cont{3} from $H_\mathcal{S}$ to a trigraph with underlying graph $\theta(G[X])$, then there exists an $X$-stable partial \cont{3} from $H$ to a trigraph with underlying graph $\theta(G[X])$.
\end{lemma}
\begin{proof}
    In the $X$-stable partial \cont{3} from~$H_{\mathcal S}$ to a trigraph $H'$ with underlying graph $\theta(G[X])$, we ignore contractions with vertices in $V(H_\mathcal{S})\setminus V(H)$ to obtain a partial \cont{3} from~$H$ to a trigraph $H''$ that is a refined subgraph of $H'$.
    Since the underlying graph of~$H$ contains $\theta(G[X])$ as an induced subgraph, $H''$ contains all edges of~$H'$, so the underlying graph of $H''$ is precisely $\theta(G[X])$.
\end{proof}

\begin{proof}[Proof of Corollary~\ref{cor:base}]
    We may assume that $\abs{V(G)}\geq3$, because otherwise $\psi(\theta(G),\sigma_G(\mathcal{A}))=\theta(G[X])$.
    By Lemmas~\ref{lem:repeat} and~\ref{lem:repeat2}, we may further assume that $\mathcal{A}$ is a maximal addable set of~$G$.

    Let us first consider the case that $X\in\mathcal{A}$.
    If~$G$ has at least four vertices, then by Propositions~\ref{prop:K_4 addable}--\ref{prop:line addable}, the statement holds.
    If $G=K_3$, then let $W$ be the set of vertices of the unique length-$12$ cycle of $\theta(G)$ which are not in $V(G)$.
    Since $X=V(G)\subseteq V(\theta(G))\setminus W$, by Lemma~\ref{lem:yoperation}, there exists an $X$-stable partial \cont{3} from~$H$ to a trigraph with underlying graph~$\psi(\theta(G),\sigma_G(X))-W$ which is equal to $\psi^*(\theta(G),\sigma_G(X))$.
    
    We now consider the case that $X\subsetneq X'$ for some $X'\in\mathcal{A}$.
    By the previous paragraph, there exists an $X'$-stable partial \cont{3} from $H$ to a trigraph $H^*_1$ with underlying graph $G^*_1:=\psi^*(\theta(G),\sigma_G(X'))$.
    Let $v_1$, $v_2$, and $v_3$ be the vertices in $X'$ with $v_3\in X'\setminus X$ and let~$x$ be the unique vertex of degree $3$ in $G^*_1-V(G)$.
    For distinct $i,j\in[3]$, let~$w_{i,j}$ and~$w_{j,i}$ be the internal vertices of the unique length-$3$ path in~$H^*_1$ between~$v_i$ and~$v_j$ such that $v_iw_{i,j}\in E(H^*_1)$, and let~$y_i$ be the common neighbour of~$v_i$ and~$x$ in~$H^*_1$.
    
    By Lemma~\ref{lem:protection2} with $F:=H^*_1$ and $(t_1,t_2,t_3,t_4):=(v_1,v_2,v_3,x)$, there exists a $\{v_1,v_2\}$-stable $\{v_3\}$-fixing partial \cont{3} from $H^*_1$ to a trigraph whose underlying graph is the graph obtained from $G^*_1-\{w_{1,3},w_{3,1},w_{2,3},w_{3,2},y_3\}$ by adding an edge~$v_3x$.
    We then contract $v_3$ and $x$ to obtain an $X$-stable partial \cont{3} from~$H$ to a trigraph with underlying graph $\theta(G[X])$ as desired.

    We finally assume that $X\nsubseteq X'$ for every $X'\in\mathcal{A}$.
    Since $\mathcal{A}$ is maximal, by Propositions~\ref{prop:K_4 addable}--\ref{prop:line addable}, if $G\neq K_6^{=}$, then every edge of $G$ has both ends in some set in $\mathcal{A}$.
    Hence, $G=K_6^{=}$, and by Proposition~\ref{prop:K_6 addable}, $X$ is the set of degree-$5$ vertices of~$G$.
    We label the vertices of $G$ as in Figure~\ref{fig:bagtypes}.
    Note that $X=\{v_2,v_5\}$.
    By Proposition~\ref{prop:K_6 addable}\ref{cond:K_6 addable2}, there exists $B\subseteq V(G)$ such that both $B$ and $V(G)\setminus B$ are cliques of size~$3$ in~$G$ and for all $Q\in\mathcal{A}$, $\abs{Q\cap B}$ is even and $Q$ contains at most one of~$v_2$ and~$v_5$.
    By symmetry between $v_i$ and $v_{7-i}$ for each $i\in[3]$, we may assume that $B=\{v_4,v_5,v_6\}$.
    Then $\mathcal{A}=\{\{v_1,v_2,v_3\},\{v_2,v_4,v_6\},\{v_3,v_5,v_6\},\{v_1,v_4,v_5\}\}$, because $\mathcal{A}$ is maximal.
    
    Let
    \[
        W:=\bigcup_{S\in\mathcal{A}}(V(C'_{G,S})\setminus S)\quad\text{ and }\quad Z:=V(\theta(G))\setminus W.
    \]
    In other words, $W$ is the set of all internal vertices of length-$4$ paths of~$\theta(G)$ between two vertices of~$G$, except for the length-$4$ path of~$\theta(G)$ between~$v_2$ and~$v_5$.
    We first apply Lemma~\ref{lem:yoperation} to obtain a $Z$-stable partial \cont{3} from~$H$ to a trigraph $H'_1$ with underlying graph $G'_1:=\psi(\theta(G),\sigma_G(\mathcal{A}))-W$.
    For each edge $v_iv_j$ of~$G$, let $w_{i,j}$ and $w_{j,i}$ be the internal vertices of the unique length-$3$ path in $H'_1$ between~$v_i$ and~$v_j$ such that $v_iw_{i,j}\in E(H'_1)$.
    Let $z_1$, $z_2$, and $z_3$ be the internal vertices of the unique length-$4$ path in~$H'_1$ between~$v_2$ and~$v_5$ such that $\{v_2z_1,v_5z_3\}\subseteq E(H'_1)$.

    Let $S_0:=\sigma_G(\{v_1,v_2,v_3\})$, and let $x_0$ be the unique vertex of degree~$3$ in $\psi^*(\theta(G),S_0)-V(G)$.
    For each $i\in[3]$, let~$y_{0,i}$ be the common neighbour of~$x_0$ and~$v_i$ in~$H'_1$, let $S_i:=\sigma_G(\{v_i,v_4,v_5,v_6\}\setminus\{v_{7-i}\})$, let $x_i$ be the unique vertex of degree~$3$ in $\psi^*(\theta(G),S_i)-V(G)$, and for each $j\in\{i,4,5,6\}\setminus\{7-i\}$, let $y_{i,j}$ be the common neighbour of $x_i$ and~$v_j$ in~$H'_1$.

    By swapping $v_1$ and $v_3$ and swapping $v_4$ and $v_6$ if necessary, we may assume that $v_4w_{4,2}\in R(H'_1)$ or $v_6w_{6,2}\in B(H'_1)$.
    We apply the first two contractions given by Lemma~\ref{lem:protection1} with $(t_1,t_2,t_3,t_4):=(v_4,v_6,v_2,x_2)$.
    Let $s_1:=w_{4,6}$ if $v_4w_{4,2}\in B(H'_1)$, and otherwise let $s_1:=w_{4,2}$.
    Let $H'_2$ be the trigraph obtained from $H'_1$ by contracting $\{s_1,y_{2,4}\}$ to~$y_{2,4}$, and let~$H'_3$ be the trigraph obtained from~$H'_2$ by contracting $\{w_{6,4},y_{2,6}\}$ to~$y_{2,6}$.

    Let $k:=1$ if $v_1w_{1,2}\in R(H'_3)$, and otherwise let $k:=3$.
    We now apply the first two contractions given by Lemma~\ref{lem:protection1} with $(t_1,t_2,t_3,t_4):=(v_k,v_{4-k},v_2,x_0)$.
    Let $s_2:=w_{k,4-k}$ if $v_kw_{k,2}\in B(H'_3)$, and otherwise let $s_2:=w_{k,2}$.
    Let $H'_4$ be the trigraph obtained from $H'_3$ by contracting $\{s_2,y_{0,k}\}$ to~$y_{0,k}$, and let $H'_5$ be the trigraph obtained from $H'_4$ by contracting $\{w_{4-k,k},y_{0,4-k}\}$ to~$y_{0,4-k}$.
    Note that $H'_1,\ldots,H'_5$ is an $X$-stable partial \cont{3}, and $\deg_{H'_5}(v_\ell)=5$ for every $\ell\in[6]\setminus\{2,5\}$.

    Thus, for each $j\in\{1,3\}$, each of $v_j$ and $v_{j+3}$ has at most two incident red edges in $H'_5-E(F_j)$.
    By applying Corollary~\ref{cor:protection1} to $F:=F_1$ with $(t_1,t_2,t_3,t_4):=(v_1,v_4,v_5,x_1)$, there exists a $(V(H'_5)\setminus V(\psi^*(\theta(G),S_1)))\cup\{v_5\}$-stable $\{v_1,v_4\}$-fixing partial \cont{3} from~$H'_5$ to a trigraph~$H'_6$ whose underlying graph $G'_6$ is the graph obtained from
    \[
        G'_5-\{w_{1,4},w_{4,1},w_{1,5},w_{5,1},w_{4,5},w_{5,4},y_{1,1},y_{1,4}\}
    \]
    by adding edges $v_1x_1$ and $v_4x_1$.
    Again, by applying Corollary~\ref{cor:protection1} to $F:=F_3$ with $(t_1,t_2,t_3,t_4):=(v_3,v_6,v_5,x_3)$, there exists a $(V(H'_6)\setminus V(\psi^*(\theta(G),S_3)))\cup\{v_5\}$-stable $\{v_3,v_6\}$-fixing partial \cont{3} from~$H'_6$ to a trigraph~$H'_7$ whose underlying graph $G'_7$ is the graph obtained from
    \[
        G'_6-\{w_{3,6},w_{6,3},w_{3,5},w_{5,3},w_{5,6},w_{6,5},y_{3,3},y_{3,6}\}
    \]
    by adding edges $v_3x_3$ and $v_6x_3$.

    We apply Lemma~\ref{lem:protection1} with $(t_1,t_2,t_3,t_4):=(v_4,v_6,v_2,x_2)$.
    Recall that we have performed the first two contractions given by Lemma~\ref{lem:protection1}.
    By applying the remaining contractions, we find a $(V(H'_7)\setminus V(\psi^*(\theta(G),S_2)))\cup\{v_2\}$-stable $\{v_4,v_6\}$-fixing partial \cont{3} from~$H'_7$ to a trigraph $H'_8$ whose underlying graph~$G'_8$ is the graph obtained from
    \[
        G'_7-\{w_{2,4},w_{4,2},w_{2,6},w_{6,2},w_{4,6},y_{2,4},y_{2,6}\}
    \]
    by adding edges $v_4x_2$ and $v_6x_2$.
    
    We now apply Lemma~\ref{lem:protection1} with $(t_1,t_2,t_3,t_4):=(v_k,v_{4-k},v_2,x_0)$.
    We have performed the first two contractions given by Lemma~\ref{lem:protection1}.
    Thus, with the remaining contractions, we find a $(V(H'_8)\setminus V(\psi^*(\theta(G),S_0)))\cup\{v_2\}$-stable $\{v_1,v_3\}$-fixing partial \cont{3} from~$H'_8$ to a trigraph whose underlying graph $G'_9$ is the graph obtained from
    \[
        G'_8-\{w_{1,2},w_{2,1},w_{1,3},w_{3,1},w_{2,3},w_{3,2},y_{0,1},y_{0,3}\}
    \]
    by adding edges $v_1x_0$ and $v_3x_0$; see Figure~\ref{fig:base addable}.
    We then contract the following twelve pairs of vertices in order.
    \[
        \begin{array}{llll}
            1)\ \{v_1,x_0\}\text{ to }x_0,
            &2)\ \{v_3,x_0\}\text{ to }x_0,
            &3)\ \{v_4,x_2\}\text{ to }x_2,
            &4)\ \{v_6,x_2\}\text{ to }x_2,\\
            5)\ \{y_{0,2},y_{2,2}\}\text{ to }y_{0,2},
            &6)\ \{y_{1,5},y_{3,5}\}\text{ to }y_{1,5},
            &7)\ \{x_0,x_2\}\text{ to }x_0,
            &8)\ \{x_1,x_3\}\text{ to }x_1,\\
            9)\ \{x_0,x_1\}\text{ to }x_0,
            &10)\ \{y_{0,2},z_1\}\text{ to }z_1,
            &11)\ \{y_{1,5},z_3\}\text{ to }z_3,
            &12)\ \{x_0,z_2\}\text{ to }z_2.
        \end{array}
    \]
    Thus, we find an $X$-stable partial \cont{3} from $H$ to a trigraph with underlying graph $\theta(G[X])$ as desired.
\end{proof}
    
\subsection{Constructing {\boldmath$3$}-contractible tree decompositions}\label{subsec:construct}

In this subsection, we show the following proposition which shows that every $\mathcal{F}_3$-minor-free multigraph admits a $3$-contractible tree decomposition if we add some edges without creating $\mathcal{F}_3$-minors.

\begin{proposition}\label{prop:treedecomp}
    Let $G$ be an $\mathcal{F}_3$-minor-free multigraph.
    There exists an $\mathcal{F}_3$-minor-free multigraph~$G'$ admitting a $3$-contractible tree decomposition such that $G$ is a spanning subgraph of~$G'$.
    In particular, if~$G$ is simple, then $G'$ can be chosen as a simple graph.
\end{proposition}

To prove Proposition~\ref{prop:treedecomp}, we will use the following theorem and propositions.

\begin{theorem}[Ananchuen and Lewchalermvongs~\cite{AL2021}]\label{thm:4connected}
    If a graph is $4$-connected and $\{Q_3,V_8\}$-minor-free, then it either has at most $7$ vertices, or is isomorphic to~$L(K_{3,3})$.
\end{theorem}

\begin{proposition}\label{prop:4connected}
    If a multigraph $G$ is $\mathcal{F}_3$-minor-free and has no separator of size at most~$3$, then the simplification of~$G$ is isomorphic to a graph in $\mathcal{K}$.
\end{proposition}
\begin{proof}
    It suffices to show the statement for simple graphs, so we assume that~$G$ is simple.
    Since $G$ has no separator of size at most~$3$, if $\abs{V(G)}\leq5$, then~$G$ is a complete graph.
    Thus, we may assume that $\abs{V(G)}\geq6$, and therefore $G$ is $4$-connected.
    Since $L(K_{3,3})\in\mathcal{F}_3$, by Theorem~\ref{thm:4connected}, we may assume that $\abs{V(G)}\leq7$.
    
    Suppose first that $\abs{V(G)}=6$.
    Since~$G$ is $4$-connected, every vertex has degree at least~$4$ in~$G$, and therefore $\Delta(\overline{G})\leq1$.
    Since $G$ has no subgraph isomorphic to~$K_6^-\in\mathcal{F}_3$, $\overline{G}$ has at least two edges.
    Hence,~$G$ is isomorphic to~$K_6^\equiv$ or~$K_6^{=}$.

    Thus, we may assume that $\abs{V(G)}=7$.
    Again by the $4$-connectedness of~$G$, every vertex of $G$ has degree at least~$4$, and therefore $\Delta(\overline{G})\leq2$, which implies that every component of $\overline{G}$ is either a cycle or a path.
    Since~$G$ has no subgraph isomorphic to~$\overline{C_7}\in\mathcal{F}_3$, $\overline{G}$ is not a subgraph of a cycle of length~$7$.
    Therefore, $\overline{G}$ has a cycle of length at most~$6$.
    Let~$C$ be the longest cycle in~$\overline{G}$.
    If $\abs{V(C)}=5$, then~$G$ has a subgraph isomorphic to $C_5+\overline{K_2}\in\mathcal{F}_3$, a contradiction.
    If $\abs{V(C)}=4$, then $V(G)\setminus V(C)$ is a $3$-vertex separator of~$G$, a contradiction.
    If $\abs{V(C)}=3$, then~$G$ has a subgraph isomorphic to $K_{3,\widehat{1},3}\in\mathcal{F}_3$, a contradiction.
    Thus, $\abs{V(C)}=6$, and therefore~$G$ is isomorphic to~$\overline{C_6}+K_1$.
\end{proof}

\begin{proposition}\label{prop:completing}
    Let $S$ be a minimal separator of a multigraph~$G$ with $\abs{S}\leq3$.
    Let $G_S$ be the multigraph obtained from~$G$ by making $S$ a clique.
    Then $G$ is $\mathcal{F}_3$-minor-free if and only if $G_S$ is $\mathcal{F}_3$-minor-free.
\end{proposition}
\begin{proof}
    We may assume that $S$ is neither an empty set nor a clique, because otherwise the statement clearly holds.
    Since $G$ is a subgraph of $G_S$, the backward direction is obvious.

    For the forward direction, suppose that~$G_S$ has a minor $H\in\mathcal{F}_3$.
    We are going to show that $G$ has a minor in $\mathcal{F}_3$.
    Let $(T_u)_{u\in V(H)}$ be a minor model of~$H$ in~$G_S$, and let $X$ be the set of vertices~$u$ of~$H$ such that $V(T_u)\cap S\neq\emptyset$.
    We may assume that $X\neq \emptyset$, because otherwise the statement clearly holds.
        
    Suppose that there are distinct vertices $v$ and $w$ of $H$ and distinct components $D$ and $D'$ of $G-S$ such that~$T_v$ is a subtree of~$D$ and~$T_w$ is a subtree of~$D'$.
    Thus,~$X$ is a separator of~$H$ of size at most~$3$, so by Lemma~\ref{lem:internally}, we have that~$X$ is an independent set of size $3$ in $H$.
    It follows that for each $u\in X$, we have $\abs{V(T_u)\cap S}=1$, so that for each $u\in V(H)$, $T_u$ is a subtree of~$G$.
    Since $X$ is an independent set in~$H$, $(T_u)_{u\in V(H)}$ is a minor model of~$H$ in~$G$.

    Thus, we may assume that there is a component $D$ of $G-S$ such that for every $v\in V(H)\setminus X$, $V(T_v)\cap V(D)=\emptyset$.
    We now consider the following three cases.

    \medskip
    \noindent\textbf{Case 1.} For some $u\in X$, $\abs{V(T_u)\cap S}\geq2$.
    
    Note that $\abs{X}\leq 2$ as $\abs{S}\leq3$.
    Since $S$ is a minimal separator of $G$, $G[V(T_u)\cup V(D)]$ has a spanning tree, say~$T'_u$.
    For every $v\in V(H)\setminus\{u\}$, let $T'_v:=T_v-V(D)$.
    Since $\abs{X}\leq2$ and $\abs{V(T_u)\cap S}\geq2$, if there exists $u'\in X\setminus \{u\}$, then $\abs{V(T_{u'})\cap S}=1$, so that $T'_{u'}$ is a subtree of~$G$ and there is an edge of~$G$ between a vertex of~$T'_u$ and a vertex of~$T'_{u'}$.
    Hence, $(T'_w)_{w\in V(H)}$ is a minor model of $H$ in $G$.

    \medskip
    \noindent\textbf{Case 2.} $\abs{V(T_v)\cap S}\leq1$ for every $v\in V(H)$ and $X$ is not a clique of size~$3$ in~$H$.
    
    Since $1\leq\abs{X}\leq 3$, $X$ contains some vertex $u$ such that $X\setminus\{u\}$ is an independent set in~$H$.
    Since $S$ is a minimal separator of $G$, there exists an edge~$e$ of~$G$ between the vertex in $V(T_u)\cap S$ and a vertex in~$V(D)$.
    Let $T'_u$ be the subtree of~$G$ obtained from the disjoint union of $T_u-V(D)$ and a spanning tree of~$D$ by adding~$e$.
    For every $v\in V(H)\setminus\{u\}$, let $T'_v:=T_v-V(D)$, which is a subtree of~$G$ since $\abs{V(T_v)\cap S}\leq1$.
    Since $S$ is a minimal separator of~$G$, for every $u'\in X\setminus\{u\}$, $G$ has an edge between a vertex of $T'_u$ and a vertex of $T'_{u'}$.
    Hence, $(T'_w)_{w\in V(H)}$ is a minor model of $H$ in $G$.

    \medskip
    \noindent\textbf{Case 3.} $\abs{V(T_v)\cap S}\leq1$ for every $v\in V(H)$ and $X$ is a clique of size~$3$ in~$H$.

    Let $H'$ be the graph obtained from~$H$ by applying a $(\Delta,Y)$-operation on $H[X]$.
    By Lemma~\ref{lem:DeltaY}, $H'$ has a spanning subgraph isomorphic to a graph in~$\mathcal{F}_3$.
    Thus, it suffices to show that $H'$ is a minor of~$G$.
    Let $x$ be the vertex in $V(H')\setminus V(H)$, and let $T'_x$ be a spanning tree of $D$.
    For every $v\in V(H)$, let $T'_v:=T_v-V(D)$, which is a subtree of~$G$ since $\abs{V(T_v)\cap S}\leq1$.
    Since $S$ is a minimal separator of~$G$, for every $u\in X$, $G$ has an edge between a vertex of $T'_x$ and a vertex of~$T'_u$.
    Thus, $(T'_w)_{w\in V(H')}$ is a minor model of~$H'$ in~$G$, and this completes the proof.
\end{proof}

We now prove Proposition~\ref{prop:treedecomp}.

\begin{proof}[Proof of Proposition~\ref{prop:treedecomp}]
    We proceed by induction on $\abs{V(G)}+\abs{E(G)}$.
    Suppose first that $G$ is not simple.
    Let $e$ be an edge of $G$ which is a loop or a parallel edge.
    By the inductive hypothesis, there exists an $\mathcal{F}_3$-minor-free multigraph~$G^*$ and a $3$-contractible tree decomposition of~$G^*$ such that $G-\{e\}$ is a spanning subgraph of $G^*$.
    Thus, the statement holds for $G$ by taking $G'$ as the multigraph obtained from~$G^*$ by adding~$e$ and taking the same tree decomposition.

    Hence, we may assume that $G$ is simple.
    For a set $S\subseteq V(G)$, let $G_S$ be the graph obtained from $G$ by making $S$ a clique.
    The statement clearly holds for $\abs{V(G)}\leq5$ by $G':=G_{V(G)}$ and a tree decomposition with only one bag.
    Thus, we may assume that $\abs{V(G)}\geq6$.

    Suppose that $G$ is disconnected.
    Let $C_1$ be an arbitrary component of $G$ and let $C_2:=G-V(C_1)$.
    By the inductive hypothesis, for each $i\in[2]$, there exists an $\mathcal{F}_3$-minor-free graph~$C'_i$ and a $3$-contractible tree decomposition $(T_i,(B_t)_{t\in V(T_i)})$ of~$C'_i$ such that $C_i$ is a spanning subgraph of~$C'_i$.
    We may assume that~$T_1$ and~$T_2$ are disjoint.
    For each $i\in[2]$, we choose an arbitrary vertex~$v_i$ of~$C'_i$, and let $G'$ be the graph obtained from $C'_1\cup C'_2$ by adding an edge $v_1v_2$.
    Since every graph in~$\mathcal{F}_3$ is $3$-connected, $G'$ is $\mathcal{F}_3$-minor-free.
    For each $i\in[2]$, let~$u_i$ be a node of $T_i$ such that $v_i\in B_{u_i}$.
    Let $T'$ be a tree obtained from $T_1\cup T_2$ by adding a new node $u$ with neighbourhood $\{u_1,u_2\}$, and let $B'_u:=\{v_1,v_2\}$.
    For every $t\in V(T')\setminus\{u\}$, let $B'_t:=B_t$.
    Thus, $(T',(B'_t)_{t\in V(T')})$ is a $3$-contractible tree decomposition of~$G'$.

    Therefore, we may assume that $G$ is connected.
    Suppose that $G$ has a cut-vertex $v$.
    For a component $C$ of $G-v$, let $C_1:=G[V(C)\cup\{v\}]$ and let $C_2:=G-V(C)$.
    By the inductive hypothesis, for each $i\in[2]$, there exists an $\mathcal{F}_3$-minor-free graph~$C'_i$ and a $3$-contractible tree decomposition $(T_i,(B_t)_{t\in V(T_i)})$ of~$C'_i$ such that $C_i$ is a spanning subgraph of~$C'_i$.
    We may assume that~$T_1$ and~$T_2$ are disjoint.
    Let $G':=C'_1\cup C'_2$.
    Since every graph in~$\mathcal{F}_3$ is $3$-connected, $G'$ is $\mathcal{F}_3$-minor-free.
    For each $i\in[2]$, let~$u_i$ be a node of $T_i$ such that $v\in B_{u_i}$.
    By adding an edge between~$u_1$ and~$u_2$ to~$T_1\cup T_2$, we find a $3$-contractible tree decomposition of $G'$.
    
    Hence, we may assume that $G$ is $2$-connected.
    Suppose that there exists a $2$-vertex separator~$S$ of~$G$.
    For a component~$C$ of $G-S$, let $C_1:=G_S[V(C)\cup S]$ and let $C_2:=G_S-V(C)$.
    By Proposition~\ref{prop:completing}, both~$C_1$ and~$C_2$ are $\mathcal{F}_3$-minor-free.
    By the construction of $G_S$, $S$ is also a minimal separator of~$G_S$, and therefore for each component~$D$ of~$G_S-S$, $G_S$ has at least two edges between~$S$ and~$V(D)$.
    Therefore, for each $i\in[2]$, we have that $\abs{V(C_i)}+\abs{E(C_i)}<\abs{V(G)}+\abs{E(G)}$.
    By the inductive hypothesis, for each $i\in[2]$, there exists an $\mathcal{F}_3$-minor-free graph~$C'_i$ and a $3$-contractible tree decomposition $(T_i,(B_t)_{t\in V(T_i)})$ of~$C'_i$ such that $C_i$ is a spanning subgraph of $C'_i$.
    We may assume that~$T_1$ and~$T_2$ are disjoint.
    Let $G':=C'_1\cup C'_2$.
    Since every graph in~$\mathcal{F}_3$ is $3$-connected and $S$ is a clique in $G'$, $G'$ is $\mathcal{F}_3$-minor-free.
    For each $i\in[2]$, by Lemma~\ref{lem:complete bag}, there exists a node $u_i$ of $T_i$ such that $S\subseteq B_{u_i}$.
    By adding an edge between $u_1$ and $u_2$ to $T_1\cup T_2$, we find a $3$-contractible tree decomposition of $G'$.

    Thus, we may now assume that $G$ is $3$-connected.
    By Proposition~\ref{prop:completing}, we may assume that every $3$-vertex separator of $G$ is a clique in $G$.

    \begin{claim}\label{clm:allocating}
        Let $(T,(B_t)_{t\in V(T)})$ be a tree decomposition of~$G$ satisfying~\ref{def:td contractible edge}.
        If~$T$ has a node $x$ such that there exists a minimal separator $S$ of $G[B_x]$ with $\abs{S}\leq3$, then we choose an arbitrary component $C$ of $G[B_x]-S$ and let $N_1^x:=\{t\in N_T(x):B_t\cap(V(C)\cup S)\neq\emptyset\}$ and $N_2^x:=N_T(x)\setminus N_1^x$.
        Let $T'$ be a tree obtained from $T-x$ by adding new vertices $x_1,x_2$ and the edges in $\{x_it:i\in[2],t\in N_i^x\}\cup\{x_1x_2\}$, let $B'_{x_1}:=V(C)\cup S$, let $B'_{x_2}:=B_x\setminus V(C)$, and for each $t\in V(T)\setminus\{x\}$, let $B'_t:=B_t$.
        Then $(T',(B'_t)_{t\in V(T')})$ is a tree decomposition of~$G$ satisfying~\ref{def:td contractible edge}.
    \end{claim}
    \begin{subproof}
        Let $\mathcal{T}:=(T,(B_t)_{t\in V(T)})$ and $\mathcal{T}':=(T',(B'_t)_{t\in V(T')})$.
        We first show that $\mathcal{T}'$ is a tree decomposition of~$G$.
        Since $\mathcal{T}$ satisfies~\ref{def:td bag union} and $B_x=B'_{x_1}\cup B'_{x_2}$, $\mathcal{T}'$ also satisfies~\ref{def:td bag union}.
        Since $\mathcal{T}$ satisfies~\ref{def:td edge} and $G$ has no edge between $B'_{x_1}\setminus S$ and $B'_{x_2}\setminus S$, $\mathcal{T}'$ also satisfies~\ref{def:td edge}.

        We verify that $\mathcal{T}'$ satisfies~\ref{def:td vertex}.
        Since $T-x=T'-\{x_1,x_2\}$ and $B'_t=B_t$ for every $t\in V(T)\setminus\{x\}$, it suffices to show that for every $v\in B_x$, $T'[\{t\in V(T'):v\in B'_t\}]$ is connected.
        We fix a vertex $v\in B_x$.
        Since~$\mathcal{T}$ satisfies~\ref{def:td vertex} and $B'_{x_1}\cap B'_{x_2}=S$, if $v\in S$, then $T'[\{t\in V(T'):v\in B'_t\}]$ is connected.
        Thus, we may assume that $v\notin S$.
        For every node $t$ of $T$ with $v\in B_t$, by~\ref{def:td vertex}, if $v\in B'_{x_i}$, then $T-x$ has a path between~$t$ and a node in $N_i^x$.
        Thus, $T[\{t\in V(T):v\in B_t\}]$ is isomorphic to $T'[\{t\in V(T'):v\in B'_t\}]$.
        Hence,~$\mathcal{T}'$ satisfies~\ref{def:td vertex}, and therefore it is a tree decomposition of $G$.

        We now show that $\mathcal{T}'$ satisfies~\ref{def:td contractible edge}.
        Since~$G$ is $3$-connected, every $3$-vertex separator of $G$ is minimal.
        Since every $3$-vertex separator of~$G$ is a clique in~$G$, and $\mathcal{T}$ satisfies~\ref{def:td contractible edge}, by the construction of $\mathcal{T}'$, it suffices to show that for every edge $vw$ of $T'$ incident with $x_1$ or $x_2$, $B'_v\cap B'_w$ is a separator of~$G$.
        
        Suppose that $vw=x_1x_2$.
        Since $\mathcal{T}'$ is a tree decomposition of $G$ where both $B'_{x_1}\setminus S$ and $B'_{x_2}\setminus S$ are nonempty, by Lemma~\ref{lem:adhesion}, $S=B'_{x_1}\cap B'_{x_2}$ is a separator of~$G$.
        
        Thus, we may assume that $v=x_i$ and $w\in N_i^x$ for some $i\in[2]$.
        Since $B'_{x_1}\cap B'_{x_2}=S$, by~\ref{def:td vertex}, we have that $B'_w\cap(B'_{x_{3-i}}\setminus S)=\emptyset$, and therefore $B'_v\cap B'_w=B_x\cap B_w$.
        Since $\mathcal{T}$ satisfies~\ref{def:td contractible edge}, $B'_v\cap B'_w$ is a separator of~$G$.
        Hence, $\mathcal{T}'$ satisfies~\ref{def:td contractible edge}.
    \end{subproof}

    We remark that a tree decomposition with only one bag satisfies~\ref{def:td contractible edge}.
    By Claim~\ref{clm:allocating}, there exists a tree decomposition $\mathcal{T}^*:=(T^*,(B_t^*)_{t\in V(T^*)})$ of $G$ satisfying~\ref{def:td contractible edge} where each bag has no separator of size at most~$3$.
    We now show that $\mathcal{T}^*$ satisfies~\ref{def:td contractible vertex}.
    Let $u$ be an arbitrary node of~$T^*$, and let $\mathcal{A}_u:=\{B^*_u\cap B^*_{u'}:u'\in N_{T^*}(u),\abs{B^*_u\cap B^*_{u'}}=3\}$.
    Since~$G$ is $\mathcal{F}_3$-minor-free, by Proposition~\ref{prop:4connected}, $G[B^*_u]$ is isomorphic to a graph in~$\mathcal{K}$.
    By~\ref{def:td contractible edge}, every adhesion set in $\mathcal{T}^*$ is a minimal separator of~$G$, and therefore $Y(G[B^*_u],\mathcal{A}_u)$ is a minor of~$G$.
    Since~$G$ is $\mathcal{F}_3$-minor-free, $\mathcal{A}_u$ is addable to~$B^*_u$ in~$G$.
    By the arbitrary choice of~$u$, $\mathcal{T}^*$ satisfies~\ref{def:td contractible vertex}, and therefore it is a $3$-contractible tree decomposition of~$G$.
    This completes the proof by induction.
\end{proof}

\subsection{Proof of Proposition~\ref{prop:mainbackward}}\label{subsec:backwardproof}

In this subsection, we complete the proof of Proposition~\ref{prop:mainbackward}.
We first prove Proposition~\ref{prop:mainbackward} by using the following proposition.

\begin{proposition}\label{prop:stronger}
    Let $G$ be a graph with a $3$-contractible tree decomposition $(T,(B_t)_{t\in V(T)})$, let~$\mathcal{A}$ be a set of $3$-vertex cliques in~$G$ such that $\{B_v\cap B_w:vw\in E(T),\abs{B_v\cap B_w}=3\}\subseteq\mathcal{A}$, and let~$H$ be a trigraph with underlying graph $\psi(\theta(G),\sigma_G(\mathcal{A}))$ such that $\DeltaR(H)\leq3$.    
    If~$\mathcal{A}$ is addable to $B_t$ in~$G$ for every $t\in V(T)$, then $\tww(H)\leq3$.
\end{proposition}

\begin{proof}[Proof of Proposition~\ref{prop:mainbackward} assuming Proposition~\ref{prop:stronger}]
    Since every graph in~$\mathcal{F}_3$ is connected, and the twin-width of a graph is the maximum over all the twin-width of its components, we may assume that $G$ is connected.
    By Proposition~\ref{prop:mainforward}, if~$G$ has a minor in $\mathcal{F}_3$, then $H$ has twin-width at least~$4$.
    Thus, we may assume that~$G$ is $\mathcal{F}_3$-minor-free.

    We now show that $\tww(H)\leq3$.
    By Proposition~\ref{prop:treedecomp}, there exists an $\mathcal{F}_3$-minor-free multigraph~$G'$ and a $3$-contractible tree decomposition $(T,(B_t)_{t\in V(T)})$ of~$G'$ such that~$G$ is a spanning subgraph of~$G'$.
    There exists a trigraph $H_1$ such that~$H$ is an induced subgraph of~$H_1$, $\DeltaR(H_1)=\DeltaR(H)\leq3$, and the underlying graph of~$H_1$ is an \subd{2} of~$G'$.
    To show that $\tww(H_1)\leq3$, it suffices to prove that $\tww(H_1)\leq3$.

    Let $G_0$ be the simplification of~$G'$, and let $\mathcal{A}:=\{B_v\cap B_w:vw\in E(T),\abs{B_v\cap B_w}=3\}$.
    Note that $(T,(B_t)_{t\in V(T)})$ is a $3$-contractible tree decomposition of~$G_0$, and that $\mathcal{A}$ is addable to $B_t$ in $G_0$ for every $t\in V(T)$.
    Since $\DeltaR(H)\leq3$, by Lemma~\ref{lem:theta}, there exists a partial \cont{3} from $H_1$ to a trigraph $H_2$ whose underlying graph is an induced subgraph of $\theta(G_0)$.
    
    There exists a trigraph $H_3$ with underlying graph $\psi(\theta(G_0),\sigma_{G_0}(\mathcal{A}))$ such that $\DeltaR(H_3)\leq3$ and $H_2$ is an induced subgraph of~$H_3$.
    By Proposition~\ref{prop:stronger} with~$\mathcal{A}$, $\tww(H_3)\leq3$, and therefore $\tww(H_2)\leq3$.
    Since there exists a partial \cont{3} from $H_1$ to $H_2$, we have that $\tww(H)\leq\tww(H_1)\leq3$.
\end{proof}

We now show Proposition~\ref{prop:stronger}.

\begin{proof}[Proof of Proposition~\ref{prop:stronger}]
    We proceed by induction on $\abs{V(T)}$.
    We may assume that $\abs{V(G)}\geq3$, because otherwise the statement trivially holds.
    
    Suppose first that $\abs{V(T)}=1$.
    Note that $G\in\mathcal{K}\setminus\{K_1,K_2\}$.
    Let $X$ be a clique of size~$2$ in $G$.
    By Corollary~\ref{cor:base}, there exists an $X$-stable partial \cont{3} from $H$ to a trigraph with underlying graph $\theta(G[X])$, which is a cycle.
    Hence, $\tww(H)\leq3$.

    Thus, we may assume that $\abs{V(T)}\geq2$.
    By the inductive hypothesis, we may assume that for every leaf~$\ell$ and its neighbour~$\ell'$, both $B_\ell\setminus B_{\ell'}$ and $B_{\ell'}\setminus B_\ell$ are nonempty.
    Let~$p$ be a leaf of~$T$, let $q$ be the neighbour of~$p$ in~$T$, and let $X:=B_p\cap B_q$.
    Note that $\abs{X}\geq1$ as $G$ is connected.
    Let $G':=G-(B_p\setminus X)$ and let $\mathcal{T}:=(T-p,(B_t)_{t\in V(T)\setminus\{p\}})$.

    \begin{claim}\label{clm:hypothesis}
        $\mathcal{T}$ is a $3$-contractible tree decomposition of $G'$.
    \end{claim}
    \begin{subproof}
        It is readily seen that $\mathcal{T}$ is a tree decomposition of $G'$ as $B_p$ is the unique bag of $(T,(B_t)_{t\in V(T)})$ containing the vertices in $B_p\setminus X$.
        Since $(T,(B_t)_{t\in V(T)})$ satisfies~\ref{def:td contractible vertex}, $\mathcal{T}$ also satisfies~\ref{def:td contractible vertex}.

        We now show that $\mathcal{T}$ satisfies~\ref{def:td contractible edge}.
        We may assume that $T-p$ has an edge, because otherwise we are done.
        Let $vw$ be an arbitrary edge of $T-p$, and let $S:=B_v\cap B_w$.
        Note that $S$ is disjoint from $B_p\setminus X$.
        Since $(T,(B_t)_{t\in V(T)})$ satisfies~\ref{def:td contractible edge}, $S$ is a minimal separator of~$G$ and a clique in~$G'$.
        Thus, to show that $\mathcal{T}$ satisfies~\ref{def:td contractible edge}, it suffices to show that $S$ is a minimal separator of~$G'$.

        Suppose that $S=X$.
        By~\ref{def:td vertex}, $q$ has a neighbour $p'$ in $T-p$ such that $S\subseteq B_{p'}\cap B_q$.
        Since $(T,(B_t)_{t\in V(T)})$ satisfies~\ref{def:td contractible edge}, $B_{p'}\cap B_q$ is a minimal separator of~$G$, so by the minimality of~$S$, we have that $B_{p'}\cap B_q=S$.
        We can choose a leaf $p''$ of $T-p$ such that $p'$ and $p''$ are in the same component of $T-\{p,q\}$.
        By the assumption on the leaves of~$T$, we have that $B_{p''}\setminus X\neq\emptyset$.
        Since $B_q\setminus X\neq\emptyset$ and $B_{p'}\cap B_q=X$, $G-S$ has at least three components where one is $G[B_p\setminus X]$, and therefore $S$ is a minimal separator of~$G'$.
        
        Thus, we may assume that $S\neq X$.
        Let $G_1,\ldots,G_\ell$ be the components of $G-S$.
        Since $X$ is a clique in~$G$, and~$S$ is disjoint from $B_p\setminus X$, there exists $i\in[\ell]$ such that $B_p\setminus X\subseteq V(G_i)$, and $G_i-(B_p\setminus X)$ is connected.
        By symmetry, we may assume that $i=1$.
        Since $S\neq X$ and $S$ is a minimal separator of~$G$, we have that $\emptyset\neq X\setminus S\subseteq V(G_1)\setminus(B_p\setminus X)$, and therefore~$S$ is a separator of $G'$.
        Note that $G_1-(B_p\setminus X),G_2,\ldots,G_\ell$ are the components of $G'-S$.
        Since $S$ is a minimal separator of~$G$, every vertex in $S$ has a neighbour in $V(G_j)$ for each $j\in[\ell]$.
        Since $N_G(B_p\setminus X)=X$ which is a clique, if a vertex in $S$ has a neighbour in $B_p\setminus X$, then it has a neighbour in $X\setminus S$.
        Thus, every vertex $S$ has a neighbour in $V(G_j)\setminus(B_p\setminus X)$ for every $j\in[\ell]$, and therefore $S$ is a minimal separator of~$G'$.
        Hence, $\mathcal{T}$ satisfies~\ref{def:td contractible edge}, and proves the claim.
    \end{subproof}

    Let $\mathcal{A}_p$ be the set of cliques $A\in\mathcal{A}$ with $A\subseteq B_p$ and let $\mathcal{A}'$ be the set of cliques $A\in\mathcal{A}$ with $A\cap(B_p\setminus X)=\emptyset$.
    Let $H_p$ be the induced subgraph of $H$ such that the underlying graph of~$H_p$ is $\psi(\theta(G[B_p]),\sigma_{G[B_p]}(\mathcal{A}_p))$.
    Note that $\mathcal{A}_p$ is an addable set of $G[B_p]$ as~$\mathcal{A}$ is addable to~$B_p$ in~$G$.
    
    Suppose that $\abs{X}=1$.
    For the vertex $u\in X$, let $u'$ and $v$ be neighbours of $u$ in~$G$ such that $u'\in B_p$ and $v\notin B_p$.
    Let~$v_1$, $v_2$, and~$v_3$ be the internal vertices of the unique length-$4$ path in $\theta(G)$ between $u$ and~$v$ such that $\{uv_1,vv_3\}\subseteq E(\theta(G))$.
    By Corollary~\ref{cor:base} applied to $H_p$ with $\mathcal{A}_p$ as the addable set, there exists a $\{u,u'\}$-stable partial \cont{3} from~$H_p$ to a trigraph $H'$ with underlying graph $\theta(G[\{u,u'\}])$.
    Note that $\theta(G[\{u,u'\}])$ is a cycle of length~$7$.
    By applying the same partial \cont{3}, we find a $\{u,u'\}$-stable partial \cont{3} from~$H$ to a trigraph which is the union of $H-(V(H_p)\setminus X)$ and~$H'$.
    We denote by $u_1,\ldots,u_6$ the vertices in $V(H')\setminus\{u\}$ such that for each $i\in[5]$, $u_i$ is adjacent to~$u_{i+1}$ in~$H'$.
    We then contract the following six pairs of vertices in order.
    \[
        \begin{array}{lll}
            1)\ \{u_1,u_6\}\text{ to }u_1,\text{\hspace{0.7cm}}
            &2)\ \{u_2,u_3\}\text{ to }u_2,\text{\hspace{0.7cm}}
            &3)\ \{u_2,u_4\}\text{ to }u_2,\\
            4)\ \{u_2,u_5\}\text{ to }u_2,
            &5)\ \{u_1,v_1\}\text{ to }v_1,
            &6)\ \{u_2,v_2\}\text{ to }v_2.
        \end{array}
    \]
    Thus, we find a partial \cont{3} from~$H$ to a trigraph with underlying graph $\psi(\theta(G'),\sigma_{G'}(\mathcal{A}'))$.
    By Claim~\ref{clm:hypothesis} and the inductive hypothesis, the resulting trigraph has twin-width at most~$3$, and therefore $\tww(H)\leq3$.
    
    Suppose that $\abs{X}=2$.
    By Corollary~\ref{cor:base} applied to $H_p$ with $\mathcal{A}_p$ as the addable set, there exists an $X$-stable partial \cont{3} from~$H$ to a trigraph with underlying graph $\psi(\theta(G'),\sigma_{G'}(\mathcal{A}'))$.
    By Claim~\ref{clm:hypothesis} and the inductive hypothesis, the resulting trigraph has twin-width at most~$3$, and therefore $\tww(H)\leq3$.
    
    Thus, we may assume that $\abs{X}=3$.
    By the assumption on~$\mathcal{A}$, $X$ is contained in~$\mathcal{A}$.
    Thus, by Corollary~\ref{cor:base} applied to $H_p$ with $\mathcal{A}_p$ as the addable set, there exists an $X$-stable partial \cont{3} from~$H$ to a trigraph~$H'$ with underlying graph $\psi(\theta(G'),\sigma_{G'}(\mathcal{A}'))-(V(C'_{G,X})\setminus X)$.
    There exists a trigraph $H''$ with underlying graph $\psi(\theta(G'),\sigma_{G'}(\mathcal{A}'))$ such that $H'$ is an induced subgraph of~$H''$ and $\DeltaR(H'')=\DeltaR(H')\leq3$.
    By Claim~\ref{clm:hypothesis} and the inductive hypothesis, we have that $\tww(H'')\leq3$, and therefore $\tww(H')\leq3$.
    Since there exists a partial \cont{3} from~$H$ to~$H'$, we have that $\tww(H)\leq3$, and this completes the proof by induction.
\end{proof}

We conclude the section by noting that 
Propositions~\ref{prop:treedecomp} and~\ref{prop:stronger}
imply the following structural characterisation of $\mathcal{F}_3$-minor-free graphs.

\begin{corollary}\label{cor:finalstructure}
    A graph is $\mathcal{F}_3$-minor-free if and only if it is a subgraph of some graph which admits a $3$-contractible tree decomposition.
    \qedhere
\end{corollary}

\section{Subdivisions of twin-width at most {\boldmath$2$}}\label{sec:2tree}

In this section, we complete the proof of Theorem~\ref{thm:main}.
We first prove Theorem~\ref{thm:main}\ref{main:atmost2}.
To do this, we actually show the following stronger theorem.

\begin{theorem}\label{thm:main2}
    Let $H$ be a trigraph whose underlying graph is an \subd{2} of a multigraph~$G$ such that $\rdeg_H(v)=0$ for every $v\in V(G)$.
    Then $\tww(H)\leq2$ if and only if $G$ has no $K_4$ as a minor.
\end{theorem}

The forward direction of Theorem~\ref{thm:main2} directly follows from Lemma~\ref{lem:detecting} and the following proposition.

\begin{proposition}\label{prop:segregated2}
    If a trigraph $H$ contains a $2$-\sgre{} of $K_4$, then $\tww(H)\geq3$.
\end{proposition}

We first prove Proposition~\ref{prop:segregated2} by using the following lemma.

\begin{lemma}\label{lem:segregated2}
    Let $H$ be a trigraph containing a $2$-\sgre{} of~$K_4$ such that no proper refined subgraph of~$H$ contains a $2$-\sgre{} of~$K_4$.
    For every pair $\{u,v\}$ of distinct vertices of~$H$, either $\max\{\DeltaR(H),\DeltaR(H/\{u,v\})\}\geq3$ or $H/\{u,v\}$ contains a $2$-\sgre{} of~$K_4$.
\end{lemma}

\begin{proof}[Proof of Proposition~\ref{prop:segregated2} assuming Lemma~\ref{lem:segregated2}]
    Suppose that $H$ is a counterexample to Proposition~\ref{prop:segregated2} which minimises $\abs{V(H)}+\abs{R(H)}$.
    Since the twin-width of a refined subgraph of~$H$ is no larger than the twin-width of~$H$, by the assumption, $H$ has no proper refined subgraph containing a $2$-\sgre{} of~$K_4$.
    Since $\tww(H)\leq2$ and $\abs{V(H)}\geq2$, there exists a pair $\{u,v\}$ of vertices of $H$ such that $H/\{u,v\}$ has twin-width at most $2$.
    Since $\max\{\DeltaR(H),\DeltaR(H/\{u,v\})\}\leq2$, by Lemma~\ref{lem:segregated2}, $H/\{u,v\}$ contains a $2$-\sgre{} of~$K_4$, so by the assumption on $H$, $\tww(H/\{u,v\})\geq3$, a contradiction.
\end{proof}

We now prove Lemma~\ref{lem:segregated2}.

\begin{proof}[Proof of Lemma~\ref{lem:segregated2}]
    Let $\mathcal{X}:=(X_s)_{s\in V(K_4)}$ be a $2$-\sgre{} of~$K_4$ in~$H$ such that $\abs{\bigcup\mathcal{X}}$ is minimised.
    Let~$u$ and~$v$ be distinct vertices of~$H$.
    Suppose that $\max\{\DeltaR(H),\DeltaR(H/\{u,v\})\}\leq2$.
    By Lemma~\ref{lem:minsgre}\ref{minsgre:delta}, the minimum degree of~$H$ is at least~$2$.
    Since $\DeltaR(H/\{u,v\})\leq2$, we have that $N_H[u]\cap N_H[v]$ is nonempty.
    We are going to show that $H/\{u,v\}$ contains a $2$-\sgre{} of~$K_4$.

    By the assumption and Lemma~\ref{lem:minsgre}, $H$ and $\mathcal{X}$ satisfy~\ref{minsgre:internal}--\ref{minsgre:deg2}.
    Let $\mathcal{P}$ be the set of $\mathcal{X}$-paths in~$H$, let $H':=H/\{u,v\}$, and let $x$ be the new vertex of $H/\{u,v\}$ obtained by contracting~$u$ and~$v$.
    We consider three cases depending on the size of $\{u,v\}\cap\bigcup\mathcal{X}$.
    Unless specified otherwise, whenever we verify~\ref{sgre:connected}--\ref{sgre:isomorphic} for a $2$-\sgre{} of~$K_4$ in~$H'$, the $2$-\sgre{} will be witnessed by~$H'$.

    \medskip
    \noindent\textbf{Case 1.} $\abs{\{u,v\}\cap\bigcup\mathcal{X}}=0$.
    
    Let $P_u$ and $P_v$ be the $\mathcal{X}$-paths in~$H$ containing $u$ and $v$, respectively.
    If $P_u=P_v$, then the refined subgraph $H'-(N_H(u)\cap N_H(v))$ of $H'$ witnesses that~$\mathcal{X}$ is a $2$-\sgre{} of~$K_4$ in~$H'$.
    Thus, we may assume that $P_u\neq P_v$.
    By~\ref{sgre:degree2}, $P_u$ and~$P_v$ are internally disjoint in~$H$, and $N[u]\cap N[v]\subseteq V(P_u)\cap V(P_v)$.
    Since $N_H[u]\cap N_H[v]\neq\emptyset$, $P_u$ and~$P_v$ share an end~$w$ which is a common neighbour of~$u$ and~$v$ in~$H$.
    By Lemma~\ref{lem:minsgre}\ref{minsgre:uniqueness}, $w$ is the unique common neighbour of~$u$ and~$v$ in~$H$.
    Note that $w\in\bigcup\mathcal{X}$.
    Let $p$ be the vertex of $K_4$ with $w\in X_p$.
    Let $X'_p:=X_p\cup\{x\}$ and for every $q\in V(K_4)\setminus\{p\}$, let $X'_q:=X_q$.
    
    We show that $\mathcal{X}':=(X'_s)_{s\in V(K_4)}$ is a $2$-\sgre{} of~$K_4$ in~$H'$.
    Since $x$ is adjacent to~$w$ in~$H'$, we have that $H'[X'_p]$ is connected, so $\mathcal{X}'$ satisfies~\ref{sgre:connected}.
    For each $y\in\{u,v\}$, let $P'_y:=H'[(V(P_y)\setminus\{y,w\})\cup\{x\}]$.
    Note that each $P'_y$ is an $\mathcal{X}'$-path in~$H'$ with at least one red edge.
    Since $H'-x=H-\{u,v\}$, we have that $(\mathcal{P}\setminus\{P_u,P_v\})\cup\{P'_u,P'_v\}$ is the set of $\mathcal{X}'$-paths in $H'$, and therefore $\mathcal{X}'$ satisfies~\ref{sgre:degree2}, \ref{sgre:length}, and~\ref{sgre:isomorphic}.
    Hence,~$\mathcal{X}'$ is a $2$-\sgre{} of~$K_4$ in~$H/\{u,v\}$.

    \medskip
    \noindent\textbf{Case 2.} $\abs{\{u,v\}\cap\bigcup\mathcal{X}}=1$.

    By symmetry, we may assume that $u\in\bigcup\mathcal{X}$.
    Let $p$ be the vertex of $K_4$ with $u\in X_p$ and let $P_v$ be the $\mathcal{X}$-path in~$H$ containing $v$.
    Suppose first that $P_v$ has an end $w\in X_p$.
    Let $W$ be the set of internal vertices of the subpath $P_v$ between $v$ and $w$.
    Let $X'_p:=(X_p\setminus\{u\})\cup W\cup\{x\}$, and for every $q\in V(K_4)\setminus\{p\}$, let~$X'_q:=X_q$.
    
    We show that $\mathcal{X}':=(X'_s)_{s\in V(K_4)}$ is a $2$-\sgre{} of $K_4$ in~$H'$.
    Since $H[X_p\cup W\cup\{v\}]$ is connected, $H'[X'_p]$ is also connected.
    Thus, $\mathcal{X}'$ satisfies~\ref{sgre:connected}.
    Let $P'_v:=H'[(V(P_v)\setminus(W\cup\{v,w\}))\cup\{x\}]$.
    Note that~$P'_v$ is an $\mathcal{X}'$-path in $H'$ with at least one red edge.
    Since $\deg_H(y)=2$ for every $y\in W$, $H'$ has no $\mathcal{X}'$-path having an end in $W$.
    Since $H'-x=H-\{u,v\}$, we have that $(\mathcal{P}\setminus\{P_v\})\cup\{P'_v\}$ is the set of $\mathcal{X}'$-paths in~$H'$, and therefore~$\mathcal{X}'$ satisfies~\ref{sgre:degree2}, \ref{sgre:length}, and~\ref{sgre:isomorphic}.
    Hence,~$\mathcal{X}'$ is a $2$-\sgre{} of~$K_4$ in~$H'$.

    Thus, we may assume that $P_v$ has no vertex in $X_p$.
    Since $u$ and $v$ are nonadjacent in~$H$ with $N_H[u]\cap N_H[v]\neq\emptyset$, they have a common neighbour in~$H$.
    Since $P_v$ has no vertex in $X_p$, every common neighbour of~$u$ and~$v$ in~$H$ is an end of~$P_v$.

    If $u$ and $v$ have a common neighbour $w$ in $H$, then let $q$ be the vertex of $K_4$ with $w\in X_q$.
    Note that $X_p\neq X_q$.
    Since~$u$ is adjacent to~$w$, $H[\{u,w\}]$ is an $\mathcal{X}$-path in~$H$ between $X_p$ and $X_q$, so by~\ref{sgre:length}, $uw$ is a red edge of~$H$.
    Hence, $\{xy:y\in N_H(u)\}\subseteq R(H')$.
    By Lemma~\ref{lem:minsgre}\ref{minsgre:deg2}, $u$ has degree at least~$3$ in~$H$, and therefore $\rdeg_{H'}(x)\geq3$, contradicting the assumption that $\DeltaR(H')\leq2$.

    \medskip
    \noindent\textbf{Case 3.} $\abs{\{u,v\}\cap\bigcup\mathcal{X}}=2$.

    Let $p$ and $q$ be the vertices of $K_4$ such that $u\in X_p$ and $v\in X_q$.
    Since contraction preserves the connectedness, if $X_p=X_q$, then it is straightforward to see that $H'$ contains a $2$-\sgre{} of~$K_4$.
    Thus, we may assume that $X_p\neq X_q$.

    Suppose first that $H$ has no $\mathcal{X}$-path between~$u$ and~$v$.
    Since $X_p\neq X_q$, $u$ and $v$ are nonadjacent in~$H$.
    By Lemma~\ref{lem:minsgre}\ref{minsgre:internal}, every common neighbour of~$u$ and~$v$ is in~$\bigcup\mathcal{X}$.
    Since $X_p\neq X_q$, by~\ref{sgre:length}, we have that $\{xy:y\in N_H(u)\cup N_H(v)\}\subseteq R(H')$.
    Since $\rdeg_{H'}(x)\leq2$, by Lemma~\ref{lem:minsgre}\ref{minsgre:delta}, each of~$u$ and~$v$ has degree~$2$ in~$H$.
    Since $X_p\neq X_q$, by Lemma~\ref{lem:minsgre}\ref{minsgre:deg2}, $u$ and $v$ have no common neighbour in~$H$, a contradiction.

    Thus, we may assume that $H$ has an $\mathcal{X}$-path $P$ between $u$ and $v$.
    By Lemma~\ref{lem:minsgre}\ref{minsgre:uniqueness}, $u$ and $v$ have no common neighbour in $X_p\cup X_q$.
    Since $u$ and $v$ are the ends of $P$, by Lemma~\ref{lem:minsgre}\ref{minsgre:delta} and~\ref{minsgre:deg2}, each of~$u$ and~$v$ has degree at least~$3$ in~$H$.
    Since $\rdeg_{H'}(x)\leq2$, $u$ and $v$ have at least one common neighbour in~$V(H)\setminus V(P)$.
    Let $N_{uv}:=(N_H(u)\cap N_H(v))\setminus V(P)$.

    By Lemma~\ref{lem:minsgre}\ref{minsgre:internal} and~\ref{minsgre:uniqueness}, $N_{uv}\subseteq\bigcup\mathcal{X}$.
    Since $u$ and $v$ have no common neighbour in $X_p\cup X_q$, for each $y\in N_{uv}$, the vertex~$z$ of~$K_4$ with $y\in X_z$ is neither~$p$ nor~$q$.
    Thus, for each $y\in N_{uv}$, both $H[\{u,y\}]$ and $H[\{v,y\}]$ are $\mathcal{X}$-paths in~$H$ of length~$1$, so by~\ref{sgre:length}, both $uy$ and $vy$ are red edges of $H$.
    This implies that $\{xy:y\in(N_H(u)\cup N_H(v))\setminus V(P)\}\subseteq R(H')$, and therefore both~$u$ and~$v$ have degree at most~$3$ in~$H$.

    Since both~$u$ and~$v$ has a neighbour not in $X_p\cup X_q$, by Lemma~\ref{lem:minsgre}\ref{minsgre:deg2}, they have degree~$3$ in~$H$, and therefore $(N_H(u)\cup N_H(v))\setminus V(P)=N_{uv}$ and $\abs{N_{uv}}=2$.
    Since $\rdeg_{H'}(x)\leq2$, we deduce that $P$ is of length at most $2$.
    By~\ref{sgre:length}, $P$ has at least one red edge, and therefore $u$ or $v$ has red degree~$3$ in~$H$, contradicting the assumption that $\DeltaR(H)\leq2$, and this completes the proof.
\end{proof}

To show the backward direction of Theorem~\ref{thm:main2}, we will use the following theorem and lemma.

\begin{theorem}[Wald and Colbourn~\cite{WC1983}]\label{thm:2tree}
    A graph is $K_4$-minor-free if and only if it is a partial $2$-tree.
\end{theorem}

\begin{lemma}\label{lem:2tree}
    Let $G$ be a $2$-tree and let $H$ be a trigraph with underlying graph~$\theta(G)$ such that $\rdeg_H(v)=0$ for every $v\in V(G)$.
    If $X$ is a clique of size~$2$ in~$G$, then there exists an $X$-stable partial \cont{2} from~$H$ to a trigraph with underlying graph $\theta(G[X])$.
\end{lemma}
\begin{proof}
    We proceed by induction on $\abs{V(G)}$.
    Let $u$ and $v$ be the vertices in $X$.
    
    Suppose first that $G$ is a triangle.
    Let $w$ be the vertex in $V(G)\setminus X$.
    For every edge $e=ab$ of~$G$, let $x_{a,b}$ and $x_{b,a}$ be the internal vertices of the unique length-$3$ path in $\theta(G)$ between~$a$ and~$b$ such that $\{ax_{a,b},bx_{b,a}\}\subseteq E(\theta(G))$, and let $y_{a,b}$, $y^e$, and $y_{b,a}$ be the internal vertices of the unique length-$4$ path in~$\theta(G)$ between $a$ and $b$ such that $\{ay_{a,b},by_{b,a}\}\subseteq E(\theta(G))$.
    We do the following in order to obtain a desired partial \cont{2} for~$H$.
    \begin{enumerate}[label=\bf{Step \arabic*.},leftmargin=*]
        \item For each $t\in\{u,v\}$, contract the following three pairs of vertices in order: $\{x_{t,w},y_{t,w}\}$ to $x_{t,w}$, $\{y^{tw},y_{w,t}\}$ to $y_{w,t}$, and $\{x_{w,t},y_{w,t}\}$ to $x_{w,t}$.
        \item Contract the following five pairs of vertices in order: $\{w,x_{w,u}\}$ to~$w$, $\{x_{u,w},y_{u,v}\}$ to~$y_{u,v}$, $\{w,x_{w,v}\}$ to~$w$, $\{x_{v,w},y_{v,u}\}$ to~$y_{v,u}$, and $\{w,y^{uv}\}$ to $y^{uv}$.
    \end{enumerate}

    Suppose now that $\abs{V(G)}\geq 4$ and that the result holds for all $2$-trees with fewer vertices.
    By definition of a $2$-tree, $G$ has a vertex~$z$ such that $G_1:=G[N_G[z]]$ is a triangle and $G_2:=G-z$ is a $2$-tree.
    Let us choose an integer $i\in[2]$ such that $X\subseteq V(G_i)$.
    By the inductive hypothesis, there exists an $N_G(z)$-stable partial \cont{2} from the induced subgraph of~$H$ with underlying graph $\theta(G_{3-i})$ to a trigraph with underlying graph $\theta(G[N_G(z)])$.
    Applying this sequence of contractions gives us a $V(G_i)$-stable partial \cont{2} from $H$ to a trigraph $H'$ with underlying graph $\theta(G_i)$.
    By the inductive hypothesis, there exists an $X$-stable partial \cont{2} from $H'$ to a trigraph with underlying graph~$\theta(G[X])$, which completes the $X$-stable partial \cont{2} from~$H$ to a trigraph with underlying graph~$\theta(G[X])$.
\end{proof}

\begin{proof}[Proof of Theorem~\ref{thm:main2}]
    The forward direction holds by Lemma~\ref{lem:detecting} and Proposition~\ref{prop:segregated2}.

    We now show the backward direction.
    Suppose that $G$ has no $K_4$ as a minor.
    If $\abs{V(G)}\leq2$, then by Lemma~\ref{lem:theta} with $\mu=0$, there exists a partial \cont{2} from~$H$ to a refined subgraph of a cycle.
    By iteratively contracting pairs of adjacent vertices in the cycle, until no such pairs remain, we obtain a \cont{2} of $H$.
    Thus, we may assume that $\abs{V(G)}\geq3$.

    Let~$G_0$ be the simplification of~$G$.
    By Theorem~\ref{thm:2tree}, there exists a $2$-tree $G'$ of which $G_0$ is a subgraph.
    Let~$G^*$ be the multigraph obtained from $G'$ by adding the edges $e\in E(G)\setminus E(G')$.
    There exists an \subd{2}~$H^*$ of~$G^*$ such that~$H$ is an induced subgraph of $H^*$, $\DeltaR(H^*)=\DeltaR(H)\leq2$, and $\rdeg_{H^*}(v)=0$ for every $v\in V(G^*)$.
    Since $H$ is an induced subgraph of~$H^*$, to show that $\tww(H)\leq2$, it suffices to show that $\tww(H^*)\leq2$.
    
    By Lemma~\ref{lem:theta} with $\mu=0$, there exists a partial \cont{2} from~$H^*$ to a trigraph $H^*_1$ whose underlying graph is an induced subgraph of~$\theta(G^*)$.
    There exists a trigraph $H^*_2$ with underlying graph~$\theta(G^*)$ such that $H^*_1$ is an induced subgraph of~$H^*_2$, $\DeltaR(H^*_2)=\DeltaR(H^*_1)\leq2$, and $\rdeg_{H^*_2}(v)=0$ for every $v\in V(G^*)$.
    By Lemma~\ref{lem:2tree}, there exists a partial \cont{2} from $H^*_2$ to a cycle.
    Since every cycle has twin-width at most~$2$, we have that $\tww(H^*_2)\leq2$, and therefore $\tww(H^*_1)\leq2$.
    Since there exists a partial \cont{2} from $H^*$ to $H^*_1$ we have that $\tww(H)\leq\tww(H^*)\leq2$.
\end{proof}

We now complete the proof of Theorem~\ref{thm:main}.

\begin{proof}[Proof of Theorem~\ref{thm:main}]
    Proposition~\ref{prop:mainbackward} implies~\ref{main:atmost3}.
    Theorem~\ref{thm:main2} implies~\ref{main:atmost2}.
    
    We show~\ref{main:atmost1}.
    The forward direction follows from Lemma~\ref{lem:cycle} and that the $1$-subdivision of $K_{1,3}$~{\cite[Lemma~6.6]{AHKO2021}} has twin-width~$2$.
    Suppose that $G$ has no minor in $\{K_{1,3},C_1\}$.
    Since $G$ has no~$C_1$ as a minor, $G$ is a forest.
    Since $G$ has no~$K_{1,3}$ as a minor, every component of~$G$ is a path.
    Thus, every component of~$H$ is a path, and therefore $\tww(H)\leq1$.
    
    We show~\ref{main:atmost0}.
    The forward direction follows from Lemma~\ref{lem:cycle} and that an induced path of length~$3$ has twin-width at least~$1$.
    If $G$ has no minor in $\{K_2,C_1\}$, then $G$ has no edge.
    Thus, $H$ has no edge, and therefore $\tww(H)=0$.
\end{proof}

\subsection{Subdivisions of {\boldmath$K_4$} with large girth}

In this subsection, we investigate the twin-width of subdivisions of~$K_4$ with large girth.

\begin{theorem}\label{thm:girth}
    Let $H$ be a trigraph.
    If the underlying graph of $H$ has an induced subgraph of girth at least~$7$ which is a subdivision of $K_4$, then $\tww(H)\geq3$.
\end{theorem}

We first point out that Proposition~\ref{prop:segregated2} and Theorem~\ref{thm:girth} are incomparable.
If a subdivision $H$ of $K_4$ has a path of length at most $2$ between degree-$3$ vertices, then it contains no $2$-\sgre{} of $K_4$, so Proposition~\ref{prop:segregated2} tells us nothing about the twin-width of $H$.
However, if $H$ is of girth at least~$7$, then Theorem~\ref{thm:girth} shows that $H$ has twin-width exactly~$3$.
In the other direction, the line graph of an \subd{3} of $K_4$ has a $2$-\sgre{} by Lemma~\ref{lem:detecting} and so has twin-width at least~$3$ by Proposition~\ref{prop:segregated2}, while it does not have an induced subgraph of girth at least~$7$ which is a subdivision of~$K_4$.
We will later show that the girth condition in Theorem~\ref{thm:girth} is tight.

Theorem~\ref{thm:girth} directly follows from the following proposition.

\begin{proposition}\label{prop:girth}
    Let $H$ be a trigraph whose underlying graph is a subdivision of $K_4$.
    Let $S$ be the set of vertices $v$ of $H$ with $\rdeg_H(v)=\deg_H(v)=2$.
    For each cycle $C$ of $H$, let $q_H(C)$ be the number of vertices~$v$ of~$C$ such that~$v$ has degree~$3$ in~$H$ and red degree~$2$ in~$C$.
    If $\abs{V(C)}+2q_H(C)\geq7$ for every cycle~$C$ of~$H-S$, then $\tww(H)=3$.
\end{proposition}
\begin{proof}
    Since every subdivision of $K_4$ admits a contraction sequence keeping the maximum degree at most~$3$, we have that $\tww(H)\leq3$.
    We write $q$ for $q_H$ for convenience.
    
    We now show that $\tww(H)\geq3$.
    Suppose that $H$ is a counterexample to Proposition~\ref{prop:girth} which minimises $\abs{V(H)}$.
    Let~$\mathcal{P}$ be the set of paths $P$ in~$H$ such that each end of $P$ has degree $3$ in~$H$ and each internal vertex of~$P$ has degree~$2$ in~$H$.
    Note that every cycle of $H$ contains at least three vertices of degree~$3$ in~$H$.
    
    We begin with a simple observation.

    \begin{claim}\label{clm:3cycle}
        If $C$ is a cycle in $H-S$ of length at most~$4$, then $C$ has at least three red edges.
    \end{claim}
    \begin{subproof}
        Since $\abs{V(C)}+2q(C)\geq7$, we have that $q(C)\geq2$.
        Thus, $C$ has at least two vertices having red degree~$2$ in~$C$, and therefore $C$ has at least three red edges.
    \end{subproof}

    Claim~\ref{clm:3cycle} implies that every cycle in~$H$ of length at most~$4$ contains at least two red edges.
    Since $\tww(H)\leq2$ and $\abs{V(H)}\geq4$, there exists a pair $\{u,v\}$ of distinct vertices of~$H$ such that the trigraph $H'$ obtained from $H$ by contracting $\{u,v\}$ to~$u$ has twin-width at most~$2$.
    We consider the following three cases by considering $\deg_H(u)+\deg_H(v)$.

    \medskip
    \noindent\textbf{Case 1.} $\deg_H(u)=\deg_H(v)=3$.

    Since $\rdeg_{H'}(u)\leq2$, $u$ and~$v$ have at least one common neighbour in~$H$.
    If $u$ and $v$ are adjacent, then by Claim~\ref{clm:3cycle}, $uv$ is a red edge of~$H$ and every common neighbour of~$u$ and~$v$ in~$H$ is a common red neighbour.
    Since $\rdeg_H(u)\leq2$, $u$ and~$v$ have exactly one common neighbour in~$H$, and therefore $\rdeg_{H'}(u)=\deg_H(u)+\deg_H(v)-2-\abs{N_H(u)\cap N_H(v)}=3$, a contradiction.
    
    Thus, we may assume that~$u$ and~$v$ are nonadjacent.
    Since $\deg_H(u)+\deg_H(v)-2\abs{N_H(u)\cap N_H(v)}\leq\rdeg_{H'}(u)\leq2$, $u$ and~$v$ have at least two common neighbours~$w_1$ and~$w_2$ in~$H$.
    By Claim~\ref{clm:3cycle}, at most one of~$w_1$ and~$w_2$ is a common black neighbour of~$u$ and~$v$ in~$H$.
    Therefore, from $\deg_H(u)=\deg_H(v)=3$, $u$ and~$v$ have three common neighbours in~$H$.
    Let~$w_3$ be the other common neighbour of~$u$ and~$v$ in~$H$.
    Since $\rdeg_{H'}(u)\leq2$, there exists $i\in[3]$ such that~$w_i$ is a common black neighbour of~$u$ and~$v$ in~$H$.
    By symmetry, we may assume that $i=1$.
    Since $H$ has a cycle with vertex set $\{u,v,w_2,w_3\}$, at least one of~$w_2$ and~$w_3$, say~$w_2$, is a vertex of $K_4$.
    Then the cycle~$C$ of~$H$ with vertex set $\{u,v,w_1,w_2\}$ is a cycle of $H-S$ with $q(C)\leq1$, a contradiction.

    \medskip
    \noindent\textbf{Case 2.} $\deg_H(u)+\deg_H(v)=5$.
    
    By symmetry, we may assume that $\deg_H(v)=2$.
    Since $\rdeg_{H'}(u)\leq2$, if~$u$ and~$v$ are nonadjacent, then~$u$ and~$v$ have exactly two common neighbours $w_1$ and $w_2$ in~$H$, and one of $w_1$ and $w_2$ is a common black neighbour of $u$ and $v$ in $H$.
    Since $\deg_H(v)=2$, all of $u$, $w_1$, and $w_2$ have degree $3$ in $H$, and therefore the cycle $C$ of $H$ with vertex set $\{u,v,w_1,w_2\}$ is a cycle of $H-S$ with $q(C)\leq1$, a contradiction.
    
    Thus, we may assume that~$u$ and~$v$ are adjacent.
    Since ${\rdeg_{H'}(u)\leq2}$, $u$ and~$v$ have at least one common neighbour in~$H$.
    Hence, $H$ has a triangle containing both~$u$ and~$v$, and therefore $v$ has to have degree~$3$, contradicting the assumption that $\deg_H(v)=2$.

    \medskip
    \noindent\textbf{Case 3.} $\deg_H(u)=\deg_H(v)=2$.

    Let $P_u$ and $P_v$ be the paths in $\mathcal{P}$ containing $u$ and $v$, respectively.
    We first consider the case that $P_u=P_v$.
    Let $w_1$ and $w_2$ be the ends of $P_u$.
    Since $\rdeg_{H'}(u)\leq2$, $u$ and $v$ either are adjacent in~$H$, or have a common black neighbour in~$H$.
    Let $H'':=H'-(N_H(u)\cap N_H(v))$.
    Note that the underlying graph of~$H''$ is a subdivision of~$K_4$.
    Let~$S''$ be the set of vertices~$w$ of~$H''$ such that $\rdeg_{H''}(w)=\deg_{H''}(w)=2$.
    Note that $(S\cup\{u\})\setminus\{v\}\subseteq S''\subseteq S\cup N_{H'}[u]$.
    Thus, no cycle of $H''-S''$ contains an internal vertex of~$P_u$.
    Since $H''-(V(P_u)\setminus\{w_1,w_2\})=H-(V(P_u)\setminus\{w_1,w_2\})$, the following hold for every cycle $C$ of $H''-S''$.
    \begin{itemize}
        \item $C$ is a cycle of $H-S$.
        \item For every vertex $v$ of $C$, $\deg_{H''}(v)=\deg_{H}(v)$.
        \item For every edge $e$ of $C$, $e$ is a red edge of $H''$ if and only if $e$ is a red edge of $H$.
    \end{itemize}
    Therefore, $\abs{V(C)}+2q_{H''}(C)=\abs{V(C)}+2q_H(C)\geq7$.
    Since $\abs{V(H'')}<\abs{V(H)}$, by the assumption on $H$, $\tww(H'')\geq3$, contradicting the assumption that $\tww(H')\leq2$, because $H''$ is an induced subgraph of~$H'$.

    Hence, $P_u\neq P_v$, and thus $u$ and $v$ are nonadjacent.
    Since $\rdeg_{H'}(u)\leq2$, they have a common black neighbour~$w$ in~$H$, which is a common end of~$P_u$ and~$P_v$.
    We remark that $w$ is the unique common neighbour of~$u$ and~$v$ in~$H$, because otherwise the paths $P_u$ and $P_v$ form a cycle containing only two vertices of degree~$3$ in~$H$.
    Observe that the underlying graph of~$H'$ is a subdivision of $K_4$ where $u$ has degree $3$.
    Let $u'$ and $v'$ be the neighbours in~$H$ of~$u$ and~$v$ other than~$w$, respectively.
    Note that~$u'$ and~$v'$ are distinct.
    For every vertex~$x$ of~$H'$, we observe the following.
    \begin{enumerate}[label=(\roman*)]
        \item\label{cond:deg1} $\deg_{H'}(x)\neq\deg_H(x)$ if and only if $x\in\{u,w\}$.
        \item\label{cond:deg2} If $x\notin\{u,u',v'\}$, then $\rdeg_H(x)=\rdeg_{H'}(x)$.
        \item\label{cond:deg2'} If $x\in\{u',v'\}$, then $\rdeg_H(x)\leq\rdeg_{H'}(x)\leq\rdeg_H(x)+1$.
    \end{enumerate}

    Let $S'$ be the set of vertices $x$ of $H'$ such that $\rdeg_{H'}(x)=\deg_{H'}(x)=2$.
    
    \begin{claim}\label{clm:S'}
        $S\subseteq S'\subseteq S\cup\{u',v'\}$.
    \end{claim}
    \begin{subproof}
        Note that neither~$u$ nor~$v$ is in~$S$, because~$w$ is a common black neighbour of~$u$ and~$v$ in~$H$.
        Since $\deg_{H'}(u)=3$, we have that $u\notin S'$.
        As $uw\in B(H')$, we deduce that $w\notin S'$.
        By~\ref{cond:deg1} and~\ref{cond:deg2}, for every $z\in V(H')\setminus N_{H'}[u]$, we have that $\deg_{H'}(z)=\deg_H(z)$ and $\rdeg_{H'}(z)=\rdeg_H(z)$, and therefore $z\in S$ if and only if $z\in S'$.
        By~\ref{cond:deg2'}, if $y\in S\cap\{u',v'\}$, then $y\in S'$, and this proves the claim.
    \end{subproof}

    Let $C$ be an arbitrary cycle of $H'-S'$.
    To derive a contradiction, we now verify that $\abs{V(C)}+2q_{H'}(C)\geq7$.
    If $C$ does not contain $u$, then by Claim~\ref{clm:S'}, $C$ is a cycle of $H-S$.
    Since the edges of $C$ have the same colour in~$H$ and~$H'$, by~\ref{cond:deg1}, $q_{H'}(C)=q_H(C)$.
    Therefore, $\abs{V(C)}+2q_{H'}(C)=\abs{V(C)}+2q_H(C)\geq7$.
    Hence, we may assume that $C$ contains $u$.
    Note that~$C$ contains exactly two of~$u'$, $v'$, and~$w$.
    
    Suppose first that $\{u',v'\}\subseteq V(C)$.
    Let $D$ be the cycle of~$H$ with vertex set $V(C)\cup\{v,w\}$.
    We show that~$D$ is a cycle of $H-S$.
    Since $\{u',v'\}\subseteq V(C)$ and $V(C)\cap S'=\emptyset$, by Claim~\ref{clm:S'}, neither $u'$ nor~$v'$ is in~$S$.
    Since $w$ is a common black neighbour of~$u$ and~$v$, none of $u$, $v$, and $w$ are in $S$.
    Therefore,~$D$ is a cycle of $H-S$.
    Since $\deg_{H'}(u)=3$ and $\rdeg_{H'}(u)=\rdeg_C(u)=2$, by~\ref{cond:deg1}, $q_{H'}(C)\geq q_H(D)+1$, so
    \[
        \abs{V(C)}+2q_{H'}(C)\geq(\abs{V(D)}-2)+2(q_H(D)+1)=\abs{V(D)}+2q_H(D)\geq7.
    \]

    Suppose now that $\{u',w\}\subseteq V(C)$.
    By Claim~\ref{clm:S'}, $C$ is a cycle of $H-S$.
    Note that for every $e\in E(C)\setminus\{uu'\}$, $e\in R(H)$ if and only if $e\in R(H')$.
    Since $uw$ is a black edge of~$H$, neither~$u$ nor~$v$ contributes $q_H(C)$, and therefore $q_{H'}(C)\geq q_H(C)$.
    Thus, $\abs{V(C)}+2q_{H'}(C)\geq\abs{V(C)}+2q_H(C)\geq7$.

    We finally assume that $\{v',w\}\subseteq V(C)$.
    Let $D$ be the cycle of~$H$ with vertex set $(V(C)\setminus\{u\})\cup\{v\}$.
    We show that $D$ is a cycle of~$H-S$.
    Since $C$ is a cycle of $H'-S'$ and $v'\in V(C)$, by Claim~\ref{clm:S'}, $S$ does not contain~$v'$.
    Since $\rdeg_H(v)\leq1$ and $\deg_H(w)=3$, neither $v$ nor $w$ lies in $S$, and therefore by Claim~\ref{clm:S'},~$D$ is a cycle of $H-S$.
    Since $vw$ is a black edge of~$D$, neither~$v$ nor~$w$ contributes to $q_H(D)$.
    If~$v'$ contributes to $q_H(D)$, then it also contributes to $q_{H'}(C)$.
    Hence, we deduce that $q_{H'}(C)\geq q_H(D)$, and therefore $\abs{V(C)}+2q_{H'}(C)\geq\abs{V(D)}+2q_H(D)\geq7$.

    Thus, for every cycle $C$ of $H'-S'$, $\abs{V(C)}+2q_{H'}(C)\geq7$.
    Since $\abs{V(H')}<\abs{V(H)}$, by the assumption on $H$, $\tww(H')\geq3$, contradicting the assumption that $\tww(H')\leq2$.
    Hence, $\tww(H)\geq3$.
\end{proof}

We now show that the girth condition in Theorem~\ref{thm:girth} is tight by presenting infinitely many subdivisions of~$K_4$ which have girth $6$ and twin-width~$2$.

\begin{proposition}
    Let $G$ be a graph obtained from $K_4$ by subdividing one edge of $K_4$ at least once, and each of the other edges of $K_4$ exactly once.
    Then $\tww(G)=2$.
\end{proposition}
\begin{proof}
    Let $v_1,\ldots,v_4$ be the vertices of~$K_4$.
    We may assume that every edge in $E(K_4)\setminus\{v_3v_4\}$ is subdivided exactly once.
    For $1\leq i<j\leq4$ with $\{i,j\}\neq\{3,4\}$, we denote by $v_{i,j}$ the vertex of $G$ with neighbourhood $\{v_i,v_j\}$.
    Since $G$ has an induced cycle of length $6$, by Lemma~\ref{lem:cycle}, $\tww(G)\geq2$.
    To show that $\tww(G)\leq2$, we do the following in order.
    \begin{enumerate}[label=\bf{Step \arabic*.},leftmargin=*]
        \item Contract the following four pairs of vertices in order: $\{v_{1,3},v_{2,3}\}$ to~$v_{1,3}$, $\{v_{1,4},v_{2,4}\}$ to~$v_{1,4}$, $\{v_1,v_2\}$ to~$v_1$, and $\{v_1,v_{1,2}\}$ to~$v_1$.
        After Step 1, the underlying graph becomes a cycle.
        \item Iteratively contract pairs of adjacent vertices, until no such pairs remain.
    \end{enumerate}    
    It is straightforward to see that the resulting contraction sequence is a \cont{2} of $G$.
\end{proof}

As we have seen in Theorem~\ref{thm:girth} for $K_4$, one may expect that subdivisions of the graphs in~$\mathcal{F}_3$ has twin-width at least~$4$ if it has large girth.
However, the following proposition shows that this is not the case for~$Q_3$ by presenting subdivisions of~$Q_3$ having arbitrarily large girth and twin-width at most~$3$.

\begin{proposition}
    Let $e$, $f$ be two edges incident with a fixed vertex of~$Q_3$.
    If~$G$ is a graph obtained from~$Q_3$ by subdividing each edge other than~$e$ and~$f$ at least~$0$ times, then $\tww(G)\leq3$.
    The equality holds if~$G$ has a vertex~$x$ such that $G-x$ is of girth at least~$7$.
\end{proposition}
\begin{proof}
    We label the vertices of $Q_3$ as in Figure~\ref{fig:forbidden}.
    Without loss of generality, we may assume that $e:=v_1v_5$ and $f:=v_5v_8$.
    Note that every vertex of degree-$2$ has a neighbour not incident with~$e$ or~$f$.
    Thus, we iteratively contract a pair of degree-$2$ vertex and its neighbour not incident with~$e$ or~$f$
    to obtain a trigraph whose underlying graph is~$Q_3$.
    We remark that in the resulting trigraph, both~$e$ and~$f$ are black edges.
    We then contract the following four pairs of vertices in order: $\{v_1,v_8\}$, $\{v_5,v_7\}$, $\{v_3,v_6\}$, and $\{v_4,v_7\}$.
    Afterwards, the underlying graph is~$C_4$.
    Thus, we apply an arbitrary contraction sequence of~$C_4$.
    It is clear to check that the resulting sequence is a \cont{3} of~$G$.

    Suppose that $G$ has a vertex $x$ such that~$G-x$ is of girth~$7$.
    Since $G-x$ has an induced subgraph which is a subdivision of $K_4$, by Theorem~\ref{thm:girth}, $G$ has twin-width~$3$.
\end{proof}

\section{Twin-width of grids, walls, and the line graphs of walls}\label{sec:grid}

We investigate the twin-width of grids, walls, and the line graphs of walls.
We use the labellings of grids and walls introduced in Section~\ref{sec:intro}.

\begin{proposition}\label{prop:grid1}
    For a positive integer $n$, the $2\times n$ grid has twin-width~$2$ if and only if $n\geq4$.
\end{proposition}
\begin{proof}
    Let $G_n$ be a trigraph whose underlying graph is the $2\times n$ grid 
    such that $(1,n-1)(1,n)$ and $(1,n)(2,n)$ are the only red edges.
    Since~$G_2$ has twin-width~$2$ and $G_{n-1}$ is obtained from $G_n$ by contracting $\{(1,n-1),(2,n)\}$ to~$(1,n-1)$ and then contracting $\{(1,n-1),(1,n)\}$ to~$(1,n-1)$, it follows that for every integer $n\ge2$, $G_n$ has twin-width at most~$2$.
    This implies that the $2\times n$ grid has twin-width at most~$2$ for all integers $n\ge 2$.

    It is easy to see that the $2\times3$ grid has twin-width~$1$, and the $2\times4$ grid has twin-width~$2$.
\end{proof}

\begin{proposition}\label{prop:grid2}
    For positive integers $m$ and $n$ with $3\leq m\leq5$, the $m\times n$ grid has twin-width~$3$ if and only if $n\geq8-m$.
\end{proposition}
\begin{proof}
    We first show that $5\times n$ grid has twin-width at most~$3$.
    For all integers $n\ge2$,
    let $G_n$ be a trigraph whose underlying graph is the $5\times n$ grid 
    such that $G_n$ has exactly $5$ red edges $(2,n)(2,n-1)$, 
    $(3,n)(3,n-1)$, $(4,n)(4,n-1)$, $(2,n)(3,n)$, and $(3,n)(4,n)$.
    Let $G_1$ be a trigraph 
    whose underlying graph is the $5\times 1$ grid 
    such that $G_1$ has exactly $2$ red edges $(2,1)(3,1)$, and $(3,1)(4,1)$.
    Observe that for all integers $n\ge2$, $G_{n-1}$ is obtained from $G_n$ by contracting the following five pairs of vertices in order: 
    $\{(1,n),(2,n-1)\}$ to~$(2,n-1)$,
    $\{(5,n),(4,n-1)\}$ to~$(4,n-1)$,
    $\{(2,n),(3,n)\}$ to~$(3,n)$,
    $\{(3,n),(4,n)\}$ to~$(3,n)$,
    and $\{(3,n-1),(3,n)\}$ to~$(3,n-1)$.
    Since $G_1$ has twin-width at most~$3$, we deduce that 
    $G_n$ has twin-width at most $3$ and therefore the $5\times n$ grid has twin-width at most~$3$.

    Using the method of Schidler and Szeider~\cite{SS2021}, Schidler~\cite{schidlerprivate} verified that both the $3\times5$ grid and the $4\times4$ grid have twin-width~$3$.
    Therefore, if $n\geq8-m$, then the $m\times n$ grid has twin-width exactly~$3$.

    Suppose now that $n\leq7-m$.
    By Proposition~\ref{prop:grid1}, we may assume that $m\leq4$.
    Since the $m\times n$ grid is an induced subgraph of the $3\times4$ grid, it suffices to show that the $3\times4$ grid has twin-width at most~$2$.
    We present a partial \cont{2} of the $3\times4$ grid to a cycle as follows.
    \[
        \begin{array}{lll}
            1)\ \{(1,1),(2,2)\}\text{ to }(2,2),\text{\hspace{0.7cm}}
            &2)\ \{(2,1),(3,2)\}\text{ to }(3,2),\text{\hspace{0.7cm}}
            &3)\ \{(1,4),(2,3)\}\text{ to }(2,3),\\
            4)\ \{(2,4),(3,3)\}\text{ to }(3,3),
            &5)\ \{(3,1),(3,2)\}\text{ to }(3,2),
            &6)\ \{(3,3),(3,4)\}\text{ to }(3,3),\\
            7)\ \{(1,2),(2,3)\}\text{ to }(2,3),
            &8)\ \{(1,3),(2,3)\}\text{ to }(1,3).
        \end{array}
    \]
    Thus, the $3\times4$ grid has twin-width at most~$2$, and this completes the proof.
\end{proof}

We show that every wall has twin-width at most~$4$.

\begin{lemma}\label{lem:gridgrid}
    For all integers $m,n\geq1$, every trigraph whose underlying graph is a subdivision of the $m\times n$ grid has twin-width at most~$4$.
\end{lemma}
\begin{proof}
    Let $H_{m,n}$ be a trigraph without black edges whose underlying graph is the $m\times n$ grid with $m\leq n$.
    It is easy to see that if $G$ is a trigraph without black edges whose underlying graph is a subdivision of the $m\times n$ grid, then there is a partial \cont{4} from~$G$ to~$H_{m,n}$,
    simply by iteratively contracting pairs of a vertex of degree $2$ and its neighbour.
    Thus, it suffices to show that $\tww(H_{m,n})\leq4$.

    We now proceed by induction on~$m$ to show that $\tww(H_{m,n})\leq4$.
    If $m=1$ or $n=1$, then the maximum degree of~$H_{m,n}$ is at most~$2$, so it has twin-width at most~$2$.
    Thus, we may assume that both $m$ and $n$ are at least~$2$.
    For~$j$ from~$1$ to~$n-1$, we contract $\{(m-1,j+1),(m,j)\}$ to~$(m-1,j+1)$ to obtain a partial \cont{4} from $H_{m,n}$ to~$H_{m,n-1}$.
    Thus, by the inductive hypothesis, $\tww(H_{m,n})\leq4$.
    This completes the proof by induction.
\end{proof}

\begin{proposition}\label{prop:wall}
    For all integers $m,n\geq1$, every trigraph whose underlying graph is an $m\times n$ wall has twin-width at most~$4$.
\end{proposition}
\begin{proof}
    If the underlying graph of a trigraph $H$ is an $m\times n$ wall, then it is a refined subgraph of a trigraph whose underlying graph is a subdivision of the $m\times n$ grid, which has twin-width at most~$4$ by Lemma~\ref{lem:gridgrid}.
    Thus, by Lemma~\ref{lem:refinedtww}, $\tww(H)\leq4$.
\end{proof}

We now show that the line graph of a wall has twin-width at most~$4$.

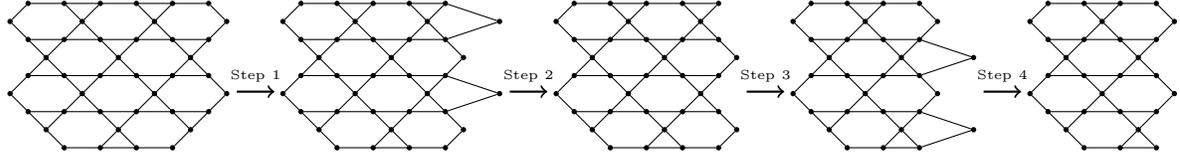
\begin{figure}[t]
    \centering
    \tikzstyle{b}=[circle, draw=black!70, solid, fill=black!70, inner sep=0pt, minimum width=5.5pt]
    \tikzstyle{v}=[circle, draw, solid, fill=black, inner sep=0pt, minimum width=1.6pt]
    \begin{tikzpicture}[scale=0.48]
        \draw(0,6.6) node(){};
        \draw(0,0.4) node(){};
        \foreach \x in {1,2,3,4,5,6}{
            \draw (\x+0.5,5) node[v](5h\x){};
            \draw (\x+0.5,4) node[v](4h\x){};
            \draw (\x+0.5,3) node[v](3h\x){};
            \draw (\x+0.5,2) node[v](2h\x){};
        }
        \foreach \x in {2,3,4,5}{
            \draw (\x+0.5,1) node[v](1h\x){};
        }
        \foreach \x in {1,3,5,7}{
            \draw (\x,4.5) node[v](4v\x){};
            \draw (\x,2.5) node[v](2v\x){};
        }
        \foreach \x in {2,4,6}{
            \draw (\x,3.5) node[v](3v\x){};
            \draw (\x,1.5) node[v](1v\x){};
        }
        
        \draw (5h1)--(5h6);
        \draw (4h1)--(4h6);
        \draw (3h1)--(3h6);
        \draw (2h1)--(2h6);
        \draw (1h2)--(1h5);

        \draw (2v1)--(1h2);
        \draw (4v1)--(1h4);
        \draw (5h2)--(1v6);
        \draw (5h4)--(2v7);
        \draw (5h6)--(4v7);
        
        \draw (5h1)--(4v1);
        \draw (5h3)--(2v1);
        \draw (5h5)--(1v2);
        \draw (4v7)--(1h3);
        \draw (2v7)--(1h5);
    \end{tikzpicture}\begin{tikzpicture}[overlay]
        \draw[->,thick] (.1,1.15) -- ++(0.5,0) node[midway,above] {\tiny Step 1};
    \end{tikzpicture}
    \begin{tikzpicture}[scale=0.48]
        \draw(0,6.6) node(){};
        \draw(0,0.4) node(){};
        \foreach \x in {1,2,3,4,5}{
            \draw (\x+0.5,5) node[v](5h\x){};
            \draw (\x+0.5,4) node[v](4h\x){};
            \draw (\x+0.5,3) node[v](3h\x){};
            \draw (\x+0.5,2) node[v](2h\x){};
        }
        \foreach \x in {2,3,4,5}{
            \draw (\x+0.5,1) node[v](1h\x){};
        }
        \foreach \x in {1,3,5,7}{
            \draw (\x,4.5) node[v](4v\x){};
            \draw (\x,2.5) node[v](2v\x){};
        }
        \foreach \x in {2,4,6}{
            \draw (\x,3.5) node[v](3v\x){};
            \draw (\x,1.5) node[v](1v\x){};
        }
        
        \draw (5h1)--(5h5);
        \draw (4h1)--(4h5);
        \draw (3h1)--(3h5);
        \draw (2h1)--(2h5);
        \draw (1h2)--(1h5);
        
        \draw (2v1)--(1h2);
        \draw (4v1)--(1h4);
        \draw (5h2)--(1v6);
        \draw (5h4)--(3v6);
        
        \draw (5h1)--(4v1);
        \draw (5h3)--(2v1);
        \draw (5h5)--(1v2);
        \draw (3v6)--(1h3);
        \draw (1v6)--(1h5);
        
        \draw (4h5)--(4v7)--(5h5);
        \draw (2h5)--(2v7)--(3h5);
    \end{tikzpicture}\begin{tikzpicture}[overlay]
        \draw[->,thick] (.1,1.15) -- ++(0.5,0) node[midway,above] {\tiny Step 2};
    \end{tikzpicture}
    \begin{tikzpicture}[scale=0.48]
        \draw(0,6.6) node(){};
        \draw(0,0.4) node(){};
        \foreach \x in {1,2,3,4,5}{
            \draw (\x+0.5,5) node[v](5h\x){};
            \draw (\x+0.5,4) node[v](4h\x){};
            \draw (\x+0.5,3) node[v](3h\x){};
            \draw (\x+0.5,2) node[v](2h\x){};
        }
        \foreach \x in {2,3,4,5}{
            \draw (\x+0.5,1) node[v](1h\x){};
        }
        \foreach \x in {1,3,5}{
            \draw (\x,4.5) node[v](4v\x){};
            \draw (\x,2.5) node[v](2v\x){};
        }
        \foreach \x in {2,4,6}{
            \draw (\x,3.5) node[v](3v\x){};
            \draw (\x,1.5) node[v](1v\x){};
        }
        
        \draw (5h1)--(5h5);
        \draw (4h1)--(4h5);
        \draw (3h1)--(3h5);
        \draw (2h1)--(2h5);
        \draw (1h2)--(1h5);
        
        \draw (2v1)--(1h2);
        \draw (4v1)--(1h4);
        \draw (5h2)--(1v6);
        \draw (5h4)--(3v6);
        
        \draw (5h1)--(4v1);
        \draw (5h3)--(2v1);
        \draw (5h5)--(1v2);
        \draw (3v6)--(1h3);
        \draw (1v6)--(1h5);
    \end{tikzpicture}\begin{tikzpicture}[overlay]
        \draw[->,thick] (.1,1.15) -- ++(0.5,0) node[midway,above] {\tiny Step 3};
    \end{tikzpicture}
    \begin{tikzpicture}[scale=0.48]
        \draw(0,6.6) node(){};
        \draw(0,0.4) node(){};
        \foreach \x in {1,2,3,4}{
            \draw (\x+0.5,5) node[v](5h\x){};
            \draw (\x+0.5,4) node[v](4h\x){};
            \draw (\x+0.5,3) node[v](3h\x){};
            \draw (\x+0.5,2) node[v](2h\x){};
        }
        \foreach \x in {2,3,4}{
            \draw (\x+0.5,1) node[v](1h\x){};
        }
        \foreach \x in {1,3,5}{
            \draw (\x,4.5) node[v](4v\x){};
            \draw (\x,2.5) node[v](2v\x){};
        }
        \foreach \x in {2,4,6}{
            \draw (\x,3.5) node[v](3v\x){};
            \draw (\x,1.5) node[v](1v\x){};
        }
        
        \draw (5h1)--(5h4);
        \draw (4h1)--(4h4);
        \draw (3h1)--(3h4);
        \draw (2h1)--(2h4);
        \draw (1h2)--(1h4);
        
        \draw (2v1)--(1h2);
        \draw (4v1)--(1h4);
        \draw (5h2)--(2v5);
        \draw (5h4)--(4v5);
        
        \draw (5h1)--(4v1);
        \draw (5h3)--(2v1);
        \draw (4v5)--(1v2);
        \draw (2v5)--(1h3);

        \draw (3h4)--(3v6)--(4h4);
        \draw (1h4)--(1v6)--(2h4);
    \end{tikzpicture}\begin{tikzpicture}[overlay]
        \draw[->,thick] (.1,1.15) -- ++(0.5,0) node[midway,above] {\tiny Step 4};
    \end{tikzpicture}
    \begin{tikzpicture}[scale=0.48]
        \draw(0,6.6) node(){};
        \draw(0,0.4) node(){};
        \foreach \x in {1,2,3,4}{
            \draw (\x+0.5,5) node[v](5h\x){};
            \draw (\x+0.5,4) node[v](4h\x){};
            \draw (\x+0.5,3) node[v](3h\x){};
            \draw (\x+0.5,2) node[v](2h\x){};
        }
        \foreach \x in {2,3,4}{
            \draw (\x+0.5,1) node[v](1h\x){};
        }
        \foreach \x in {1,3,5}{
            \draw (\x,4.5) node[v](4v\x){};
            \draw (\x,2.5) node[v](2v\x){};
        }
        \foreach \x in {2,4}{
            \draw (\x,3.5) node[v](3v\x){};
            \draw (\x,1.5) node[v](1v\x){};
        }
        
        \draw (5h1)--(5h4);
        \draw (4h1)--(4h4);
        \draw (3h1)--(3h4);
        \draw (2h1)--(2h4);
        \draw (1h2)--(1h4);
        
        \draw (2v1)--(1h2);
        \draw (4v1)--(1h4);
        \draw (5h2)--(2v5);
        \draw (5h4)--(4v5);
        
        \draw (5h1)--(4v1);
        \draw (5h3)--(2v1);
        \draw (4v5)--(1v2);
        \draw (2v5)--(1h3);
    \end{tikzpicture}
    \caption{The leftmost graph is the line graph of the elementary $5\times7$ wall, and the other graphs are the underlying graphs of trigraphs obtained after conducting Steps~1 to 4, respectively.}
    \label{fig:linecontra}
\end{figure}

\begin{proposition}\label{prop:linewall}
    For all integers $m,n\geq1$, every trigraph whose underlying graph is the line graph of an $m\times n$ wall has twin-width at most~$4$.
\end{proposition}
\begin{proof}
    For convenience, for adjacent vertices $(a,b)$ and $(a',b')$ in an elementary wall, we denote by $((a+a')/2,(b+b')/2)$ the edge between them.

    We first claim that every trigraph $H$ whose underlying graph is the line graph of the elementary $m\times n$ wall has twin-width at most~$4$.
    Since the elementary $m\times n$ wall is an induced subgraph of the elementary $m'\times n'$ wall for some odd integers~$m'$ and~$n'$, it suffices to show the claim for odd integers~$m$ and~$n$.
    Thus, we may assume that $m=2k+1$ and $n=2\ell+1$ for positive integers~$k$ and~$\ell$.

    We proceed by induction on~$k$ to prove the claim.
    For $k=1$, we present a \cont{4} of~$H$ as follows.
    \begin{enumerate}[label=\bf{Step \arabic*.},leftmargin=*]
        \item For each $j\in[n]\setminus\{1\}$, contract $(1.5,j)$ and $(2.5,j)$.
        \item For each $h\in[\ell]$, contract $(1,2h+0.5)$ and $(3,2h+0.5)$.
        After Step 2, the underlying graph becomes a path.
        \item Iteratively contracting pairs of adjacent vertices, until no such pairs remain.
    \end{enumerate}

    Suppose that $k\geq2$.
    We present a partial \cont{4} from~$H$ to a trigraph whose underlying graph is the line graph of the elementary $(m-2)\times n$ wall as follows; see Figure~\ref{fig:linecontra} for an illustration.
    \begin{enumerate}[label=\bf{Step \arabic*.},leftmargin=*]
        \item For each $j\in[n]\setminus\{1\}$, contract $\{(m-1.5,j),(m-0.5,j)\}$ to $(m-1.5,j)$.
        \item For each $h\in[\ell]$, contract $\{(m-2,2h+0.5),(m,2h+0.5)\}$ to $(m-2,2h+0.5)$.
        \item For each $j\in[n]$, contract $\{(m-2.5,j),(m-1.5,j)\}$ to $(m-2.5,j)$.
        \item For each $h\in[\ell]$, contract $\{(m-3,2h-0.5),(m-1,2h-0.5)\}$ to $(m-3,2h-0.5)$.
        \item Contract $\{(m-2.5,1),(m-3,1.5)\}$ to $(m-3,1.5)$.
    \end{enumerate}
    By the inductive hypothesis, the resulting trigraph has twin-width at most~$4$, and therefore $\tww(H)\leq4$.
    This proves the claim by induction.

    Now, let us prove this proposition by induction on the number of vertices.
    Let~$H$ be a trigraph whose underlying graph is the line graph of an $m\times n$ wall~$W$.
    Note that every vertex of~$H$ has degree at most~$4$.
    If~$W$ is the elementary $m\times n$ wall, then we already showed that the twin-width of~$H$ is at most $4$.
    Thus, we may assume that there is an $m\times n$ wall~$W'$ such that~$W$ is obtained from~$W'$ by subdividing some edge~$e$ of~$W'$ to two edges~$e_1$ and~$e_2$.
    Let~$H'$ be the trigraph obtained from~$H$ by contracting~$e_1$ and~$e_2$.
    Then the underlying graph of~$H'$ is
    isomorphic to the line graph of~$W'$ 
    by the isomorphism mapping the new vertex arising from the contraction to~$e$ of~$W'$.
    Therefore, the twin-width of~$H'$ is at most~$4$ by the inductive hypothesis.
    This implies that $H$ has twin-width at most~$4$.
\end{proof}

\section{Maximum number of edges in a graph having a \boldmath{\subd{1}} of twin-width at most \boldmath{$3$}}\label{sec:extremal}

In this section, we prove Theorem~\ref{thm:final}.
We consider only simple graphs in this section.
To prove Theorem~\ref{thm:final}, we use the following definition and proposition.
A $3$-contractible tree decomposition is \emph{tight} if either 
\begin{itemize}
    \item it has only one bag and its unique bag has size $3$ 
    or 
    \item every adhesion set is of size exactly~$3$, at most one of its bags induces a graph isomorphic to~$K_4$ or~$K_6^{=}$, and every other bag is of size exactly~$5$.
\end{itemize}

\begin{proposition}\label{prop:exactextremal}
    Every $\mathcal{F}_3$-minor-free graph~$G$ has at most $\lfloor\frac12(7\abs{V(G)}-15)\rfloor$ edges, unless it has at most~$2$ vertices.
    Such a graph $G$ has exactly $\lfloor\frac12(7\abs{V(G)}-15)\rfloor$ edges if and only if it admits a tight $3$-contractible tree decomposition.
\end{proposition}
\begin{proof}
    We proceed by induction on $\abs{V(G)}$.
    By Proposition~\ref{prop:treedecomp}, 
    we may assume that $G$ admts a $3$-contractible tree decomposition $(T,(B_t)_{t\in V(T)})$.
    We choose $(T,(B_t)_{t\in V(T)})$ such that $\abs{V(T)}$ is minimised.
    
    Note that if $3\leq\abs{V(G)}\leq5$, then $ \abs{E(G)}\leq \binom{\abs{V(G)}}{2}= \lfloor\frac12(7\abs{V(G)}-15)\rfloor$,
    where the equality holds if and only if $G$ is a complete graph.
    Since each of $K_3$, $K_4$, and $K_5$ has a tight $3$-contractible tree decomposition with only one bag, the statement holds when  $3\leq\abs{V(G)}\leq5$.

    Therefore, we may assume that $\abs{V(G)}\geq6$.
    If $\abs{V(T)}=1$, then $G$ is isomorphic to a graph in $\mathcal{K}\setminus\{K_n:n\in[5]\}$, and therefore $\abs{E(G)}\leq\lfloor\frac12(7\abs{V(G)}-15)\rfloor$, where the equality holds if and only if~$G$ is isomorphic to~$K^{=}_6$.
    Thus, the statement holds for the cases when $\abs{V(T)}=1$.
    
    Hence, we may now assume that $\abs{V(T)}\geq2$.
    By the inductive hypothesis, for every leaf~$\ell$ of~$T$ and its neighbour $\ell'$ in~$T$, both $B_\ell\setminus B_{\ell'}$ and $B_{\ell'}\setminus B_\ell$ are nonempty because otherwise we can shrink~$T$ to obtain a smaller $3$-contractible tree decomposition.

    If there is a vertex~$u$ of degree at most $2$, then  
    \[
        \abs{E(G)}\le \abs{E(G-u)}+\deg_G(u) 
        \le 
        \left\lfloor \frac{7\abs{V(G-u)}-15}{2}\right\rfloor+2 < \left\lfloor\frac{7\abs{V(G)}-15}{2} \right\rfloor.
    \]
    Thus, we may assume that every vertex of~$G$ has degree at least~$3$.
    This implies that 
    $\abs{B_v}\ge 4$ for each leaf~$v$ of~$T$.
    We will use the following observation.

    \begin{observation}\label{obs:ratioinK}
        Let $H\in \mathcal{K}$ with $\abs{V(H)}\ge 4$. Then the following hold.
        \begin{itemize}
            \item $\abs{E(H)}\le  \frac{7}{2}(\abs{V(H)}-1) -1$.
            \item $\abs{E(H)}-1\le  \frac{7}{2}(\abs{V(H)}-2)-1$.
            \item $\abs{E(H)}-3\le \frac{7}{2}(\abs{V(H)}-3)-1 $ unless $H\in \{K_5,K_4,K_6^=\}$.
            \item $\abs{E(H)}-3= \frac{7}{2}(\abs{V(H)}-3)-\frac12 $ if $H=K_4$ or $H=K_6^=$.
            \item $\abs{E(H)}-3= \frac{7}{2}(\abs{V(H)}-3)$ if $H=K_5$.
        \end{itemize}
    \end{observation}

    Since $\abs{V(T)}\geq2$, there are distinct leaves~$v$ and~$v'$ of~$T$.
    Let $S_v:=B_v\cap \bigcup_{u\in V(T)\setminus\{v\}} B_u$ and let $S_{v'}:=B_{v'}\cap\bigcup_{u\in V(T)\setminus\{v'\}} B_u$.
    Note that $\abs{V(G)\setminus (B_v\setminus S_v)}\ge \abs{B_{v'}}\ge 4$.
    
    If $\abs{S_v}<3$ 
    or $G[B_v]$ is isomorphic to no graph in $\{K_5,K_4,K_6^=\}$, then by Observation~\ref{obs:ratioinK} and the inductive hypothesis,
    \begin{align*}
        \abs{E(G)}&= \left(\abs{E(G[B_v])}-\binom{\abs{S_v}}{2}\right) + \abs{E(G-(B_v\setminus S_v))}\\
        &\leq\frac{7}{2}\abs{B_v\setminus S_v}-1+\left\lfloor\frac{7\abs{V(G-(B_v\setminus S_v))}-15}{2}\right\rfloor<\left\lfloor\frac{7\abs{V(G)}-15}{2}\right\rfloor.        
    \end{align*}
    Thus, we may assume that $\abs{S_v}=3$
    and $G[B_v]$ is isomorphic to a graph in $\{K_4,K_5,K_6^=\}$.
    By symmetry, we may also assume that $\abs{S_{v'}}=3$ and $G[B_{v'}]$ is isomorphic to a graph in $\{K_4,K_5,K_6^=\}$.
    
    Note that $G-((B_v\setminus S_v)\cup (B_{v'}\setminus S_{v'}) )$ has at least three vertices.
    If neither~$G[B_v]$ nor~$G[B_{v'}]$ is isomorphic to~$K_5$, then by Observation~\ref{obs:ratioinK} 
    and the inductive hypothesis 
    applied to $G-((B_v\setminus S_v)\cup (B_{v'}\setminus S_{v'}) )$, we deduce that 
    \begin{linenomath}
    \begin{align*} 
        \abs{E(G)}&=\left(\abs{E(G[B_v])}-\binom{\abs{S_v}}{2}\right)+\left(\abs{E(G[B_{v'}])}-\binom{\abs{S_{v'}}}{2}\right)+\abs{E(G-((B_v\setminus S_v)\cup (B_{v'}\setminus B_{v'}))}\\
        &=\frac{7\abs{B_v\setminus S_v}-1}{2}+\frac{7\abs{B_{v'} \setminus S_{v'}}-1}{2}+\left\lfloor\frac{7\abs{V(G)}-7\abs{B_v \setminus S_v}-7\abs{B_{v'}\setminus S_{v'}}-15}{2}\right\rfloor \\
        &< \left\lfloor\frac{7\abs{V(G)}-15}{2}\right\rfloor.
    \end{align*}
    \end{linenomath}
    Hence, we may assume that  $G[B_v]$ or $G[B_{v'}]$ is isomorphic to $K_5$.
    
    Without loss of generality, $G[B_v]$ is isomorphic to~$K_5$. 
    By Observation~\ref{obs:ratioinK} and the inductive hypothesis, we deduce that 
    \begin{linenomath}
    \begin{align*}
        \abs{E(G)}&= \left(\abs{E(G[B_v])}-\binom{\abs{S_v}}{2}\right) + \abs{E(G-(B_v\setminus S_v))}\\
        &=\frac{7}{2}\abs{B_v\setminus S_v}+\left\lfloor\frac{7\abs{V(G-(B_v\setminus S_v))}-15}{2}\right\rfloor\leq \left\lfloor\frac{7\abs{V(G)}-15}{2}\right\rfloor.        
    \end{align*}
    \end{linenomath}

    Now let us discuss the equality condition. By our analysis, when $G$ has more than $5$ vertices, to have the equality, it is necessary that 
    $G[B_v]$ is isomorphic to $K_5$, 
    $\abs{S_v}=3$, 
    and $G-(B_v\setminus S_v)$ admits a tight $3$-contractible tree decomposition $(T',(B'_t)_{t\in V(T')})$.
    We claim that in this case, $G$ also has a tight $3$-contractible tree decomposition.
    Suppose that $G[B_v]$ is isomorphic to $K_5$, 
    $\abs{S_v}=3$, 
    and $G-(B_v\setminus S_v)$ admits a tight $3$-contractible tree decomposition $(T',(B'_t)_{t\in V(T')})$.
    Since $S_v$ is a clique, by Lemma~\ref{lem:complete bag}, there is a node~$w$ of~$T'$ such that $S_v\subseteq B_{w}$.
    Let $T''$ be the tree obtained from $T$ by attaching a new leaf~$x$ to~$w$. Let $B''_t:=B'_t$ for all $t\in V(T')$ and $B''_x:=B_v$.
    Then it is easy to see that 
    $(T'',\{B''_t\}_{t\in V(T'')})$ is a tight $3$-contractible tree decomposition of~$G$.
    
    For the converse, observe that if $(T,(B_t)_{t\in V(T)})$ is a tight $3$-contractible tree decomposition, then so is $(T\setminus v,(B_t)_{t\in V(T\setminus v)})$, and so this proves by the inductive hypothesis that if $(T,(B_t)_{t\in V(T)})$ is tight, then the equality holds. 
    This completes the proof.
\end{proof}

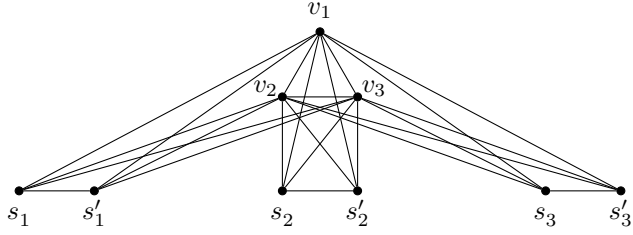
\begin{figure}[t]
    \centering
    \tikzstyle{v}=[circle, draw, solid, fill=black, inner sep=0pt, minimum width=3pt]
    \begin{tikzpicture}[scale=0.5]
        \draw (0,4.232) node[v,label={[xshift=0.0mm,yshift=0.0mm]$v_1$}](v1){};
        \draw (-1,2.5) node[v,label={[xshift=-2.2mm,yshift=-1.7mm]$v_2$}](v2){};
        \draw (1,2.5) node[v,label={[xshift=2.2mm,yshift=-1.7mm]$v_3$}](v3){};
        \draw (v1)--(v2)--(v3)--(v1);

        \draw (-8,0) node[v,label={[xshift=0.0mm,yshift=-6.5mm]$s_1^{\phantom{'}}$}](s1){};
        \draw (-6,0) node[v,label={[xshift=0.0mm,yshift=-6.5mm]$s'_1$}](s1'){};
        \draw (-1,0) node[v,label={[xshift=0.0mm,yshift=-6.5mm]$s_2^{\phantom{'}}$}](s2){};
        \draw (1,0) node[v,label={[xshift=0.0mm,yshift=-6.5mm]$s'_2$}](s2'){};
        \draw (6,0) node[v,label={[xshift=0.0mm,yshift=-6.5mm]$s_3^{\phantom{'}}$}](s3){};
        \draw (8,0) node[v,label={[xshift=0.0mm,yshift=-6.5mm]$s'_3$}](s3'){};

        \foreach \x in {1,2,3}{
            \draw (s\x)--(s\x');
            \draw (s\x)--(v1);
            \draw (s\x)--(v2);
            \draw (s\x)--(v3);
            \draw (s\x')--(v1);
            \draw (s\x')--(v2);
            \draw (s\x')--(v3);
        }
    \end{tikzpicture}
    \caption{The graph $G_9$ in the proof of Theorem~\ref{thm:final}. Note that $G_8=G_9-s_3$.}
    \label{fig:construction}
\end{figure}

We now restate and prove Theorem~\ref{thm:final}.

\edgenumber*

\begin{proof}
    By Proposition~\ref{prop:exactextremal}, 
    if an $n$-vertex graph~$G$ has more than $\lfloor(7n-15)/2\rfloor$ edges and $n\ge 3$, then 
    $G$ has a minor in~$\mathcal{F}_3$.
    By Proposition~\ref{prop:mainforward}, every \subd{1} of~$G$ has twin-width at least~$4$. This proves the first part.

    For the second part, 
    we construct an $n$-vertex graph $G$ with $\lfloor(7n-15)/2\rfloor$ edges such that every \subd{1} of~$G$ has twin-width at most~$3$.
    For odd $n$, we construct an $n$-vertex graph $G_n$ as follows.
    Let $v_1$, $v_2$, and~$v_3$ be the vertices of~$K_3$ 
    and let $s_1,s_2,\ldots,s_{(n-3)/2},s_1',s_2',\ldots,s_{(n-3)/2}'$ be the vertices of~$\frac{n-3}{2}K_2$ such that~$s_i$ is adjacent to~$s_i'$ for each $i\in[(n-3)/2]$.
    Let $G_n:=K_3+\frac{n-3}{2}K_2$; see Figure~\ref{fig:construction} for $G_9$.
    For even $n$, let $G_n:=G_{n+1}-s_{(n-2)/2}$.
    Note that $\abs{E(G_n)}=\lfloor(7n-15)/2\rfloor$ for every integer $n\geq3$.
    We are going to show that for each integer $n\geq3$, every \subd{1} of $G_n$ has twin-width at most $3$.
    Since $G_n$ is an induced subgraph of~$G_{n+1}$, it suffices to show the statement for odd integers $n$.

    Let $G'$ be an \subd{1} of $G_n$ for odd $n\geq3$.
    If $n=3$, then $G'$ is a cycle of length at least~$6$, which has twin-width~$2$, so the statement holds.
    Thus, we may assume that $n\geq5$.
    We are going to construct partial \cont{3}s $\sigma_0,\sigma_1,\ldots,\sigma_{(n-1)/2}$ such that for each $i\in\{0,1,\ldots,(n-1)/2\}$, the subsequence $\sigma_i$ is a partial \cont{3} from a trigraph~$H_i$ to a trigraph~$H_{i+1}$ where $H_0:=G'$ and $H_{(n+1)/2}$ is the $1$-vertex graph.
    Note that we obtain a \cont{3} of~$G'$ by concatenating these partial \cont{3}s.

    For each edge $e$ of $G_n$, we denote by $P_e$ the path in $G'$ replacing~$e$.
    Since $G'$ is an \subd{1} of~$G_n$, for every edge~$e$ of~$G_n$, the length of $P_e$ is at least~$2$.
    For each $i\in[(n-3)/2]$ and each $j\in[3]$, let~$Q_{i,j}$ be the set of internal vertices of~$P_{s_iv_j}$ and~$P_{s'_iv_j}$, that is, $Q_{i,j}:=(V(P_{s_iv_j})\cup V(P_{s'_iv_j}))\setminus\{s_i,s'_i,v_j\}$.
    
    We first construct $\sigma_0$ as follows.
    \begin{enumerate}[label=\bf{Step \arabic*.},leftmargin=*]
        \item For each $1\leq i<j\leq3$, iteratively contract pairs of adjacent internal vertices of $P_{v_iv_j}$ so that only one vertex, say $p_{i,j}$, remains.
        \item Contract $\{p_{1,2},p_{1,3}\}$ to~$p_{1,2}$, and then contract $\{p_{1,2},p_{2,3}\}$ to~$p_{1,2}$.
    \end{enumerate}
    Let $H_1$ be the resulting trigraph, and let $p:=p_{1,2}$.
    It is straightforward to see that $\sigma_0$ is a partial \cont{3}, and the edges incident with $p$ are the only red edges of $H_1$.
    
    For $i$ from $1$ to $(n-3)/2$, we construct $\sigma_i$ as follows.
    \begin{enumerate}[label=\bf{Step \arabic*.},leftmargin=*]
        \item Iteratively contract pairs of adjacent vertices in $Q_{i,1}$ so that only two vertices $q_{i,1}$ and $q'_{i,1}$ remain.
        \item Contract $\{q_{i,1},q'_{i,1}\}$ to $q_{i,1}$.
        \item Iteratively contract pairs of adjacent internal vertices of $P_{s_is'_i}$ so that only one vertex $r_i$ remains.
        \item Contract $\{q_{i,1},r_i\}$ to $q_{i,1}$.
        \item For each $j\in\{2,3\}$, iteratively contract pairs of adjacent vertices in $Q_{i,j}$ so that only two vertices $q_{i,j}$ and $q'_{i,j}$ remain.
        \item Contract the following seven pairs of vertices in order.
    \end{enumerate}
    \[
        \begin{array}{llll}
            1)\ \{q_{i,2},q'_{i,2}\}\text{ to }q_{i,2}, 
            &2)\ \{q_{i,3},q'_{i,3}\}\text{ to }q_{i,3},
            &3)\ \{s_i,s'_i\}\text{ to }s_i,
            &4)\ \{s_i,q_{i,1}\}\text{ to }s_i,\\
            5)\ \{s_i,q_{i,2}\}\text{ to }s_i,
            &6)\ \{s_i,q_{i,3}\}\text{ to }s_i,
            &7)\ \{p,s_i\}\text{ to }p.
        \end{array}
    \]
    Let $H_{i+1}$ be the resulting trigraph.
    It is straightforward to see that $\sigma_i$ is a partial \cont{3}, and the edges incident with $p$ are the only red edges of $H_{i+1}$.
    By combining $\sigma_0,\sigma_1,\ldots,\sigma_{(n-3)/2}$ with an arbitrary contraction sequence $\sigma_{(n-1)/2}$ of~$H_{(n-1)/2}$, we obtain a \cont{3} of $G'$, and this completes the proof.
\end{proof}

\section{Conclusion}\label{sec:conclu}

For an \subd{2}~$H$ of a multigraph~$G$, we characterise when $\tww(H)\leq d$ for each $d\leq3$ in terms of a set of forbidden minors for~$G$.
As an application, we precisely determine the twin-width of all large grids, large walls, and the line graphs of large walls.
Furthermore, we determine the maximum number of edges in an \subd{1} $H$ of a multigraph~$G$ if $H$ has $n$ vertices and $\tww(H)\leq3$.

We present two open problems.
The first is to extend Theorem~\ref{thm:main}\ref{main:atmost3} to \subd{1}s.

\begin{question}
     If $G$ is an $\mathcal{F}_3$-minor-free multigraph, does there exist a $3$-contraction sequence for every \subd{1} of~$G$?
\end{question}

The other problem is to compute the twin-width of narrow grids, narrow walls, and the line graphs of narrow walls.

\begin{question}
    Let $m$ and $n$ be positive integers.
    Determine the twin-width of the following graphs:
    \begin{itemize}
        \item the $6\times n$ grid for $n\geq9$,
        \item an $m\times n$ wall for $m\leq10$ or $n\leq6$, and
        \item the line graph of an $m\times n$ wall for $m\leq10$ or $n\leq5$.
    \end{itemize}
\end{question}

\providecommand{\bysame}{\leavevmode\hbox to3em{\hrulefill}\thinspace}
\providecommand{\MR}{\relax\ifhmode\unskip\space\fi MR }
\providecommand{\MRhref}[2]{%
  \href{http://www.ams.org/mathscinet-getitem?mr=#1}{#2}
}
\providecommand{\href}[2]{#2}

\end{document}